\DeclareSymbolFont{AMSb}{U}{msb}{m}{n}
\documentclass[reqno,centertags,11pt,a4paper,noamsfonts]{amsart}
\pdfoutput=1
\usepackage[svgnames]{xcolor}
\usepackage{graphicx}
\usepackage[english]{babel}
\usepackage[babel=true,expansion=alltext,protrusion=alltext-nott,final]{microtype}
\usepackage[margin={3cm},marginparwidth={2.5cm}]{geometry}
\usepackage[utf8]{inputenc}
\usepackage{libertine}
\usepackage[libertine]{newtxmath}
\useosf
\usepackage[varqu,varl]{inconsolata}
\usepackage[T1]{fontenc}
\usepackage{textcomp}
\usepackage{amsmath}
\usepackage{amsthm}
\usepackage[inline]{enumitem}
\numberwithin{equation}{section}
\usepackage[normalem]{ulem}
\usepackage{tikz}
\usetikzlibrary{arrows}
\usepackage{bm}
\usepackage{soul}
\usepackage{float}
\usepackage{mathtools}
\mathtoolsset{showonlyrefs}
\usepackage{dsfont}
\usepackage{nicematrix}

\newcommand{\RR}{\mathbb{R}}
\newcommand{\RRd}{\mathbb{R}^d}

\newcommand{\NN}{\mathbb{N}}

\newcommand{\HH}{\mathcal{H}}

\newcommand{\LL}{\mathcal{L}}

\newcommand{\dd}{\mathrm{d}}

\newcommand{\calP}{{\mathcal P}}

\newcommand{\unA}{\underline{A}}
\newcommand{\unlambda}{\underline{\lambda}}

\newcommand{\eps}{\varepsilon}
\newcommand{\xdots}{x_{1},\dots,x_{N}}
\newcommand{\ydots}{y_{1},\dots,y_{N}}

\newcommand{\ytildedots}{\tilde{y}_{1},\dots,\tilde{y}_{N}}
\newcommand{\Prob}{\mathcal{P}}

\newcommand{\xp}{\overline{x}_p}
\newcommand{\xpe}{\overline{x}_{p,\eps}}
\newcommand{\xpS}{\overline{x}_{p,S^\mathsf{c}}}
\newcommand{\xvec}{\mathbf{x}}
\newcommand{\yvec}{\mathbf{y}}
\newcommand{\ytildevec}{\tilde{\mathbf{y}}}

\newcommand{\xvecS}{\xvec_{S^\mathsf{c}}}

\newcommand{\xbar}{\overline{x}}
\newcommand{\nup}{\nu_p}

\newcommand{\gammabar}{\overline{\gamma}}
\newcommand{\gammap}{\gamma_p}

\newcommand{\barp}{\mathrm{bar}_p}
\newcommand{\Wbar}{W_p\mathrm{bar}}

    \newcommand{\oG}{\overline{G}}
\newcommand{\zp}{\overline{z}_p}
    \newcommand{\bz}{\overline{z}}
\DeclareMathOperator{\weaklystar}{\rightharpoonup\kern-2.2ex ^* \, \,}

\DeclareMathOperator*{\argmin}{argmin}

\DeclareMathOperator{\spt}{spt}
\DeclareMathOperator{\diam}{diam}
\DeclareMathOperator{\id}{Id}
\DeclareMathOperator{\dist}{dist}

\def\XXint#1#2#3{{\setbox0=\hbox{$#1{#2#3}{\int}$ }
\vcenter{\hbox{$#2#3$ }}\kern-.6\wd0}}

\definecolor{darkred}{rgb}{0.4,0,0} 
\definecolor{darkgreen}{rgb}{0,.4,0}

\definecolor{darkgreen}{rgb}{0,.4,0}

\newcommand{\diag}{\text{diag}}

\newcommand{\suchthat}{\ensuremath{\ : \ }}     
\newcommand{\cJ}{\mathcal{J}}
\newcommand{\bj}{\textbf{j}}
\newcommand{\bk}{\textbf{k}}

\theoremstyle{plain}
\newtheorem{theorem}{Theorem}[section]
\newtheorem{proposition}[theorem]{Proposition}
\newtheorem{corollary}[theorem]{Corollary}
\newtheorem{lemma}[theorem]{Lemma}

\newtheorem*{theorem*}{Theorem}

\theoremstyle{definition}
\newtheorem{definition}[theorem]{Definition}
\newtheorem{remark}[theorem]{Remark}
\newtheorem{example}[theorem]{Example}
\usepackage[pagebackref=false,debug=false]{hyperref}
\hypersetup{
	linktoc=all,
	colorlinks=true,
	linkcolor=Green,citecolor=Orange,urlcolor=DarkBlue,
	pdftitle={p-Wasserstein barycenters},
	pdfauthor={Camilla Brizzi, Gero Friesecke, Tobias Ried}, 
	pdfsubject={Wasserstein barycenter, optimal transport,  multi-marginal optimal transport},
	pdfkeywords={Wasserstein barycenter, optimal transport,  multi-marginal optimal transport}} 
\usepackage{graphicx} 
\begin{document}

\title{On the $q$-integrability of $p$-Wasserstein barycenters}
\author[C.~Brizzi]{Camilla Brizzi}
\address[C.~Brizzi]{Technische Universität München, Department of Mathematics, Boltzmannstraße 3, 85748 Garching, Germany}
\email{camilla.brizzi@tum.de}

\author[L.~Portinale]{Lorenzo Portinale}
\address[L.~Portinale]{Universit\`a degli studi di Milano, Milano, Italy}
\email{lorenzo.portinale@unimi.it}
\date{\today}

\begin{abstract}
We study the $L^q$-regularity of the density of barycenters of $N$ probability measures on $\RR^d$ with respect to the $p$-Wasserstein metric ($1<p<\infty$). 
According to a previous result by the first author and collaborators \cite{BFR24-p}, if one marginal is absolutely continuous, so is the $W_p$-barycenter. The next natural question is whether the $L^q$- regularity on the marginals is also preserved for any $q > 1$, as in the classical case ($p=2$) of Agueh--Carlier \cite{AC11}, or for $W_p$-geodesics ($N=2$). Here we prove that this is the case if one marginal belongs to $L^q$  and the supports of all the marginals satisfy suitable geometric assumptions. 
However, we show that, as soon as $N>2$, it is possible to find examples of $W_p$-barycenters which are \textit{not} $q$-integrable, even if one marginal is compactly supported and bounded, thus highlighting the role played by the geometry of the supports. 
Furthermore, we provide a general estimate of the $L^q$-norm, including a detailed study of the sources of singularities, and a characterization of the $W_p$-barycenters \textit{à la} Agueh--Carlier in terms of the associated Kantorovich potentials. Finally, we  explicitly compute the $W_p$-barycenters of measures obtained as push-forwards of  special affine transformations. In this case, regularity holds without any additional requirement on the supports. 
\end{abstract}

\maketitle
\setcounter{tocdepth}{1}
\tableofcontents
\section{Introduction}
Wasserstein barycenters---defined first in \cite{AC11}---are an important generalization of the classical notion of barycenters of points, as they can be seen as Fréchet means over the Wasserstein space of probability measures. The concept rapidly gained popularity as a valuable tool for
meaningful geometric interpolation between probability measures and computation of representative summaries of input datasets.  
Thus, applications of Wasserstein barycenters span several domains such as data science, statistics, machine learning and image processing (see, among others, \cite{RPDB11, BLGL15, BVFRT22, PZ20}). 
This has sparked growing interest across multiple mathematical communities, making research on this topic---both with theoretical and applied perspectives---highly active. In this article, we study regularity and integrability properties of $p$-Wasserstein barycenters, which are the natural generalizations of $2$-Wasserstein barycenters introduced by Agueh and Carlier \cite{AC11} to the $p$-Wasserstein distance for $1<p<\infty$ and which have been extensively studied in \cite{BFR24-p}. We refer to the work \cite{CCE24} for the case of the so-called Wasserstein medians, corresponding to $p=1$, which, due to the lack of strict convexity presents additional difficulties. 
Further generalizations of the metric include $h$-Wasserstein distance, with $h$ being a nonnegative strictly convex function (see \cite{BFR24-h}) and  variants of OT, such as unbalanced OT (see, \cite{FMS21, BUZ25}) or entropic OT (see, for instance, \cite{CD14,C25}). Another active research direction, which goes beyond the interest of our paper, is the generalization of the  $2$-Wasserstein barycenters in more general spaces, as for instance Riemannian manifolds \cite{KP17}, Alexandrov spaces \cite{J17}, Radon spaces \cite{K18}, and abstract Wiener spaces \cite{HLZ24}.

The $p$-Wasserstein barycenters are weighted averages of probability measures in $\RR^d$ with respect to the $p$-Wasserstein metric. The $p$-Wasserstein distance, $W_p$, between two probability measures $\mu$ and $ \nu $, with finite $p$\textsuperscript{th} moments \footnote{the set of probability measures with finite $p$\textsuperscript{th} moments is $\mathcal{P}_p(\RR^d)\coloneq \left\{\mu\in\mathcal{P}(\RR^d): \int_{\RR^d} |x|^p \,\dd\mu(x) < \infty\right\}$}, is defined as 
 \begin{equation*}
     W_p(\mu,\nu) \coloneq \min_{\eta\in\Pi(\mu,\nu)}\left(\int_{\RR^{2d}}|x-y|^p \, \dd\eta(x,y)\right)^{\frac1p},
 \end{equation*}
where $\Pi(\mu,\nu)$ denotes the set of transport plans from $\mu$ to $\nu$, i.e.\
\[\Pi(\mu, \nu) \coloneq \left\{ \eta \in \Prob(\RR^{2d}): \pi^1_{\#}\eta = \mu, \pi^2_{\#}\eta = \nu\right\}.\]
Thus, given $N\ge 2$ probability measures $\mu_1, \dots, \mu_N \in \mathcal{P}_p(\RR^d)$, which we refer to as \textit{marginals}, and given weights $\lambda_1,\dots,\lambda_N>0$ such that $\sum_{i=1}^N \lambda_i = 1$, the $p$-Wasserstein barycenter is the solution \footnote{existence of a solution here can be inferred by a straightforward adaptation of \cite[Proposition 2.3]{AC11} or by using the multi-marginal formulation, see \eqref{MMpbary} below.}
\begin{align}\label{eq:pW-barycenter}\tag{$W_p$bar}
\Wbar((\mu_i,\lambda_i)_{i=1, \dots, N}) \coloneq	 \argmin_{\nu \in \mathcal{P}_p(\RR^d)} \sum_{i=1}^N \lambda_i W_p^p(\mu_i,\nu).
\end{align}
Notice that, in the case of $N=2$, the $p$-Wasserstein barycenter is nothing but the $p$-Wasserstein geodesic parametrized with a suitable rescaled time (see \eqref{eq:Wp_geodesics}).
In this special case, the $p$-Wasserstein geodesic is not only absolutely continuous under the assumption of the first marginal being absolutely continuous, but it also inherits finer regularity properties, as for instance $L^q$-regularity of the density, as discussed  in \cite[Lemma 4.22]{San15}.
A natural question is therefore whether this result also holds when $N\ge 3$. When $p=2$, it has been shown in \cite{AC11} that the $L^\infty$-regularity of the density of one marginal is preserved by the Wasserstein barycenter. Even if not explicitly written there, the same proof can be easily used to show that any $L^q$-regularity, with $q\ge 1$, is preserved. The case $p\neq 2$ is much more delicate. The main difference here is that the Hessian of the $p$-cost function $|\cdot|^p$ is (twice) the identity matrix when $p=2$, while for any $p\neq 2$ it is not constant and it degenerates at $0$ for $p>2$, whereas it is singular for $1<p<2$. 
One of the key results in \cite{BFR24-p} states that, if all the measures are absolutely continuous with respect to the Lebesgue measure $\LL^d$, then $\nu_p:=\Wbar((\mu_i,\lambda_i)_{i=1, \dots, N})$ is in turn absolutely continuous (cf. \cite[Theorem 1.4]{BFR24-p}). Moreover, for $p\ge 2$, they show that it is sufficient to have only one absolutely continuous marginal, say $\mu_1$. Here, with simple considerations, we extend this result to every $1<p<\infty$  (cf. Theorem \ref{th:absolutecontinuity} in the next section). In \cite{BFR24-p}, the problem of a possibly singular Hessian is overcome by partitioning the space and developing an \textit{ad hoc} strategy to treat the singularities. 

In contrast, when dealing with  $L^q$ regularity for $q>1$, it turns out that the bare assumption $\mu_1 \in L^q(\RR^d)$ (or even $\mu_1 \in L^\infty(\RR^d)$ and compactly supported) does not suffice to guarantee $L^q$ integrability for $\nu_p$.
In Example~\ref{ex:counterexamples}, we show that, even in simple cases, the $L^q$ regularity of the barycenter may fail if $q$ is bigger than a natural threshold which depends on $p$. In the aforementioned examples, it is evident that the lack of integrability may arise from the potential degeneracy/singularity of the Euclidean $p$-barycenter map when $p \neq 2$, due in turn to the previously discussed behavior of the Hessian of $|\cdot|^p$. 

Nonetheless, this potential source of singularity can be overcome with tailored assumptions on the supports of the reference measures. In particular, we assume compactness and some geometric properties which guarantee a positive distance between the support of the marginals $(\mu_i)_{i=1,\dots,N}$ and the one of the barycenter. Under these assumptions, we show in Theorem~\ref{thm:integrability_distant_intro} that, given $f_1,g_p\in L^1$, such that $\mu_1=f_1\LL^d$ and $\nup=g_p\LL^d$, if $f_1 \in L^q$, 
then there exists a constant $C>0$ 
, such that 
\begin{align}\label{eq:result}
    ||g_p||_{L^q}\le \frac{C}{\lambda_1^{d(1-\alpha_p)}}||f_1||_{L^q} 
        \, , 
\end{align}
where $\alpha_p := \frac{p-2}{p-1} \in (-\infty, 1)$.
We refer to the next section for a more detailed discussion. 
This analysis is coherent with the result by M. Goldman and L. Koch in \cite{GK25}, where the partial $C^{1,\alpha}$ regularity for the optimal map of the $p$-Wasserstein distance is proved to hold away from fixed points, i.e. at $x$ such that $T(x) \neq x$.

However, we provide a class of examples where the Euclidean $p$-barycenter is uniformly bounded from below and never degenerates, and therefore the estimate \eqref{eq:result} on the $L^q$-norm of the density of the barycenter holds without any assumption on the supports. 
This is the case of barycenters of marginals defined as some affine transformations of one source measure, discussed in the final section of this paper.
It would certainly be very interesting to provide a more complete overview on which setting ensures integrability, and most of all, provide a counterexample (if any exists) of the nonintegrability of the $p$-Wasserstein barycenter with $p \neq 2$ when \textit{all} the marginals are absolutely continuous, which remains for the time being out of the reach of the current work.

\subsection{Main results and strategy}
\,
A central notion in this work is given by the Euclidean \emph{$p$-barycenter} map, which to any $\xvec=(x_1, \dots, x_N) \in \RR^{Nd}$ associates
		\begin{align}\label{eq:p-barycenter}
			\xbar_p(x_1, \dots, x_N) \coloneq \argmin_{z\in\RR^d} \sum_{i=1}^N \lambda_i |x_i-z|^p
                 \in \RR^d 
                 \, .
		\end{align}
  Existence and uniqueness are a direct consequence of strict convexity and coercivity of the function $w\mapsto|w|^p$.
  The first contribution of this paper is the $L^q$ integrability of $p$-Wasserstein barycenters for compactly supported measures with $\mu_1 \in L^q$ when the measures satisfy suitable geometric conditions, which will be discussed in the following. Throughout this section, we shall always assume that $\mu_1 \ll \LL^d$. 

%

\smallskip 
\noindent 
\textbf{Integrability for distant supports}. \ 
We consider the following set of assumptions:
\begin{itemize}
    \item Compact supports:  there exists $M>0$, such that 
\begin{align}
\label{eq:compact_supports_intro}
\tag{$\mathcal{C}_{\rm pt}$}
        \spt\mu_i\subset B_{\frac M 2}, \quad \text{for every} \ i=1,\dots,N 
    \end{align}
    \item  Distant supports I: it holds 
 \begin{align}
 \label{eq:D>0_intro}
 \tag{$hp_1$}
    \mathcal{D}
        :=
    \inf_{(x_1, \dots, x_N) \in \bigtimes_{i=1}^N \spt\mu_i}
    \big| 
        \xp(x_1, x_2, \dots , x_N) - \xp(x_2, \dots, x_N) 
    \big|  >0
        \, .
 \end{align}
    \item  Distant supports II: it holds 
 \begin{align}
 \label{eq:ass on supports_intro}\tag{$hp_2$}
    m
        :=
    \inf_{i \in \{1, \dots, N\}}
         \inf_{(x_1, \dots, x_N) \in \bigtimes_{i=1}^N \spt\mu_i}
    \big| 
        x_i 
            -
        \xp(x_1, x_2, \dots , x_N)
    \big| >0
        \, .
 \end{align}
\end{itemize}
Note that $\mathcal{D}
        =
\mathcal{D}(\mu_1, \dots, \mu_N)$, same for $m$, but for simplicity we omit the dependence. In case of compact supports, \eqref{eq:D>0_intro} is equivalent (see discussion at the beginning of Subsection~\ref{subsubsect:integrability_p>2}) to 
\begin{align}
\label{eq:distant_supports_intro}
\tag{$hp_1^*$}
   \text{dist} 
   \bigg(
        \spt\mu_1 
            , 
        \xp
        \Big(   
            \bigtimes_{i=2}^N \spt\mu_i
        \Big)
   \bigg) >0
    \, . 
\end{align}
Thus, in case of compactly supported measures, one can immediately see that  \eqref{eq:D>0} is weaker than \eqref{eq:ass on supports} (which is a requirement on every marginal, not only $\mu_1$), that is in turn equivalent to
\begin{align}
\label{eq:distant_supports_intro_i}
\tag{$\mathcal{S}_j$}
   \text{dist} 
   \bigg(
        \spt\mu_j
            , 
        \xp
        \Big(   
            \bigtimes_{i \neq j} \spt\mu_i
        \Big)
   \bigg) > 0
    \, , 
\end{align}
for every $j = 1 , \dots, N$. The latter, stronger condition is needed to ensure $q$-integrability for the $p$-barycenter when $p \in (1,2)$.
    
Our first main result is that, under the aforementioned geometric assumptions, one can show integrability for the $p$-barycenter. 
\begin{theorem}[Integrability with distant support]
\label{thm:integrability_distant_intro}
Fix $q \in (1,+\infty)$, and let $\mu_1, \dots , \mu_N$ satisfy   \eqref{eq:compact_supports_intro} and so that
  $\mu_1 = f_1 \dd \LL^d$ with  $f_1 \in L^q$.  Then there exists  $C= C(p,d,M) \in \RR_+$ so that
\begin{itemize}
    \item if $p \geq 2$ and \eqref{eq:D>0_intro} is satisfied, 
    then $g_p \in L^q$ 
    such that
    \begin{align}
        \| g_p \|_{L^q}
            \leq 
        \frac
            {C}
            {\big( \lambda_1^{ d(1-\alpha_p)}\mathcal{D}^{d(p-2)}\big)^{q'}}
                \| f_1 \|_{L^q} 
                    \, ;
    \end{align}  
       \item if $p \in (1,2)$ and \eqref{eq:ass on supports_intro} is satisfied, then 
    $g_p \in L^q$ and 
    such that
    \begin{align}
        \| g_p \|_{L^q}
            \leq 
        \frac
            {C}
            {\big( \lambda_1^{ d(1-\alpha_p)}m^{d(2-p)}\big)^{q'}}
                \| f_1 \|_{L^q} 
                    \, ,
    \end{align}
 where $q'\in (1,+\infty)$ is the conjugate of $q$. 
\end{itemize}
\end{theorem}

The proof of this theorem can be found in Section~\ref{sec:integrabilityP}, first for $p\geq 2$ (Subsection \ref{subsubsect:integrability_p>2}) 
and then for $p \in (1,2)$ (Subsection \ref{subsubsect:integrability_p<2}).

\smallskip 
\noindent 
\textbf{Counterexamples to integrability}. \ 
Our next contribution is to show the failure of $q$-integrability of the $p$-Wasserstein barycenters with $p\neq 2$, even for $f_1 \in L^\infty$ and compactly supported. This is in sharp contrast to the case $p=2$, where, instead, $f_1 \in L^q$ ensures that also $g_p \in L^q$. 

\begin{theorem}[Nonintegrability examples]
\label{thm:nonint_intro}
    For every $p \neq 2$ and $N>2$, there exist measures $\mu_1,\dots, \mu_N$ which are compactly supported, $\mu_1 = f_1 \LL^d$ with $f_1 \in L^\infty(\RR^d)$, so that $\nu_p = g_p \LL^d$ with $g_p \in L^q$ if and only if $q<q_0$, where $q_0=q_0(p) \in (1, +\infty)$. In particular, $g_p \notin L^q$ for every $q \geq q_0$.
\end{theorem}

We provide below more details about the construction of these nonintegrable examples, but we shall first discuss the general strategy behind Theorem \ref{thm:integrability_distant_intro} and Theorem \ref{thm:nonint_intro}.

\bigskip 
\noindent 
\textit{Strategy}: \ 
We use the same approach as in \cite{BFR24-p} to consider the equivalent multi-marginal optimal transport definition of $p$-Wasserstein barycenter \eqref{eq:pW-barycenter}, where the cost function $c_p: \RR^d \times \dots \times \RR^d \to [0,\infty)$ is given by 
		\begin{align}\label{eq:p-cost}
			c_p(x_1, \dots, x_N) \coloneq \min_{z\in\RR^d} \sum_{i=1}^N \lambda_i |x_i-z|^p = \sum_{i=1}^N \lambda_i |x_i - \xbar_p(x_1, \dots, x_N)|^p
                \,  ,
		\end{align}
    where the map $\xp$ is the one introduced in \eqref{eq:p-barycenter}. 
         Then the multi-marginal optimal transport (MMOT) problem reads
		\begin{align}\label{MMpbary}\tag{C$_{p-\text{MM}}$}
			C_{p\text{-MM}} \coloneq \min_{\gamma \in \Pi(\mu_1, \dots, \mu_N)} \int_{\RR^{Nd}} c_p(x_1, \dots, x_N) \,\dd\gamma(\xdots),
		\end{align}
		where\footnote{We denote by $\pi^i$ the projection on the $i^{\text{th}}$ marginal space, $\pi^i:\RR^{Nd}\to\RR^d$, $\pi^i(\xvec)=x_i$.}  
		\begin{equation*}
			\Pi(\mu_1,\dots,\mu_N):=\{\gamma\in\Prob(\RR^{Nd}) \, : \, \pi^i_\sharp \gamma = \mu_i, \ \text{for all} \ i=1,\dots,N \}
		\end{equation*}
		is the set of admissible transport plans between the marginals $\mu_1, \dots, \mu_N$.	
   Note that $C_{p\text{-MM}}<+\infty$ whenever $\mu_1, \dots,\mu_N \in \mathcal{P}_p(\RR^d)$.
    Moreover, the existence of optimizers follows from standard direct method.  
The $p$-Wasserstein barycenter is then the probability measure on $\RR^d$ \begin{align}\label{Wpbary_pushforward}\tag{$W_p$bar-MM}
   { \xbar_p}_{\#}\gamma_p\, , \quad \text{where $\gamma_p$ is any optimizer of \eqref{MMpbary}.}
\end{align}
Note that  $\nup=W_p\text{bar}(\mu_i,\lambda_i)_{i=1\dots,N}=(\xbar_p)_{\#}\gamma_p$ (see \eqref{eq:equivalence2}). 
In \cite[Theorem 1.2]{BFR24-p} (cf. also Theorem~\ref{th:sparsityplan} below), it has also been shown that, when $\mu_1\ll\LL^d$, the unique optimal plan $\gammap$ for \eqref{MMpbary} is of \textit{Monge type} in the multi-marginal sense, that is, its support is concentrated on a graph of a function defined on the support of $\mu_1$, or, equivalently,
$
    \gamma_p=(\textrm{Id},T_2,\dots,T_N)_\sharp\mu_1,
$
where $\mu_i={T_i}_\sharp\mu_1$. This generalizes the sparsity result by Gangbo and \'Swie\k{}ch in \cite{GS98}, which holds for $p=2$. Moreover, this implies the existence of a map $\barp:\spt\mu_1\to\spt\nup$ (see eq.~\eqref{eq:barp}) such that 
\begin{align}
    \nup={\barp}_\sharp\mu_1.
\end{align}
By means of a standard change of variables formula, for injective and smooth enough $\barp$, the two densities are linked by the equation
\begin{align}
   g_p \circ \barp=\frac{f_1}{|\det \nabla \barp| } .
\end{align}

The injectivity of $\barp$ directly relates to the interplay between the multi-marginal \eqref{Wpbary_pushforward} and the coupled-two marginal \eqref{eq:pW-barycenter} nature of $p$-Wasserstein barycenter. We can indeed show (Corollary \ref{cor:barp_gleich_S1}) that $\barp$ can be alternatively identified as the unique optimal map  $S_1$ for $W_p(\mu_1,\nup)$, which, being both measures absolutely continuous, is invertible. The proof of this fact relies on a characterization of $p$-Wasserstein barycenters in terms of Kantorovich potentials (cf. Proposition \ref{lem:characterization}), which extends the characterization proved in \cite{AC11} for $p=2$ to the general case $1<p<\infty$, and which allows us to show (cf. Proposition \ref{prop:characterisation_onedirection}) that 
  \begin{align}
    \label{eq:charact_optimality_intro}
        \xp \circ \big( R_1, \dots , R_N \big) = {\rm id}
            \qquad 
        \nup - \text{a.e.}
            \, ,
    \end{align}
    where $R_i = T_i \circ S_1^{-1}$.

The challenge is thus to understand the regularity of $\barp$ and $\barp^{-1}$ and to find a lower bound for $|\det \nabla \barp(x)|$.
\begin{itemize}
     \item  We first analyze the case where the first marginal $\mu_1$ is absolutely continuous and all the other marginals are single Diracs. In this case, 
     $\barp^{-1}$ can be computed explicitly (see \eqref{eq:inverse_b}). This allows for a detailed analysis of the potential sources of singularities and their order of magnitude (cf. Proposition~\ref{prop:pbary} and Proposition~\ref{prop:pbary_p<2}).

    \item 
    This analysis provides natural and sufficient assumptions on the supports of the measures (see \eqref{eq:D>0_intro} and \eqref{eq:ass on supports_intro}), which guarantee a uniform lower bound, independent of the choice of the points where the Diracs are concentrated. We then take inspiration from  the proof of regularity of Wasserstein geodesics of ~\cite[Section 4.3]{San15}: we study the case of general marginals by approximating them with a sequence of finite sums of Diracs, which, thanks to standard stability of optimal transport, leads to \eqref{eq:result} (see  Section \ref{sec:main_proof_by_approx}).

     \item
    Thanks to this analysis (see in particular \eqref{eq:sharp_behavior_nonsing} and \eqref{eq:sharp_behavior_nonsing_p<2}), we are able to exhibit the claimed counterexamples in Theorem \ref{thm:nonint_intro}. This shows that also in the very regular case of the first marginal being the uniform measure on a ball, integrability and regularity fail for certain values of $q$ (see Examples \ref{ex:counterexamples}). 
\end{itemize}

Let us delve more into the proof of Theorem \ref{thm:nonint_intro} and 
provide details about the construction of nonintegrable examples. To this purpose, we consider $\mu_2, \dots, \mu_N$ atomic measures with a single atom, i.e.  there exists $\hat x := (\hat x_2,\dots,\hat x_N)$ such that $\mu_i=\delta_{\hat x_i}$, for every $i=2,\dots, N$. We fix $B \subset \RR^d$ an open, bounded set of $\LL^d$-measure $1$, and consider $\mu_1 = \mathds{1}_B \LL^d$. Note that $f_1 \in L^q$ for every $q \in [1,+\infty]$, and it is compactly supported. As we show in Example \ref{ex:counterexamples}, under certain assumptions on $\hat x_2, \dots, \hat x_N$ which depend on whether $p \in (1,2)$ or $p>2$, one can prove the existence of a $q_0 = q_0(p)$ so that Theorem \ref{thm:nonint_intro} holds true. Let us define $\bz_p:= \xp(\hat x_2, \dots, \hat x_N)$ the Euclidean $p$-barycenter of the $N-1$ points with the same weights $\lambda_2, \dots, \lambda_N$.

\begin{enumerate}
    \item For $p > 2$, we pick $\hat x \notin {\rm diag}(\RR^{d(N-1)})$. Assume that $\bz_p \in B$. Then
       \begin{align}
    \label{eq:discrete_non_integrability1_intro}
        g_p \in L^q 
            \Leftrightarrow 
        q < \frac1{\alpha_p} = \frac{p-1}{p-2}
            \, .
    \end{align}
        \item For $p \in (1,2)$, assume that $\hat x_i \neq \bz_p$ and $\hat x_i \in b(B)$, where $b(\cdot) := \xp(\cdot, \hat x_2, \dots, \hat x_N)$. Then
       \begin{align}
    \label{eq:discrete_non_integrability2_intro}
        g_p \in L^q 
            \Leftrightarrow 
        q < 
        \frac{1}{2-p}
            \, .
    \end{align}
\end{enumerate}

\smallskip 
\noindent 
\textbf{General estimate}. \ 
The previous examples show that, in general, one can not expect to obtain $q$-integrability without extra assumptions on the measures. Nevertheless, we can provide a general estimate by analyzing possible sources of singularities. Obviously, this does not generally ensure integrability, as the upper bound may blow up when the geometric assumptions are not in place, as expected.

To this purpose, we need to fix some notation. For any point $\xvec=(\xdots)\in\RR^{Nd}$, we set 
\begin{align}\label{eq:Hi_Lambda_i}
    H_i(\xvec) &:= \lambda_i\nabla_{x_i}^2(|x_i-z|^p)\big|_{z=\xp} \quad \text{and} \\
    \Lambda_i (\xvec)&:= \text{minimum eigenvalue of the matrix}\  H_i(\xvec)\Big(\sum_{k=1}^NH_k(\xvec) \Big)^{-1}H_i(\xvec).
\end{align}
Moreover, we define the following partition of the space:
for every $S\subset\{1,\dots, N\}$, set
  \begin{align}\label{eq:defDS_intro}
     D_S:=\big\{ \xvec=(\xdots)\in\RR^{Nd} 
        \suchthat 
    \xp(\xvec)=x_i \,\, \text{for} \, i\in S, \, \text{and} \, \xp(\xvec)\neq x_i \, \text{for} \, i\not\in S \big\}.
 \end{align}
As $\mu_1$ is assumed to be the reference measure, we also define the family $\mathcal{F}_1$ of subsets of  $\{1,\dots, N\}$ that contain $1$, i.e.
$    
    \mathcal{F}_1:=\{S\subset \{1,\dots, N\} \, : \, 1 \in S \}
        .
$
Finally, for every $S\subset\{1,\dots, N\}$, we set 
\begin{align}\label{eq:D1S_intro}
    D^1_{S}
        &:=
    \pi_1(D_S\cap\spt\gamma_p)
        \, ,
\end{align}
Note that the sets $D_S^1$ form a partition of $\spt\mu_1$.

\begin{proposition}[General $L^q$-norm estimate]
\label{prop:general_Lq_estimate_intro}
  For every $\mu_1, \dots, \mu_N$, let  $\mu_1=f_1 \LL^d$, $\nup=g_p \LL^d$. Assume that $f_1\in L^q(\RRd)$, for some $q>1$. Then the $L^q$-norm of $g_p$ can be controlled by
  \begin{align}
  \label{eq:prop_general_Lq}
      \| g_p \|_{L^q}^q
            \le 
    \int_{\bigcup\limits_{S\in\mathcal{F}_1}D_S^1}|f_1(x_1)|^q\dd x_1 +\frac 12\int_{\bigcup\limits_{S\notin\mathcal{F}_1}D_S^1}\left(\frac{ \max_{i \notin S}\|H^1_{i}(x_1)\|}{\min_{i\notin S}\Lambda^1_i(x_1)}\right)^{d(q-1)}|f_1(x_1)|^q\dd x_1
        \, ,
  \end{align}
  where $\|\cdot\|$ denotes the operator norm.
\end{proposition}

The proof of such estimate is obtained with similar techniques as in \cite{BFR24-p}: partitioning the space and using a slightly improved version of  \cite[Lemma 5.2]{BFR24-h}, which is a local injectivity estimate  on the support of the optimal plan $\gammap$ for \eqref{MMpbary} of the \textit{a priori} highly non-injective map $\xp$. In this way one can bound $|\det \nabla\barp|$ from below with a function that depends on the geometry of the support of the marginals.  
This bound does not require any assumption, but the lower bound may naturally degenerate in certain cases. Under the stronger assumption \eqref{eq:ass on supports_intro}, one can show integrability and thus provide an alternative proof to Theorem \ref{thm:integrability_distant_intro} in this case.

\smallskip 
\noindent 
\textbf{Barycenters of affine transformations}. \ 
In the final part of this work,  we analyze the optimality of affine transformations for the $p$-Wasserstein distance, and we perform some explicit computations of $p$-Wasserstein barycenters between marginals which are obtained by pushing a reference measure with affine transformations with a particular rigid structure.

First of all, for a given measure $\mu \in \Prob_p(\RR^d)$ and   $\nu=T_\sharp \mu$, with $T$ affine, we show that the map $T$ is optimal under strong restrictions on the structure of the associated matrix, in sharp contrast with  the case $p=2$, where any affine  map with positive definite linear part is optimal, due to Brenier's theorem. 
For $p \neq 2$, let $A \in \mathbb{M}(d)$ be symmetric and $b \in \RR^d$ 
satisfying $b \in {\rm range}({\rm Id} - A)$ whenever $A \neq {\rm Id}$. Consider the map $T:\RR^d \to \RR^d$ given by $T(x) = Ax + b$, for $x \in \RR^d$. We show in Theorem~\eqref{thm:optimality_affine_Wp} that 
\begin{align}
    T \text{ is optimal between }\mu \text{ and }\nu = T_\sharp \mu
        \qquad \Longleftrightarrow \qquad 
    \sigma(A)\subset\{1,\zeta\}
        \, , \quad 
    \text{for some  }
        \zeta\geq 0.
\end{align}
Secondly, we apply and generalize this result to describe some explicit examples of $p$-Wasserstein barycenters.
We fix $\mu\in \calP_p(\RR^d)$ a reference probability measure, and consider probability measures of the form $\mu_i = \big( A_i \cdot + v_i \big)_{\#} \mu$, in two cases: either $A_i = {\rm Id}$  or $v_i = v \in \RR^d$ for every $i = 1 , \dots, N$, and in the second case, 
        \begin{itemize}
            \item $A_1$ is invertible. 
            \item $A_i$ are symmetric, positive semidefinite, $d\times d$ matrices. 
            \item The spectrum of the matrices satisfies $\sigma(A_i) \subset \{ 1, \zeta_i \}$ for some $\zeta_i \in [0,+\infty)$. 
            \item The matrices commute, i.e. $[A_i,A_j]=0$ for every $i,j=1, \dots, N$, and the eigenspaces associated with the eigenvalue $1$ of every $A_j \neq {\rm Id}$ coincide. 
        \end{itemize}
Under such assumptions, we show in Proposition \ref{prop:bary_affine}  that the $p$-Wasserstein barycenter is the push-forward of $\mu$ with respect to the barycenter of the affine transformations, i.e. 
\begin{align}
\label{eq:statement_bary_affine_intro}
    \overline\nu_p 
        = 
    \big( 
        \overline{A}\cdot + \overline v
    \big)_{\#}\mu
        \, , \qquad 
    \overline{A}:= \xp(A_1, \dots, A_N) 
        \quad \text{and} \quad 
    \overline{v} := \xp(v_1, \dots, v_N)
        \, , 
\end{align}
where with the barycenters of matrices we mean 
\begin{align}
\label{eq:bary_matrix_intro}
    \xp(A_1, \dots, A_N) 
        :=
    \argmin_{B}
    \sum_{i=1}^N 
        \lambda_i 
        \big| A_i - B \big|^p 
            \, , \qquad 
        |A| := \emph{Tr}(A^T A)^{\frac12}
            = 
        \Big( \sum_{i,j=1}^d  A_{ij}^2 
        \Big)^{\frac12}
            \, ,
\end{align}
for every $A$. 
Note that this is in line with case $p=2$ (cf. Theorem 7.7 in \cite{Fri24}), albeit in that setting it suffices to assume commuting, symmetric, positive definite matrices.

It turns out that combinations of translations and linear transformations are generally not optimal. Thanks to the explicit formula and consequent regularity of the map $\barp$ in this class of examples, we  obtain integrability properties of the $p$-Wasserstein barycenter density without any additional geometric assumptions on the supports of the marginals.

\section{Preliminaries}
\label{sec:preliminaries}

As mentioned in the introduction, it is a known fact that the two definitions of $p$-Wasserstein barycenter, \eqref{eq:pW-barycenter} and \eqref{Wpbary_pushforward}, are equivalent (see for instance \cite{CE10}  or \cite{BFR24-p}). If
\begin{align}\label{C2Mpbary}\tag{C2M-$p$-bar}
		    C_{p\text{-C2M}} \coloneq  \min_{\nu \in \mathcal{P}_p(\RR^d)} \sum_{i=1}^N \lambda_i W_p^p(\mu_i,\nu),
		\end{align}
then
	$
        C_{p\text{-MM}}=C_{p\text{-C2M}}
    $
    and, given $\nu \in\Prob_p(\RRd)$, 
\begin{align}
\label{eq:equivalence2}
 \nu =  \Wbar((\mu_i,\lambda_i)_{i=1, \dots, N})  \iff  \nu={\xp}_\sharp\gammap, \ \text{for }\gammap \, \text{optimal for \eqref{MMpbary}}
    .
\end{align}
This strong relationship between multimarginal and coupled-two marginal minimizers implies a further optimality property of the optimal $\gammap$, when considering its projection $\gamma_i:=(\pi_i,\xp)_\sharp\gammap$  on the product space given by the space of one marginal and the space of the barycenter. More precisely, 
\begin{align}\label{eq:optimalitytwomarginals}
\gamma_i:=(\pi_i,\xp)_\sharp\gammap \in \Pi(\mu_i,\nup)  \ \text{is optimal for}  \ W^p_p(\mu_i,\nup) \text{\footnote{And therefore optimal for the equivalent problems $W_p(\mu_i,\nup)$, $\lambda_iW_p^p(\mu_i,\nup)$}}, \ \text{where} \, \nup:={\xp}_\sharp\gammap.
\end{align}

We now state the Kantorovich duality specifically for our problem \eqref{MMpbary}. The existence of optimizers, known as optimal potentials or Kantorovich potentials, for a large class of multimarginal optimal transport problems, including this case, has been originally proved in \cite{K84}. Here we state the result as in \cite[Theorem 4.1]{BFR24-p}, which gives more information about continuity and almost-everywhere differentiability of the optimizers.
For this purpose, we introduce the space \[\mathcal{Y}_p \coloneq (1+|\cdot|^p)\mathcal{C}_b(\RR^d) \coloneq \left\{ \varphi \in \mathcal{C}(\RR^d): (1+|\cdot|^p)^{-1}\varphi \text{ is bounded}\right\}\] of continuous functions of at most $p$-growth and we define the  $\lambda,p$-conjugate (or $\lambda,p$-transform)  of a function $\varphi:\RR^d \to \RR$ via 
\begin{equation}\label{eq:duality relaxed}
    \varphi^{\lambda, p}(x)\coloneq\inf_{z\in \RRd}\left\{\lambda |x-z|^p - \varphi(z) \right\}
    \quad \text{for } x\in\RRd.
\end{equation}
\begin{theorem}[MMOT Duality]
    \label{th:duality}
There holds
\begin{equation}\label{Dpbary} \tag{D-$p$-bar}
    C_{p-\text{MM}}=\sup_{\varphi_1,\dots,\varphi_N\in\mathcal{A}(c_p)}\sum_{i=1}^{N}\int_{\RRd}\varphi_id\mu_i,
\end{equation}
where 
\[\mathcal{A}(c_p)\coloneq\{ (\varphi_1,\dots,\varphi_N)\in L_{\mu_1}^1(\RRd)\times \cdots \times L_{\mu_N}^1(\RRd) \, : \, \varphi_1\oplus\cdots\oplus\varphi_N\le c_p\}. \]
Moreover, there exists a maximizer $\Phi=(\widetilde\varphi_1,\dots,\widetilde\varphi_N)\in\mathcal{B}(c_p)$, where
\begin{align*}
    \mathcal{B}(c_p)\coloneq\left\{ (\varphi_1,\dots,\varphi_N) \in\mathcal{A}(c_p) : \, \varphi_i= \psi_i^{\lambda_i, p} \text{ with } \psi_i\in \mathcal{Y}_p \text{ for every } i, \, \text{and} \, \sum_{i=1}^N \psi_i = 0   \right\}.
\end{align*}
In particular, $\widetilde\varphi_i$ is $\LL^d$-a.e. differentiable  on the convex hull of the support of $\mu_i$, for every $i=1,\dots,N$. 
\end{theorem}
\begin{proof}
For the proof we refer to Theorem 4.1 in \cite{BFR24-p}. As one can see in that proof (second part of \textbf{Step 2}), each $\widetilde\varphi_i$ is locally bounded in the convex hull of $\spt\mu_i$, for every $i=1,\dots,N$. Being each $\widetilde\varphi_i$ a $\lambda,p$-conjugate to some function, it is locally Lipschitz (see for instance Corollary C.5 in \cite{GMC96}). This is enough to ensure $\LL^d$-a.e. differentiability of $\widetilde\varphi_i$ on the convex hull of $\spt\mu_i$ for every $i=1,\dots,N$.
\end{proof}
As one can expect from \eqref{eq:optimalitytwomarginals}, the potentials constructed in the previous theorem are optimal for the associated $2$-marginal problems.
\begin{corollary}\label{cor: coupled two marg pot}
  For every $i=1,\dots,N$, let $\widetilde\varphi_i$ be as given in  Theorem~\ref{th:duality}, and let $\widetilde\psi_i$ be such that $\sum_{i=1}^N \widetilde \psi_i =0$ and $\widetilde \varphi_i=\widetilde\psi_i^{\lambda_i,p}$. Then the
  functions $(\widetilde\varphi_i, \widetilde\psi_i)$ are optimal potentials for $\lambda_iW^p_p(\mu_i,\nu_p)$, where $\nu_p=\Wbar((\mu_i,\lambda_i)_{i=1,\dots,N})$.
\end{corollary}
\begin{proof}
The proof follows the same line as in \cite{BFR24-p}, we report it here for better clarity. We recall that by construction, for every $i=1,\dots,N$, $\widetilde\varphi_i$  is the $\lambda, p$-conjugate of  $\psi_i$, and thus
    \begin{equation}\label{eq:ineq}
        \widetilde\varphi_i(x_i)+\widetilde \psi_i(z)\le \lambda_i |x_i-z|^p \quad \text{for every} \ (x_i,z) \in \RR^{2d}
    \end{equation}
    and that $\sum_{i=1}^N\widetilde \psi_i(z)=0$, for every $z\in\RRd$.\\
    If $\gamma_p$ is the optimal plan for \eqref{MMpbary},  $\widetilde\varphi_1,\dots,\widetilde\varphi_N$ are the Kantorovich potentials in \eqref{Dpbary}, if and only if the nonnegative function $c_p-\widetilde\varphi_1\oplus\cdots\oplus\widetilde\varphi_N$ is equal to $0$ on $\spt \gammap$. It follows that for every $\xvec \in\spt\gamma_p$,
\begin{align}\label{eq:equality on supp}
\sum_{i=1}^N \widetilde\varphi_i(x_i)+\sum_{i=1}^N\widetilde \psi_i(\xp(\xvec))=\sum_{i=1}^N \widetilde\varphi_i(x_i)=\sum_{i=1}^N\lambda_i|x_i-\xp(\xvec)|^p.
\end{align}
Since the unique optimal plan $\gamma_i\in\Pi(\mu_i,\nu_p)$ for $\lambda_iW_p^p(\mu_i,\nu_p)$ is given by $\gamma_i=(\pi_i,\xp)_\sharp \gamma_p$ (see \eqref{eq:equivalence2} and \eqref{eq:optimalitytwomarginals}), the pair  $(x_i,z)\in\spt\gamma_i$ iff $z=\xp(\xvec)$, where the points $x_1,\dots, x_{i-1},x_{i+1},\dots, x_N$ are such that $\xvec=(x_1,\dots,x_N)\in\spt\gamma_p$. Thus, thanks to \eqref{eq:equality on supp}, inequality \eqref{eq:ineq} is actually an equality on the support of $\gamma_i$.
  \end{proof}
\begin{remark}\label{rem: differentiability_psi}
From the proof of Theorem 4.1 in \cite{BFR24-p} one can see that every $\widetilde \psi_i$ is by construction the $\lambda,p$-conjugate of a function and continuous, for $i=1,\dots, N-1$. Then, as for $\widetilde\varphi_i$, by Corollary C.5 in \cite{GMC96}, $\widetilde \psi_i$ is locally Lipschitz and thus $\LL^d$-a.e. differentiable, for $i=1,\dots,N-1$. Differentiability $\LL^d$-a.e. of $\widetilde \psi_N$ follows by the fact that $\widetilde\psi_N(z)=-\sum_{i=1}^{N-1}\widetilde \psi_i(z)$, for every $z$.
\end{remark}
Theorem \ref{th:absolutecontinuity} below states the absolute continuity of the $p$-Wasserstein barycenter, with $1<p<\infty$, and it is therefore the starting point of our analysis. Notice that it is a slightly improved  version of \cite[Theorem 1.4]{BFR24-p}, where the assumption of all marginals being absolutely continuous has been weakened to only one marginal being absolutely continuous.

Before stating that result, we recall that, given a function $c:\RR^{Nd}\to\RR$, a set $\Gamma\subset \RR^{Nd}$ is said to be $c$-monotone if for every $\xvec^1=(x^1_1,\dots,x^1_N)$, $\xvec^2=(x^2_1,\dots,x^2_N)\in\Gamma$ we have
\begin{equation}\label{eq:c-monotonicity}
c(x^1_1,\dots,x^1_N)+c(x^2_1,\dots,x^2_N)\le c(x^{\sigma_1(1)}_1,\dots,x^{\sigma_N(1)}_N)+c(x^{\sigma_1(2)}_1,\dots,x^{\sigma_N(2)}_N),
\end{equation}
for every $\sigma_i\in S(2)$, with $S(2)$ being the set of permutations of two elements.  
Moreover, the notion of $c$-cyclical monotonicity can also be defined in the multi-marginal setting (see, for instance, Definition 2.2 in \cite{KP14}) and it implies $c$-monotonicity. If $c$ is continuous, then the support of any optimal plan $\gammabar$ for the multi-marginal OT problem associated to $c$
is $c$-cyclically monotone, see \cite[Proposition~2.3]{KP14}. Therefore, if  $c:\RR^{Nd}\to\RR$ is continuous and $\gammabar$ is optimal for
\begin{equation*}
\min_{\gamma\in\Pi(\mu_1,\dots,\mu_N)}\int_{\RR^{Nd}} c(\xdots)\,\dd\gamma,
\end{equation*}
then $\spt\gammabar$ is $c$-monotone.

\begin{theorem}\label{th:absolutecontinuity}
 Let $1<p<\infty$ and $\gamma_p \in \Pi(\mu_1, \dots, \mu_N)$ have $c_p$-monotone support. Then under the condition that $\mu_1 \ll \LL^d$
there holds\
\begin{equation*}
    {\overline{x}_{p}}_{\sharp}\gamma_p\ll\LL^d.
\end{equation*}
It follows that the $p$-Wasserstein barycenter $\nup= \Wbar((\mu_i,\lambda_i)_{i=1, \dots, N})$ of the measures $\mu_1, \dots, \mu_N$ with weights $\lambda_1, \dots, \lambda_N$ is absolutely continuous with respect to Lebesgue measure on $\RR^d$.
\end{theorem}
\begin{proof}
For $p>2$ see \cite[Theorem 1.4]{BFR24-p}, for the extension to general $1<p<\infty$, see Appendix \ref{appendix}.
\end{proof}
As shown in \cite{BFR24-p}, this regularity property has as a direct consequence a strong sparsity result for the optimal plan. More precisely, $\spt\gammap$ is contained in the graph of a map defined on the support of one of the marginals (cf. \cite[Theorem 1.2]{BFR24-p}). Differently from \cite[Theorem 1.2]{BFR24-p}, in Theorem ~\ref{th:sparsityplan} below, the Monge property of the optimal plan holds under the weaker condition $\mu_1\ll \LL^d$, for every $1<p<\infty$. This is due to the fact that the assumption $\mu_1,\dots,\mu_N\ll \LL^d$ of \cite[Theorem 1.4]{BFR24-p} has been improved to $\mu_1\ll \LL^d$ in Theorem ~\ref{th:absolutecontinuity}. 
\begin{theorem}\label{th:sparsityplan}
For any $1<p<\infty$, if  $\mu_1\ll \LL^d$,
there exists a unique optimal plan $\gammap$ for the problem \eqref{MMpbary}, and there exist measurable maps $T_i:\RR^d\to \RR^d$, $i=1, \dots, N$, such that ${T_i}_\sharp \mu_1=\mu_i$ and
\begin{equation}\label{eq:monge-degenerate}
    \gammap=(T_1,T_2,\dots,T_N)_\sharp\mu_1,
\end{equation}
with 
\begin{align}
T_i=R_i \circ S_1,
\end{align}
where $S_1:\RR^d\to \RR^d$ and $R_i:\RR^d\to \RR^d$, ${S_1}_\sharp\mu_1=\nup$ and ${R_i}_\sharp\nup=\mu_i$, are the optimal maps respectively for the $2$-marginals problems $\lambda_1W_p^p(\mu_1,\nu_p)$ and $\lambda_iW_p^p(\nu_p,\mu_i)$, for every $i=1,\dots,N$.
In particular, 
\begin{align}
  S_1=R_1^{-1} \quad \text{and} \quad T_1=\id .
\end{align}
Moreover, there exist functions $\widetilde\varphi_1$ and $\widetilde \psi_i$, differentiable respectively $\mu_1$- and $\nup$- a.e. such that
\begin{align}\label{eq:mapS}
   &S_1 = \id - (p\lambda_1)^{-\frac{1}{p-1}} |\nabla\widetilde\varphi_1|^{-\alpha_p} \nabla\widetilde\varphi_1, \ \text{on} \ \spt\mu_1 \quad \\ &\text{and} \quad R_i = \id - (p\lambda_i)^{-\frac{1}{p-1}} |\nabla\widetilde \psi_i|^{-\alpha_p} \nabla\widetilde \psi_i, \ \text{on} \ \spt\nup 
\end{align} 
for any $i=1,\dots,N$, where $\alpha_p=\frac{p-2}{p-1}$. 
\end{theorem}
The proof is the same as \cite[Theorem 1.2]{BFR24-p}, by applying Theorem \ref{th:absolutecontinuity} instead of \cite[Theorem 1.4]{BFR24-p}.
In both Theorem~\ref{th:absolutecontinuity} and Theorem~\ref{th:sparsityplan}, the choice of $\mu_1$ is arbitrary: if $\mu_i\ll\LL^d$ similar results hold, with $\spt\gammap$ parametrized with respect to the $i$-th marginal. 
\begin{remark}
Thanks to Theorem~\ref{th:sparsityplan}, one defines the map $\barp:\spt\mu_1\to\spt\nup$ as 
    \begin{align}\label{eq:barp}
        \barp(x_1):=\xp\circ({\rm Id},T_2,\dots,T_N)(x_1), \qquad \text{for} \ x_1\in\spt\mu_1.
    \end{align}
We notice thus 
\begin{align}\label{eq:barp_sharp_nu_1}
    \nup={\barp}_\sharp \mu_1.
\end{align}
\end{remark}

\begin{remark}\label{eq:optimality2marginal}
The fact that $S_1$ and $R_i$ are optimal respectively for the $2$-marginals problems $\lambda_1W_p^p(\mu_1,\nup)$ and $\lambda_iW_p^p(\nu_p,\mu_i)$ is a direct consequence of \eqref{eq:optimalitytwomarginals}. We remark that the existence of $R_i$ (see the classical result by Gangbo and McCann \cite{GMC96}) is guaranteed by the fact that $\nup\ll\LL^d$ (Theorem \ref{th:absolutecontinuity}). Notice that by dropping the assumption of absolute continuity on the marginals $\mu_i$, with $i\ge 2$, the existence of a map $S_i:=R_i^{-1}$, ${S_i}_\sharp \mu_i=\nup$, optimal for $\lambda_iW_p^p(\mu_i,\nu_p)$,  is not guaranteed.
Finally, the functions $\widetilde\varphi_1$ and $\widetilde\psi_i$ of Theorem~\ref{th:sparsityplan} are the same as the ones obtained in Theorem \ref{th:duality}. By Corollary \ref{cor: coupled two marg pot} we know that each $\widetilde\varphi_i,\widetilde \psi_i$ are Kantorovich potentials for the $2$-marginals problem $\lambda_iW^p_p(\mu_i,\nu_p)$ and the representation of $S_1$ and $R_i$ with these potentials is consistent with the standard $2$-marginal duality arguments (see for instance \cite[Remark 5.3]{Ambrosio-Brue-Semola:2024} or \cite[Section 3.6]{Fri24}).\footnote{Indeed, by duality theory one knows that the optimal transport map $T$ for a 2-marginal OT problem is obtained, when possible, by inverting w.r.t. $y$ the equality $\nabla_xc(x,y)=\nabla\varphi(x)$, which holds on the support of the optimal plan $\gamma$. In our case the $2$-marginal cost is of the type $c(x,y)=\lambda |x-y|^p$.}
\end{remark}

\subsection{Characterization of $p$-Wasserstein barycenters}
Proposition \ref{lem:characterization} and Proposition \ref{prop:characterisation_onedirection} below are natural extensions of the characterization of the $W_2$-barycenter discussed respectively in \cite[Proposition 3.8, Remark 3.9]{AC11}. In particular, \eqref{eq:charact_optimality} is the equivalent of (3.10) in \cite{AC11}. 
\begin{proposition}\label{lem:characterization}
    Fix $\bar\nu\in\Prob_p(\RR^d)$. Then the following conditions are equivalent:
    \begin{enumerate}
        \item \label{it:lem_nubar} 
        $\bar\nu=\Wbar((\mu_i,\lambda_i)_{i=1, \dots, N})$.
        \item \label{it:lem_exPot} there exist $((\varphi_i,\psi_i))_{i=1,\dots,N}$ such that $(\varphi_i,\psi_i)$ are Kantorovich potentials for $W^p_p(\mu_i,\bar\nu)$ for every $i=1,\dots, N$ and $\sum_{i=1}^N\lambda_i\psi_i(z)=0$, for every $z$.
    \end{enumerate}
\end{proposition}
\begin{proof}
The fact that \eqref{it:lem_nubar} implies \eqref{it:lem_exPot} follows directly by Theorem \ref{th:duality} and Corollary \ref{cor: coupled two marg pot}. Indeed, for each $i$, take $(\frac{\widetilde\varphi_i}{\lambda_i},\frac{\widetilde\psi_i}{\lambda_i})$, where $(\widetilde\varphi_i,\widetilde\psi_i)$ are the ones given by Theorem \ref{th:duality}.\\
   We now prove that \eqref{it:lem_exPot} implies \eqref{it:lem_nubar}.
   The admissibility property of the potentials implies that $\varphi_i(x_i)+\psi_i(z)\le |x_i-z|^p$, for every $x_i$, $z$, and thus
   \begin{align}
   \sum_{i=1}^N \lambda_i \varphi_i(x_i)=\sum_{i=1}^N \lambda_i (\varphi_i(x_i)+\psi_i(z) ) \le \sum_{i=1}^N \lambda_i |x_i-z|^p.  
   \end{align}
   By choosing $z=\xp(x_1,\dots,x_N)$ this yields
     \begin{align}
  \sum_{i=1}^N \lambda_i \varphi_i(x_i) \le \sum_{i=1}^N \lambda_i |x_i-\xp(x_1,\dots,x_N)|^p,  
   \end{align}
   meaning that $\varphi_1,\dots,\varphi_N$ are admissible potentials for the $p$-Wasserstein barycenter multi-marginal cost $c_p(x_1,\dots,x_N)=\sum_{i=1}^N \lambda_i |x_i-\xp(x_1,\dots,x_N)|^p$ (see \eqref{eq:p-cost}). 
   Therefore 
   \begin{align}
\min_{\gamma\in\Pi(\mu_1,\dots,\mu_N)}\int_{\RR^{Nd}} c_p(x_1,\dots,x_N)\dd\gamma(x_1,\dots,x_N)
&\ge  
\sum_{i=1}^N\int\varphi_i(x_i)\dd\mu_i(x_i)
\\
    &=
\sum_{i=1}^N\int \lambda_i \varphi_i(x_i)\dd\mu_i(x_i)+\sum_{i=1}^N\int \lambda_i\psi_i(z) \dd\bar\nu(z)\\ &=\sum_{i=1}^N\lambda_i W_p^p(\mu_i,\bar\nu),
   \end{align}
   where the last equality comes from \eqref{it:lem_exPot}.
   \end{proof}
\begin{proposition}
\label{prop:characterisation_onedirection}
Let $\mu_1\ll \LL^d$. If $\nu_p=\Wbar((\mu_i,\lambda_i)_{i=1, \dots, N})$, then 
    \begin{align}
    \label{eq:charact_optimality}
        \xp \circ \big( R_1, \dots , R_N \big) = {\rm Id}
            \qquad 
        \nup - \text{a.e.}
            \, ,
    \end{align}
    where $R_1,\dots,R_N$ are the maps given by Theorem \ref{th:sparsityplan}.
\end{proposition}
\begin{proof}
By Theorem \ref{th:sparsityplan}, we know that, for every $i = 1, \dots, N$, 
\begin{align}\label{eq: Si}
R_i(z) = z - (p\lambda_i)^{-\frac{1}{p-1}} |\nabla\widetilde\psi_i|^{-\alpha_p} \nabla\widetilde\psi_i(z), \quad \text{for} \ \nup\text{-a.e.}\  z,
\end{align}
where the $\widetilde\psi_i$'s are the ones given by Theorem \ref{th:duality}.
Let $z\in\spt\nup$ such that \eqref{eq: Si} holds. Then $\xp \circ \big( R_1, \dots , R_N \big)(z) = z$ if and only if $z$ is the (unique) solution of 
\begin{align}
\argmin_{y\in\RR^d}\sum_{i=1}^N\lambda_i|R_i(z)-y|^p
    ,
\end{align}
which in turn is equivalent to 
\begin{align}
   \sum_{i=1}^N\lambda_i|R_i(z)-z|^{p-2}(R_i(z)-z)=0.
\end{align}
By \eqref{eq: Si}, the above equation is equivalent to 
\begin{align}\label{eq:equivalence_identity_sumzeropotentials}
   \sum_{i=1}^N\lambda_i\left|(p\lambda_i)^{-\frac{1}{p-1}} |\nabla\widetilde\psi_i|^{-\alpha_p}\nabla\widetilde\psi_i(z)\right|^{p-2} (p\lambda_i)^{-\frac{1}{p-1}} |\nabla\widetilde\psi_i|^{-\alpha_p} \nabla\widetilde\psi_i(z)=0.
\end{align}
All in all, recalling that $\alpha_p=\frac{p-2}{p-1}$, we conclude that,  for $z\in\spt\nup$ such that \eqref{eq: Si} holds,
\begin{align}
    \xp \circ \big( R_1, \dots , R_N \big)(z) = z
        \quad \Longleftrightarrow \quad 
    \sum_{i=1}^N \nabla\widetilde\psi_i(z)=0
         \quad \Longleftrightarrow \quad 
    \nabla\Big( \sum_{i=1}^N \widetilde\psi_i(z) \Big) =0.
\end{align}
Finally, the latter equation is satisfied since $\sum_{i=1}^N\widetilde\psi_i(z)=0$ for every $z$ (see Theorem \ref{th:duality}).
\end{proof}

The following corollary follows directly from Proposition \ref{prop:characterisation_onedirection}. 
\begin{corollary}\label{cor:barp_gleich_S1}
Let $\mu_1\ll\LL^d$. Then
\begin{align}
    \barp=S_1 \quad \mu_1\text{-a.e.}.
\end{align}
In particular, $\barp$ is injective.
\end{corollary}

\subsection{Euclidean Barycenters: properties and regularity}
\label{sec:preliminaries_pbary}
\, 
Recall the definition of the $p$-barycenter in $\RR^d$ as
the map $\xp: \RR^{Nd} \to \RR^d$, defined by 
\begin{align}
    \xp(\xvec)
        = 
    \argmin_{z\in\RR^d} \sum_{i=1}^N \lambda_i |x_i-z|^p
        \,  ,\qquad 
    \xvec = (x_1, \dots, x_N) \in \RR^{Nd}
        \, . 
\end{align}
In this section we discuss some properties of this map which will be used throughout the paper.  Recall that $\alpha_p = \frac{p-2}{p-1} = 1 - \frac1{p-1}$. In particular, $\alpha_p \in (0,1)$ for $p>2$, $\alpha_2 =0$, and $\alpha_p <0$ for $p \in (1,2)$.
\begin{remark}[Properties of $\xp$]\label{rem:first_properties_xp}
Notice that

\renewcommand{\labelenumi}{\arabic*.}
\renewcommand{\labelenumii}{\arabic{enumi}.\arabic{enumii}.}
\renewcommand{\labelenumiii}{\arabic{enumi}.\arabic{enumii}.\arabic{enumiii}}
\begin{enumerate}
\item Let $\xvec=(\xdots)\in \RR^{Nd}$, then $\xp(\xvec)$ is the only solution of its Euler--Lagrange equation
    \begin{align}\label{eq:euler-lagrange}
       \sum_{i=1}^N\lambda_i \nabla_z(|x_i-z|^p)
            = 
        p\sum_{i=1}^N\lambda_i|x_i-z|^{p-2}(x_i-z)=0.
    \end{align}
Therefore, if $x_i \neq x_j$ for some $i,j$, then the barycenter solves 
    \begin{align}\label{eq:euler-lagrange_2}
        \xp(\xvec)
            = 
        \frac1
        {\sum_{i=1}^N\lambda_i|x_i-\xp(\xvec)|^{p-2}}
            \sum_{i=1}^N\lambda_i|x_i-\xp(\xvec)|^{p-2}x_i
                \, ,
    \end{align}
while clearly $\xp(x_1,\dots, x_1) = x_1$, for all $x_1 \in \RR^d$.
In particular, $\xp(\xvec)$ lies in the convex hull of the points $\xdots$.


\item 
In general, the solution of \eqref{eq:euler-lagrange}  cannot be written explicitly as a function of $\xdots$, however 
\begin{enumerate}
    \item if $p=2$, 
    \begin{align}\label{eq:2-barycenter}
        \xp(\xvec)=\sum_{i=1}^N\lambda_i x_i,
    \end{align}
    \item if $N=2$,
    \begin{align}\label{eq:euclidian_geodesics}
        \xp(\xvec)= \frac{\lambda_1^{\frac{1}{p-1}}}{(1-\lambda_1)^{\frac{1}{p-1}}+\lambda_1^{\frac{1}{p-1}}}x_1 +  \frac{(1-\lambda_1)^{\frac{1}{p-1}}}{(1-\lambda_1)^{\frac{1}{p-1}}+\lambda_1^{\frac{1}{p-1}}}x_2.
    \end{align}
\end{enumerate}
\end{enumerate}
\end{remark}

\begin{proposition}[Regularity of $\xp$]\label{prop:diff_xp}
 Recall the definition of $D_S$ in \eqref{eq:defDS_intro}. 
Define $\mathcal{D}_p \subset \RR^{Nd}$ as
 \begin{align}
 \label{eq:def_Dp}
    \mathcal{D}_p := 
    \begin{cases}
        D_{\{1,\dots, N\}}
            &\text{if } p \in (2,+\infty) \, , 
    \\
        \emptyset
            &\text{if } p=2 
                \, ,
    \\
        \RR^{Nd} \setminus D_\emptyset
            &\text{if } p \in (1,2) 
                \, .
    \end{cases}
 \end{align}
The map $\xp: \RR^{Nd} \to \RR^d$ is then continuously differentiable over $\RR^{Nd}\setminus \mathcal{D}_p$, and there we have
\begin{align}
  \label{eq:xp_diff}
    \nabla_{x_i}\xp(\cdot)
        =
    \overline{H}(\cdot)^{-1}H_i(\cdot)
  \end{align}   
for every $i=1, \dots, N$, where we define 
\begin{align}\label{eq:full_form_Hi_Hbar}
    H_i(\xvec) 
        &:= \lambda_i\nabla_{x_i}^2(|x_i-z|^p)\big|_{z=\xp}
    = p \lambda_i |x_i-z|^{p-2}\left((p-2) \frac{x_i-z}{|x_i-z|}\otimes \frac{x_i-z}{|x_i-z|}+ \id \right)\bigg|_{z=\xp(\xvec)}
        \, ,
\\
    \overline{H}(\xvec) 
        &:=\left(\sum_{k=1}^N H_k(\xvec)\right)=
        p\left(\sum_{k=1}^N  \lambda_k |x_k-z|^{p-2}\left((p-2) \frac{x_k-z}{|x_k-z|}\otimes \frac{x_k-z}{|x_k-z|}+ \id \right)\right)\bigg|_{z=\xp(\xvec)}
            \, .
   \end{align} 
\end{proposition}

Note that 
$
   D_{\{1,\dots,N\}}= {\rm diag}(\RR^{dN})
$.
\begin{proof}
    For $\xvec = (\xdots) \in \RR^{Nd}$, recall that $\xp(\xvec)$ is the unique solution of \eqref{eq:euler-lagrange}.
We set  $F(\xvec,z)\coloneq \sum_{k=1}^N\lambda_k|x_k-z|^{p-2}(x_k-z)$. Then
\begin{equation}
    \nabla_z F(\xvec,z) = - \sum_{k=1}^N\lambda_k |x_k-z|^{p-2}\left((p-2) \frac{x_k-z}{|x_k-z|}\otimes \frac{x_k-z}{|x_k-z|}+ \id \right),
\end{equation}
which degenerates, for $p>2$ if and only if $x_k=z$ for every $k$ and for $1<p<2$ if and only if $x_k=z$, for some $k$. If $p=2$, $ \nabla_z F(\xvec,z) = -\id$.
Thus, $\nabla_z F(\xvec,\xp(\xvec))$ exists and is invertible for every $\xvec \in \RR^{Nd} \setminus \mathcal{D}_p$, where $\mathcal{D}_p$ is the closed set defined in \eqref{eq:def_Dp}.

By the Implicit Function Theorem, for every $\xvec \notin \mathcal{D}_p$, there exists an open neighborhood $U_{\xvec}$ of $\xvec$, such that $\xp\in C^1(U_{\xvec})$ and
\begin{align}
   \nabla_{x_i}\xp(\yvec)&=-\nabla_{z} F(\yvec, \xp(\mathbf{y}))^{-1}\nabla_{x_i} F(\yvec, \xp(\yvec))
    \, , 
\end{align}
for every $\yvec\in U_{\xvec}$. Formula \eqref{eq:xp_diff} then follows by a direct computation. 
\end{proof}

In spite of the presence of singularities, one can show that the $p$-barycenter map enjoys global Lipschitz continuity.
\begin{lemma}[Global Lipschitzianity of $\xp$]
    \label{lem:lipschitz}
    The map $\xp: \RR^{Nd} \to \RR^d$ is globally Lipschitz with Lipschitz constant $L_p := \max \big\{ \frac1{p-1}, p-1 \big\} \in \RR$. 
\end{lemma}
\begin{proof}
    As $\xp$ may have singularities, it is convenient to introduce a regularized problem: for $\eps>0$, consider the map $\xpe: \RR^{Nd} \to \RR^d$ given by 
    \begin{align}
        \xpe(\xvec)
            :=
        \argmin_{z\in\RR^d} \sum_{i=1}^N \lambda_i \big( |x_i-z|^2+ \eps \big)^{\frac{p}2}
        \,  ,\qquad 
    \xvec = (x_1, \dots, x_N) \in \RR^{Nd} 
        \, . 
    \end{align}
    From a simple analysis of the optimality conditions, together with strict convexity of the objective function, we see that $\xpe(\xvec)$ can be characterized as the unique solution $z \in \RR^d$ to $F_\eps(\xvec,z)=0$, where $F_\eps:\RR^{Nd} \times \RR^d \to \RR^d$ is the smooth and strictly convex map 
    \begin{align}
        F_\eps(\xvec,z) = \sum_{k=1}^N\lambda_k\big( |x_k-z|^2+ \eps \big)^{\frac{p-2}2}(x_k-z)
            \, , \qquad 
        \xvec \in \RR^{Nd}
            \, , \, 
        z\in \RR^d
            \, .
    \end{align}
    Computing the derivative of $F_\eps$ in the $z$ variable, a direct computation shows that 
    \begin{align}
        \nabla_z F_\eps(\xvec,z)= -
        \sum_{i=1}^N
            H_{i,\eps}(\xvec,z),
    \end{align}
    where we set 
    \begin{align}
        H_{i,\eps}(\xvec,z)
            :=
        \lambda_i 
        \big( |x_i-z|^2 + \eps \big)^{\frac{p-2}{2}} 
        \left((p-2) \frac{(x_i-z)
    \otimes (x_i-z)}{|x_i-z|^2 + \eps}+ {\rm Id} \right)
            \, .
    \end{align}
    As $p >1$, it follows that $H_{i,\eps}$ are positive definite. Therefore, $\nabla_z F_\eps(\xvec,z)$ 
     is always invertible, thus by the implicit function theorem, the regularized problem is smooth and  
    \begin{align}
        \nabla_{x_j} \xpe(\xvec) 
            &= 
        -\nabla_z F_\eps(\xvec, \xpe(\xvec))^{-1} \nabla_{x_j} F_\eps(\xvec, \xpe(\xvec))
    \\
        \label{eq:implicit_reg}
            &=
        \bigg(
            \sum_{i=1}^N
            H_{i,\eps}(\xvec,\xpe(\xvec)) 
        \bigg)^{-1}
            H_{j,\eps}(x,\xpe(\xvec)) 
            \, , 
    \end{align}
    for every $\xvec\in\RR^{Nd}$.
    We claim that the maps $x_{p,\eps}$ are globally Lipschitz with Lipschitz constant independent of $\eps>0$. Once the claim is proven, the conclusion follows from the fact that the functions 
    \begin{align}
        \RR^d \ni z 
            \mapsto 
        \sum_{k=1}^N 
            \lambda_k \big( |x_k-z|^2+ \eps \big)^{\frac{p}2}
    \end{align}  
    are equicoercive in $\eps$, 
    and converge locally uniformly as $\eps \to 0$ to 
    \begin{align}
        \RR^d \ni z
            \mapsto 
        \sum_{k=1}^N
        \lambda_k
            |x_k-z|^p
            \, .
    \end{align} 
    As the unique minimizer of the latter function is precisely given by $\xp(\xvec)$, it follows that
    \begin{align}
        \xpe(\xvec) \to \xp(\xvec)
            \quad \text{as } \eps  \to 0
                \, , \qquad 
            \forall \xvec \in \RR^{Nd}
                \, ,
    \end{align}
    and therefore uniform-in-$\eps$ Lipschitz estimates for $\xpe$ are inherited by $\xp$ as well.

    Let us show the claimed, uniform Lipschitz continuity of $\xpe$.
    For simplicity, we set $\alpha_{i,\eps}(z):= \lambda_i 
        \big( |x_i-z|^2 + \eps \big)^{\frac{p-2}{2}}$, for $z \in \RR^d$ and $i \in \{1, \dots , N\}$.

    \smallskip 
    \noindent 
    \underline{Case $1<p<2$}. \ 
    In this case, $H_{i,\eps}(\xvec,z)$ is symmetric and satisfies
    \begin{align}
    \label{eq:bound_reg1}
        \alpha_{j,\eps}(z) (p-1) {\rm Id} 
            \leq 
        H_{j,\eps}(\xvec,z) 
            \leq 
        \alpha_{j,\eps}(z)  {\rm Id} 
            \, ,
    \end{align}
    in the L\"owner sense\footnote{Two symmetric matrices satisfy $A \leq B$ iff $B-A$ is positive semidefinite (i.e. $B-A \geq 0$).}.
    Therefore, we also get
    \begin{align}
    \label{eq:bound_reg2}
        \sum_{i=1}^N 
            H_{i,\eps}(\xvec,z) 
            \geq 
        \Big(
            \sum_{i=1}^N
            \alpha_{i,\eps}(z) 
        \Big) 
        (p-1) {\rm Id}  
            \, .
    \end{align}    
    For a matrix $A$, we denote by $\| A \|$ the operator norm (i.e. the maximum modulus of the eigenvalues for symmetric matrices) and we have that $\| AB \|\leq \| A \| \| B \|$, for every $A,B$. 
    Therefore, by \eqref{eq:implicit_reg} and \eqref{eq:bound_reg1} we conclude that 
    \begin{align}
        \big\| 
            \nabla_{x_j} \xpe(\xvec) 
        \big\|
            &\leq 
        \Big\| 
            \bigg(
            \sum_{i=1}^N
            H_{i,\eps}(\xvec,\xpe(\xvec)) 
        \bigg)^{-1} 
        \Big\|
        \big\| 
            H_{j,\eps}(\xvec,\xpe(\xvec)) 
        \big\|
\\
        &\leq
        \frac1{p-1}
        \Big(
            \sum_{i=1}^N
            \alpha_{i,\eps}(z) 
        \Big) ^{-1}
        \alpha_{j,\eps}(z) 
            \leq 
        \frac{1}{p-1}
            \, , 
    \end{align}
    for every $\xvec \in \RR^{Nd}$, $j \in \{1, \dots, N\}$, and $\eps>0$. This shows that the maps $\xpe$ are globally Lipschitz with constant $\frac1{p-1}$, uniformly in $\eps>0$. 

    \smallskip 
    \noindent 
    \underline{Case $p \geq 2$}. \ 
    In this case, $H_{i,\eps}(\xvec,z)$ is symmetric and satisfies
    \begin{align}
    \label{eq:bound_reg3}
        \alpha_{j,\eps}(z) {\rm Id} 
            \leq 
        H_{j,\eps}(\xvec,z) 
            \leq 
        \alpha_{j,\eps}(z) (p-1) {\rm Id} 
            \, .
    \end{align}
    The rest of the proof goes verbatim, with final Lipschitz constant equal to $p-1$.
\end{proof}

\subsubsection{One-variable map}
In this part, we fix $N-1$ points $\hat x_2,\dots,\hat x_N\in \RRd$. Call $\hat x \in \RR^{(N-1)d}$ such that $\hat x=(\hat x_2,\dots,\hat x_N)$, and define the function $b_{\hat x}:\RR^d \to \RR^d$ simply given by $b_{\hat x}(x_1) := \xp(x_1, \hat x)$. The next proposition sums up the main properties of the map $b_{\hat x}$. With the notation $A(x) \lesssim B(x)$ we mean that there exists a constant $C=C(d)\in\RR_+$ (in particular independent of $x_1, \hat x$) such that $  A(\cdot)  \leq CB(\cdot) $ as quadratic forms, everywhere on $\RR^d$. Similarly with $A(x) \gtrsim B(x)$. When both are true, we simply write $A(x) \simeq B(x)$. We denote by $\xp(\hat x)$ the barycenter of $N-1$ points with their corresponding weights, i.e., 
\begin{align}
    \xp(\hat x)
        = 
    \argmin_{z\in\RR^d} \sum_{i=2}^N \lambda_i |\hat x_i-z|^p
        \,  ,\qquad 
    \hat x = (\hat x_2, \dots, \hat x_N) \in \RR^{(N-1)d} 
        \, . 
\end{align}
Define
\begin{align}
       \oG(\hat x,z):=\sum_{i=2}^{N}\lambda_i|\hat x_i -z|^{p-2}(\hat x_i-z)
        \, . 
\end{align}
For simplicity, throughout the section, we omit the dependence on $\hat x$ and simply write $b=b_{\hat x}$ and $\oG(z):= \oG(\hat x, z)$. By optimality, $b(x_1)$ is the unique solution $z\in\RR^d$ of 
\begin{align}
\label{eq:optimality_b}
    \lambda_1 |x_1-z|^{p-2}(x_1-z)
        + \oG(z)=0,
\end{align}
which is nothing but \eqref{eq:euler-lagrange} rewritten for this setting.
This in particular implies that 
\begin{align}
\label{eq:moduli_b-id}
    \lambda_1 |x_1 - b(x_1)|^{p-1} = |\oG(b(x_1))|
        \, , \qquad 
    \forall x_1 \in \RR^d 
        \, .
\end{align}
By substituting back into \eqref{eq:optimality_b} with $z = b(x_1)$, we find 
\begin{align}
\label{eq:implicit_b}
    x_1 - b(x_1) 
        = 
    -\frac1{\lambda_1^{1-\alpha_p}}
        \frac
        {\oG(b(x_1)) }
        {|\oG(b(x_1))|^{\alpha_p}}
            \, , \qquad 
        \forall x_1 \in \RR^d 
            \, .
\end{align}  
This shows that, for any given $\hat x \in \RR^{(N-1)d}$, the map $b=b_{\hat x}$ is a bijection of $\RR^d$, and it has a unique fixed point, given by $\xp(\hat x)$. Furthermore,  the inverse function has the explicit form 
   \begin{align}
   \label{eq:inverse_b}
        b^{-1}(z)
            =
        z 
        - 
    \frac1{\lambda_1^{ 1 - \alpha_p}}
        \frac
            {\oG(z)}
            {|\oG(z)|^{\alpha_p}}
        \, , \qquad 
    z \in \RR^d
        \, .
   \end{align}
\begin{remark}\label{rem:b_special_cases}
If $p=2$, $N=2$ or $\hat x\in \mathrm{diag(\RR^{(N-1)d})}$, then  $b,b^{-1}\in C^1(\RR^d)$.  Indeed, if $p=2$ or $N=2$, from Remark ~\ref{rem:first_properties_xp}, one can easily infer that 
\begin{align}
    &p=2, \quad \nabla b(x_1)=\lambda_1\id, \quad \text{for every} \, x_1\in \RRd, \\
    & N=2, \quad  \nabla b(x_1)=\frac{\lambda_1^{1-\alpha_p}}{\lambda_1^{1-\alpha_p}+(1-\lambda_1)^{1-\alpha_p}}\id, \quad \text{for every} \, x_1\in \RRd.
\end{align}
If $\hat x \in \mathrm{diag}(\RR^{d(N-1)})$, we get $b(x_1)=\xp(x_1,\hat x_2)$, where the weight corresponding to $\hat x_2$ is $\sum_{i=2}^N\lambda_i$, thus falling back to the case $N=2$. 
\end{remark}
 
In what follows, $\sigma(A) \subset \mathbb{C}$ denotes the spectrum (i.e. the set of eigenvalues) of a matrix $A$.
   
\begin{proposition}[Estimates on  $\nabla b^{-1}$, $p\ge 2$]
\label{prop:pbary}
 Let us consider $p\geq2$ and fix $\hat x \in \RR^{(N-1)d}$, set $b=b_{\hat x}$, then $b^{-1}\in C^1(\RRd\setminus\{\xp(\hat x)\})$. Moreover for every $ z \in \RR^d$ with $z \neq \xp(\hat x)$,  we have that $\sigma(\nabla b^{-1}(z) ) \subset \RR$ and
\begin{align}
    \label{eq:nonsharp_lb}
        1
            \leq 
        \zeta 
            \leq 
        \bigg(
        1 
            + 
        C_p
         \Big(
    \frac   
        {1-\lambda_1}{\lambda_1}
        \Big)^{1-\alpha_p}
    \Big(
    \frac
      {
            M(z)
        }
        {
           | z - \xbar_p(\hat x) |
        }
    \Big)^{p-2}
        \bigg)
        \, , \quad 
    \forall \zeta \in \sigma(\nabla b^{-1}(z) )
        \, ,
   \end{align}
   where $M(z) := \max_{i \geq 2} |\hat x_i - z|$ and $C_p \in (0,+\infty)$ is a constant only depending on p.

  If  $N>2$ and $p>2$, then for every $\hat x \notin \mathrm{diag}(\RR^{(N-1)d})$
    and every compact set $K \subset \RR^d$, there exists a constant $c=c(\hat x,K,p) \in (1,+\infty)$ such that 
   \begin{align}
   \label{eq:sharp_behavior_nonsing}
    \frac1{c} 
        < 
    \Big( 
        \lambda_1^{1-\alpha_p}|z-\xbar_p(\hat x)|^{\alpha_p}
    \Big)    
        \zeta 
    < 
    c 
        \, , \quad 
    \forall \zeta \in \sigma(\nabla b^{-1}(z) )
        \, ,
   \end{align}
for every $z \in K$.
\end{proposition}

\begin{proof}
For simplicity, we set $\bz_p := \xp(\hat x)$.
From \eqref{eq:moduli_b-id}, we also see that $b(x_1) = x_1$ is solved uniquely by $x_1 = \bz_p$, being the unique solution $z \in \RR^d$ to $\oG(z) =0$, as it is given by 
\begin{align}
    \bz_p 
        \in 
    \argmin_{z \in \RR^d}
    \left\{
        \sum_{i=2}^N 
            \lambda_i |z- \hat x_i|^p 
    \right\}
        \, .
\end{align}

In particular, $\bz_p$ is the unique fixed point of the inverse $b^{-1}$ as well. Note that as $\alpha_p \in [0,1)$ as $p \geq 2$, it readily follows that $b^{-1}$ is also  continuous in $z = \bz_p$.

If $p=2$, $N=2$ or $\hat x\in\mathrm{diag}(\RR^{(N-1)d})$, we know from Remark \ref{rem:b_special_cases}, that $b^{-1}\in C^1({\RRd})$.
    For $N>2$, $p>2$ and $\hat x\not\in\mathrm{diag}(\RR^{(N-1)d})$, $b^{-1} \in C^1(\RR^d \setminus \{ \bz_p \} )$. Indeed, although everywhere continuous and $\oG \in C^1(\RRd)$, as $\alpha_p \in (0,1)$, from \eqref{eq:inverse_b} one can see that $b^{-1}$ is not differentiable at $\bz_p$.
    Using that for $\alpha \in \RR_+$  and $y \in \RR^d$ 
 \begin{align}
 \label{eq:gradient_cmp}
    \nabla
    \Big(
    \frac
        {\cdot}
        {|\cdot|^\alpha}
    \Big)(y)
        =
    \frac1{|y|^\alpha}  
        \text{Id}
            -
    \frac{\alpha}{|y|^{\alpha+1}}
        y \otimes \nabla |y|
        =
    \frac1{|y|^\alpha}
    \Big[
        \text{Id}
            -
        \alpha \frac
            {y \otimes y}
            {|y|^2}
   \Big]   
        \, , 
 \end{align} 
 for $z \neq \bz_p$ we can differentiate \eqref{eq:inverse_b} and obtain

\begin{align}
\label{eq:formula_grad_b_inverse}
    \nabla b^{-1}(z)
        =
    \text{Id}
        - 
    \frac1{\lambda_1^{1-\alpha_p}|\oG(z)|^{\alpha_p}}
    \mathfrak{M}(z)
        \nabla \oG(z)
            \, , \qquad 
    z \in \RR^d\setminus\{\bz_p\} 
        \, .
\end{align}
where we set 
\begin{align}
    \mathfrak{M}(z) :=
        \text{Id}
            - 
        \alpha_p 
        \frac
            {\oG(z) \otimes \oG(z)}
            {|\oG(z)|^2}
                \, , \quad 
            z\in \RR^d\setminus\{\zp\} 
                \, .
\end{align}
For every $z \in \RR^d\setminus\{\zp \}$, $|\oG(z)|^{-2} \oG(z) \otimes \oG(z)$ is a rank-one matrix with eigenvalues $\{0,1\}$ and thus $\mathfrak{M}(z)$  is symmetric and positive semidefinite and satisfies  ($\alpha_p \in [0,1)$ for $p \geq 2$) 
\begin{align}
\label{eq:rank1_bound}
    0 
        <
    (1-\alpha_p) \text{Id}
        \leq 
    \mathfrak{M}(z)
        \leq 
    \text{Id}\, .
\end{align} 
Moreover from the very definition of $\oG$, we have that 
\begin{align}
    - \nabla \oG(z) 
        &=
    (p-2) 
    \sum_{i=2}^N 
        \lambda_i |z- \hat x_i|^{p-4}
            (z- \hat x_i) \otimes (z - \hat x_i)
                +
    \Big(
        \sum_{i=2}^N 
            \lambda_i |z - \hat x_i|^{p-2}
    \Big) 
        \text{Id}
\\
\label{eq:nabla_Gp}
        &=
    \sum_{i=2}^N 
        \lambda_i |z- \hat x_i|^{p-2}
        \bigg[
            (p-2)
            \frac
                { (z- \hat x_i) \otimes (z - \hat x_i) }
                { |z- \hat x_i|^2 }   
                    +
            \text{Id}
        \bigg]
\end{align}
for $z \in \RR^d$. 
Therefore,  $-\nabla G(z)$ is symmetric and positive definite and 
\begin{align}
\label{eq:bound_-gradG}
    A(z) 
    {\rm Id}
        \leq 
    -\nabla \oG(z)
        \leq 
    (p-1)A(z) 
    {\rm Id}
        \, ,
\end{align}
for every $z \in \RR^d$, where we also define
\begin{align}
    A(z) = A(\hat x,z) := \sum_{i=2}^N \lambda_i |\hat x_i-z|^{p-2}
        \, .
\end{align} 
All these matrix inequalities are intended in the L\"owner sense. As the goal is to provide lower and upper bounds for the spectrum of $\nabla b^{-1}(z)$, thanks to \eqref{eq:formula_grad_b_inverse} it is enough to study the spectral properties of the matrix $-\mathfrak{M}(z) \nabla G(z)$. Note that generally such matrix fails to be symmetric. Nonetheless, it is similar to the matrix $\mathfrak{T}(z):=-\mathfrak{M}(z)^{\frac12} \nabla G(z) \mathfrak{M}(z)^{\frac12}$, as $\mathfrak{M}(z)$ is invertible by \eqref{eq:rank1_bound} and
\begin{align}
    -\mathfrak{M}(z) \nabla G(z)
        = 
    \mathfrak{M}(z)^{\frac12} \mathfrak{T}(z)  \mathfrak{M}(z)^{-\frac12}  
        \, .
\end{align}
As $\mathfrak{T}(z)$ is positive semidefinite and symmetric, this shows that $-\mathfrak{M}(z) \nabla G(z)$ (hence $\nabla b^{-1}(z)$) has real and nonnegative eigenvalues. Moreover,  by conjugating \eqref{eq:bound_-gradG} by $\mathfrak{M}(z)^{\frac12}$ we conclude that\footnote{Note that L\"owner order is preserved by conjugation: if $A,B$ are symmetric matrices so that $A \leq B$, then $CAC \leq CBC$, for every symmetric matrix $C$, as $B-A \geq 0$ clearly ensures that $C(B-A)C \geq 0$.}
\begin{align}
    A(z) \mathfrak{M}(z) \leq \mathfrak{T}(z) \leq (p-1) A(z) \mathfrak{M}(z) 
        \, , 
\end{align}
which, together with \eqref{eq:rank1_bound} and the nonnegativity of $A(z)$, provides 
\begin{align}
    (1-\alpha_p) A(z) {\rm Id} 
        \leq 
    \mathfrak{T}(z) 
        \leq 
    (p-1) A(z) {\rm Id} 
        \, , 
\end{align}
for every $z \in \RR^d$. By similarity, we conclude that every eigenvalue of $-\mathfrak{M}(z) \nabla G(z)$ belongs to the interval $[(1-\alpha_p)A(z),(p-1)A(z)]$, hence by \eqref{eq:formula_grad_b_inverse} 
\begin{align}
\label{eq:double_bound_nabla_b}
    1
        +
    (1-\alpha_p) 
    \frac
        {
            A(z)
        }
        {
            \lambda_1^{1-\alpha_p}|\oG(z)|^{\alpha_p}
        }
        \leq
    \zeta
        \leq
    1
        +
    (p-1) 
    \frac
        {
            A(z)
        }
        {
            \lambda_1^{1-\alpha_p}|\oG(z)|^{\alpha_p}
        }
            \, , \quad 
        \forall \zeta \in \sigma(\nabla b^{-1}(z))
            \, , \quad 
\end{align}
for every $z \in \RR^d$.  

 The upper bound \eqref{eq:nonsharp_lb} follows directly from \eqref{eq:double_bound_nabla_b} as soon as we prove
 \begin{align}
     \label{eq:final_lb}
     |\oG(z)| \ge  (1-\lambda_1) \frac{2^{2-p}}{p-1} |z- \bz_p|^{p-1}
        \, .
 \end{align} 
 Indeed, the bound \eqref{eq:final_lb} would imply 
 \begin{align}
    (p-1)
     \frac
        {
            A(z)
        }
        {
            \lambda_1^{1-\alpha_p}|\oG(z)|^{\alpha_p}
        }
        \leq 
    C_p\frac
        {
            (1-\lambda_1)^{-\alpha_p}\sum_{i=2}^N \lambda_i M(z)
        }
        {
            \lambda_1^{1-\alpha_p}|z-\bz_p|^{(p-1)\alpha_p}
        }
        =
    C_p\frac
        {
            (1-\lambda_1)^{1-\alpha_p} M(z)
        }
        {
            \lambda_1^{1-\alpha_p}|z-\bz_p|^{(p-2)}
        }
            \, ,
 \end{align}
for some $C_P \in (0,+\infty)$, as claimed in \eqref{eq:nonsharp_lb}. 

 The lower bound on the norm of $\oG(z)$ instead follows by the strong convexity of the function $u \mapsto g(u):=\frac1{p}|u|^p$, which can be recast as 
 \begin{align}
 \label{eq:strong_conv_p>2}
     \langle 
        \nabla g(u) - \nabla g(v) 
            , 
        u- v
     \rangle
        \geq 
    \frac{2^{2-p}}{p-1} | u - v |^p
        \, .
 \end{align}
 Using the latter inequality and the fact that $\oG(z) = \sum_{i=2}^N \lambda_i \nabla g(x_i-z)$, we conclude that
 \begin{align}
     \langle 
        \oG(z) - \oG(\bz_p)
            , 
        \bz_p - z 
     \rangle 
        &= 
    \sum_{i=2}^N 
        \lambda_i 
     \langle 
        \nabla g(x_i - \bz_p) - \nabla g(x_i - z)
            , 
        z - \bz_p 
     \rangle 
    \\ \label{eq:convex_lb}
        &\geq
    (1-\lambda_1) \frac{2^{2-p}}{p-1} |z-\bz_p|^p 
        \, .
 \end{align}
Using that $\oG(\bz_p) = 0$ and by means of a simple Cauchy-Schwarz inequality provides us with the sought lower bound.

Let us now assume that $p>2$,  $N>2$, and that $\hat x\not \in \mathrm{diag}(\RR^{(N-1)d})$. 
We want to show the validity of \eqref{eq:sharp_behavior_nonsing}, for every compact set $K\subset 
\RR^d$. 
Without loss of generality, we assume that $K \cup \{\bz_p\}$ is convex, if not one can simply work with the conv$(K \cup \{ \zp \})$, where conv denotes the convex hull.
For off-diagonal $\hat x$, we have that 
\begin{align}
\label{eq:infAsupA}
    0 
        < 
    \inf_{z \in K}
        A(z) 
        \leq 
    \sup_{z \in K}
        A(z) 
        < 
    +\infty 
        \,
\end{align}
and therefore from \eqref{eq:bound_-gradG} we conclude that 
\begin{align}
\label{eq:lb_grad}
    \frac1c{\rm Id}
        \leq 
    -\nabla \oG(z)
        \leq 
    c {\rm Id}
        \qquad 
    \forall z \in K
        \, .
\end{align}
for some $c=c(\hat x,K) \in (0,+\infty)$.
In particular, $-\nabla \oG(z)   \neq 0$ everywhere, and therefore for $p \neq 2$ (as $\alpha_p \neq 0$), $\nabla b^{-1}$ has a singularity exactly when $\oG(z)=0$, i.e. in $z=\bz_p$. 
In a more quantitative way:  
we have that 
     \begin{align}
     \label{eq:ancora}
        \oG(z) 
            =
        \oG(z) - \oG(\bz_p)
            =
        \int_0^1 
            \nabla \oG(\bz_p + t (z-\bz_p) ) (z-\bz_p)\dd t
                \, .
     \end{align}
    Therefore from \eqref{eq:lb_grad} and the convexity of $K \cup \{\bz_p\}$, we directly deduce that 
    \begin{align}
        | \oG(z) |
            \leq 
        |z-\bz_p|
            \int_0^1
                \| \nabla \oG(\bz_p + t (z-\bz_p) )\| \dd t
            \leq 
            c |z-\bz_p|
                \, .
    \end{align}
    Instead, from \eqref{eq:ancora}, by taking the scalar product with $z-\bz_p$, and applying  \eqref{eq:lb_grad} once again, we conclude that 
    \begin{align}
        | \oG(z) | |z-\bz_p| 
            \geq 
        \langle 
            - \oG(z) 
                ,
           z-\bz_p
        \rangle 
            \geq 
        \frac1c | z-\bz_p|^2 
            \, .
    \end{align}
    All in all, we have shown that 
\begin{align} 
    \frac1{\tilde c} |(z - \bz_p)|
        \leq 
    | \oG(z) | 
        \leq 
    \tilde c |(z - \bz_p)|    
        \, ,
\end{align}
for every $z \in K$, where $\tilde c = \tilde c(\hat x, K) \in (0,+\infty)$.
In particular, from \eqref{eq:double_bound_nabla_b}, \eqref{eq:infAsupA}, and \eqref{eq:lb_grad}, the validity of \eqref{eq:sharp_behavior_nonsing} readily follows.
\end{proof}

The case of $p \in (1,2)$ presents a different type of singularities. In particular, they may occur when $z=\hat x_i$ rather than $z=\xp(\hat x_2, \dots, \hat x_N)$, as explained in the next proposition. For convenience, we set $\beta_p := - \alpha_p \in (0,+\infty)$ whenever $p \in (1,2)$. We denote by $\displaystyle \lambda_{\min} := \min_{i=2, \dots, N} \lambda_i >0$.

\begin{proposition}[Estimates of $\nabla b^{-1}$, $p \in (1, 2)$]
\label{prop:pbary_p<2}
 Fix $p \in (1,2)$ and $\hat x \in \RR^{d(N-1)}$, set $b=b_{\hat x}$, then $b^{-1}\in C^1(\RRd\setminus\bigcup_{i=2}^N \{\hat x_i\})$. Moreover for every $ z \in \RR^d$ with $z \neq \hat x_i$ for every $i=2, \dots, N$,  we have that $\sigma(\nabla b^{-1}(z))\subset\RR$ and 
there exists $C_p \in (0,+\infty)$ such that
      \begin{align}
\label{eq:sharp_behavior_nonsing_p<2}
        C_p
        \frac
        {\lambda_{\min}(1-\lambda_1)^{\beta_p}}
        {\lambda_1^{1+ \beta_p}}
            \bigg( \frac{|z - \xp(\hat x)|}
            {M(z)}
            \bigg)^{\beta_p}
            \bigg( \frac{M(z)}{m(z)} \bigg)^{2-p}
                \leq 
            \zeta-1 
                \leq 
        (1+\beta_p)
        \frac
        {(1-\lambda_1)^{\beta_p}}
        {\lambda_1^{1+ \beta_p}}
            \bigg( \frac{M(z)}{m(z)} \bigg)^{2-p} 
                \hspace{-2mm} \, ,  \quad 
   \end{align}
for every $\zeta \in 
\sigma(\nabla b^{-1}(z))$, and $M(z) := \max_{i \geq 2} |\hat x_i - z|$ and $m(z) := \min_{i \geq 2} |\hat x_i-z|$.
\end{proposition}

\begin{proof}
The proof follows the same lines as the proof of the same result for $p \geq 2$. Once again, we set $\bz_p := \xp(\hat x)$.
Notice that in this case, $\oG\in C^0(\RRd)$ and $\oG\in C^1(\RRd\setminus\{\hat x_2,\dots,\hat x_N\})$. On the other hand, the function $\frac{\cdot}{|\cdot|^{\alpha_p}}=(\cdot)|\cdot|^{\beta_p}$ is everywhere differentiable. Thus $b^{-1}\in C^1(\RRd\setminus\{\hat x_2,\dots,\hat x_N\})$. We have seen in  \eqref{eq:formula_grad_b_inverse} that 
\begin{align}
\label{eq:formula_grad_b_inverse_p<2}
    \nabla b^{-1}(z)
        =
    \text{Id}
        - 
    \frac1{\lambda_1^{1+\beta_p}}
    \bigg[
        \text{Id}
            +
        \beta_p 
        \frac
            {\oG(z) \otimes \oG(z)}
            {|\oG(z)|^2}
    \bigg]
        |\oG(z)|^{\beta_p}
              \nabla \oG(z)
        \, , \qquad 
    z \in \RR^d\setminus\{\hat x_2,\dots,\hat x_N\} 
        \, .
\end{align}

For every $z \in \RR^d\setminus\{\hat x_2,\dots,\hat x_N\}$, arguing as in \eqref{eq:rank1_bound} we obtain
\begin{align}
\label{eq:rank1_bound_p<2}
    0 
        <
     \text{Id}
        \leq 
    \text{Id}
        + 
    \beta_p \frac
        {\oG(z) \otimes \oG(z)}
        {|\oG(z)|^2}
        \leq 
    (1+\beta_p)\text{Id}
        \, .
\end{align} 
Arguing as in \eqref{eq:bound_-gradG} and \eqref{eq:double_bound_nabla_b}, we obtain that for every $z \in \RR^d\setminus\{\hat x_2,\dots,\hat x_N\}$, 
\begin{align}
\label{eq:double_bound_nabla_b_p<2}
    \bigg(
    1
        +
    \frac{p-1}{ \lambda_1^{1+\beta_p}}
            \tilde A(z)
        |\oG(z)|^{\beta_p}
    \bigg)
        \leq
    \zeta 
        \leq
    \bigg(
    1
        +
    \frac{
    1+\beta_p}{ \lambda_1^{1+\beta_p}}
            \tilde A(z)
        |\oG(z)|^{\beta_p}
    \bigg)
        \, ,
\end{align}
for $\zeta \in \sigma(\nabla b^{-1}(z))$,
where $\tilde A(z) := \sum_{i=2}^N \lambda_i \frac1{|z-  \hat x_i|^{2-p}}$. Note that
\begin{gather}
   \lambda_{\min} \frac{1}{m(z)^{2-p}}\leq \tilde A(z) 
        \leq 
    \sum_{i=2}^N 
        \lambda_i \frac{1}{m(z)^{2-p}} 
            = 
        (1-\lambda_1) \frac{1}{m(z)^{2-p}} 
    \, , \quad \text{and} 
\\
    |\oG(z)|^{\beta_p}
        \leq 
        \bigg(
    \sum_{i=2}^N 
        \lambda_i 
        M(z)^{p-1}
        \bigg)^{\beta_p}
           =
        (1-\lambda_1)^{\beta_p}
        M(z)^{2-p}
            \, .
\end{gather}
The inequalities in \eqref{eq:sharp_behavior_nonsing_p<2}  follow from \eqref{eq:double_bound_nabla_b_p<2} 
as soon as we show the corresponding lower bound on $|\oG|$ (in a similar spirit to \eqref{eq:final_lb}) 
 \begin{align}
     \label{eq:final_lb_p<2}
     |\oG(z)| \ge C_p (1-\lambda_1) M(z)^{p-2}|z- \bz_p|
        \, ,
 \end{align}
 for some $C_p \in \RR_+$ depending only on $p$, by using that $\beta_p = (2-p)(1+\beta_p)$.

 We are thus left to show the validity of \eqref{eq:final_lb_p<2}. The counterpart of \eqref{eq:strong_conv_p>2} for $p \in (1,2)$ reads 
 \begin{align}
 \label{eq:strong_conv_p<2}
     \langle 
        \nabla g(u) - \nabla g(v) 
            , 
        u- v
     \rangle
        \geq 
    p(p-2)2^{p-1} (|u| + |v|)^{p-2} | u - v |^2
        \, .
 \end{align}
Since $\zp$ lies in the convex hull of the $\hat x_i$ (cf. Remark \ref{rem:first_properties_xp}), we have also have that $|\hat x_i - \zp| \leq 2M$ for every $i=2, \dots, N$. Therefore, arguing as in \eqref{eq:convex_lb}, from \eqref{eq:strong_conv_p<2} we conclude that 
\begin{align}
    \langle \oG(z), \zp - z \rangle 
        \geq 
   p(p-2)2^{p-1} 3^{p-2} (1-\lambda_1) M(z)^{p-2}| z - \zp|^2
        \, , 
\end{align}
which indeed implies the sought lower bound \eqref{eq:final_lb_p<2}.
\end{proof}

Note that if $N=2$, $m(z)=M(z) = |z - \xp(\hat x)|$, coherently  with Remark~\ref{rem:b_special_cases}.

\section{$L^q$-regularity (and counterexamples) of the $p$-Wasserstein barycenter}
\label{sec:integrabilityP}
In this section, we prove Theorem~\ref{thm:integrability_distant_intro}, on the $L^q$-integrability of the $p$-Wasserstein barycenter. We start from the case where $\mu_2, \dots, \mu_N$ are single Deltas. We provide a preliminary integrability result for distant supports and a counterexample to the general integrability for $p>2$, even under the stronger assumption that $f_1 \in L^\infty$ and is compactly supported. This is in stark contrast to what happens in the case of $p$-Wasserstein geodesics, i.e. $N=2$, or when $p=2$ (see Remark~\ref{rem:reg_geodesics_and_W2}).

Finally, we extend the proof of the integrability of the density of the barycenter under the assumption of distant supports (cf. \eqref{eq:D>0}) beyond the empirical case to general $\mu_2, \dots, \mu_N$, arguing by approximation. 

The analysis relies on the preliminary analysis on the map $\xp$ (and thus $b$) which we discussed in the previous section.

Throughout this section, we assume that $\mu_1\ll \LL^d$. By Theorem~\ref{th:absolutecontinuity} we know that $\nup\ll\LL^d$. We denote with $f_1,g_p\in L^1(\RRd) $ the corresponding densities, i.e. $\mu_1=f_1 \dd\LL^d$ and $\nup=g_p\dd\LL^d $.

\subsection{Estimates, integrability, and counterexamples for concentrated measures}
\label{sec:integrability}
Let us assume that there exist $\hat x_2,\dots,\hat x_N$ such that $\mu_i=\delta_{\hat x_i}$, for every $i=2,\dots, N$. 
Then $T_i(\cdot)=\hat x_i$ is constant for every $i=2,\dots,N$.
Notice that in this case, the map $\barp$ coincides with the map $b(x_1) := \xp(x_1, \hat  x_2, \dots, \hat x_N)$.  Recall that $\barp$ is globally injective (see Corollary~\ref{cor:barp_gleich_S1}).

In this section, we see that if the supports are distant enough, $\barp$ is a diffeomorphism with a quantitative lower bound on $|\det \nabla \barp|$, which provides in this simpler setting a first integrability result (cf. Proposition~\ref{prop:discrete_integrability}). Otherwise, we may encounter counterexamples (see Example~\ref{ex:counterexamples}).

We assume compactly supported measures, i.e.  there exists $M>0$, such that 
\begin{align}
\label{eq:compact_supports}
\tag{$\mathcal{C}_{\rm pt}$}
        \spt\mu_i\subset B_{\frac M 2}, \quad \text{for every} \ i=1,\dots,N 
        \, .
    \end{align}

For simplicity, we denote by $\zp := \xp(\hat x_2, \dots, \hat x_N)$. 

 Recall that $b^{-1}$ is an injective map which is $C^1$ up to a finite number of points. Therefore, as $\barp = b$ and by means of a change of variables formula \cite[Theorem 7.1]{Ambrosio-Brue-Semola:2024} one can show that  
 \begin{align}
 \label{eq:formula_gp}
     g_p
        =
    (f_1 \circ \barp^{-1})
        J_{\barp^{-1}}
            \, , \qquad \text{where} \quad 
        J_{\barp^{-1}}(z) = |\det \nabla \barp^{-1}(z)|  
            \, ,
 \end{align}
for $\LL^d$-a.e. $z \in \RR^d$.
 
\begin{proposition}[Discrete integrability]
\label{prop:discrete_integrability}
Let $\mu_1 = f_1 \dd \LL^d$ satisfying \eqref{eq:compact_supports} and with $f_1 \in L^q$. 
Assume 
\begin{enumerate}
\item\label{ass:ass1_discrte_integrability} $p \geq 2$ and $\zp \notin \spt\mu_1$, or
\item\label{ass:ass2_discrte_integrability} $1<p<2$ and $\barp(x_1)\neq \hat x_i$, for every $i=2,\dots, N$, and for every $x_1\in\spt\mu_1$,
\end{enumerate} 
Then $g_p \in L^q$ and we have 
\begin{align}
\| g_p \|_{L^q} \leq \frac{C}{\lambda_1^{d q'(1-\alpha_p)}}\| f_1 \|_{L^q}  \, ,
\end{align}
where $C \in \RR_+$ depends on $M$, $p$, $d$, and - respectively - on
\begin{enumerate}
\item $\dist\Big(\barp(\spt\mu_1), \zp\Big) >0$ , or $\qquad$ (2) $\displaystyle \min_{i=2,\dots,N}$ $\dist\Big(\barp(\spt\mu_1), \hat x_i\Big) >0$ \, .
\end{enumerate}
\end{proposition}

\begin{proof}

 We employ this formula to compute the $L^q$-norm of the barycenter $\nu_p$ in this discrete setting: by applying a change of variables formula for $b^{-1}$ once again, we find that
 \begin{align}
 \label{eq:Lq-estimate-1}
     \| g_p \|_{L^q(\RR^d)}^q
        =
    \int 
        \big|
            (f_1 \circ \barp^{-1})(z)
        \big|^q 
            \cdot 
        J_{\barp^{-1}}(z)^q
            \dd z
        &=
    \int 
        |f_1(x_1)|^q 
            \cdot 
            (J_{\barp^{-1}} \circ \barp)(x_1)
            ^{q-1}
            \dd x_1    
                \, .
 \end{align}
In the case $p\ge2$, an application of Proposition~\ref{prop:pbary}  (in particular \eqref{eq:nonsharp_lb}) and the fact that $M \geq M(z) \geq |z-\xp(\hat x)|$ for every $z \in \barp(\spt\mu_1)$ (by \eqref{eq:compact_supports} as well as \eqref{eq:euler-lagrange_2}) ensure that
\begin{align}
\label{eq:asymptotics_Jacobian}
    J_{\barp^{-1}} \circ \barp(\cdot)
   &
   \lesssim 
    M^{d(p-2)}\lambda_1^{d(\alpha_p-1)}
      |\barp(\cdot)-\zp|^{d(2-p)}
                    \, .
\end{align}
As $\barp$ is injective and $\barp(\zp) = \zp$, from $\zp \notin \spt\mu_1$ we ensure that ${\rm dist}{(\barp(\spt\mu_1), \zp)>0}$. Therefore we conclude 
\begin{align}
    \| g_p \|_{L^q}^q 
        &\lesssim 
    \frac{C}{\lambda_1^{d(q-1)(1-\alpha_p)}}
    \| f_1 \|_{L^q}^q
        \text{dist}
            \Big(
                \barp(\spt\mu_1), \zp
            \Big)^{-d(q-1)(p-2)}
            \, ,
\end{align}
where $C=C(p,d,M) \in \RR_+$, as claimed.

A very similar argument applies for the case of $1<p<2$, and by Proposition~\ref{prop:pbary_p<2} we get
\begin{align}
\label{eq:asymptotics_Jacobian_p<2}
    J_{\barp^{-1}} \circ \barp(\cdot)
   \lesssim 
    M^{d(2-p)}
    \lambda_1^{d(\alpha_p-1)}
      (m \circ \barp)(\cdot)^{d(p-2)}
                    \, ,
\end{align}
where we recall $m(z) := \min_{i=2,\dots,N}|z-\hat x_i|$.
We thus obtain 
\begin{align}
    \| g_p \|_{L^q}^q 
 \lesssim 
    \frac{C}{\lambda_1^{d(q-1)(1-\alpha_p)}}
    \| f_1 \|_{L^q}^q
        \min_{i=2,\dots, N}\dist\Big(\barp(\spt\mu_1), \hat x_i\Big)^{d(q-1)(p-2)}\, ,
\end{align}
as claimed.
\end{proof}

\begin{remark}\label{rem:reg_geodesics_and_W2}
We remark that 
\begin{itemize}
    \item When $N=2$, thanks to \eqref{eq:barp_sharp_nu_1} and \eqref{eq:euclidian_geodesics}, one easily infers that \begin{align}\label{eq:Wp_geodesics}
\nup\coloneq\Wbar((\mu_1,(1-t)),(\mu_2,t))=\left((1-t(p))\mathrm{Id}+t(p)T_2\right)_\sharp\mu_1
    \, ,  
\end{align}
where $t(p)\coloneq t^{\frac{1}{p-1}}\big( t^{\frac{1}{p-1}}+(1-t)^{\frac{1}{p-1}} \big)^{-1}$ and $T_2$, given by Theorem~\ref{th:sparsityplan}, is also the optimal map for $W_p(\mu_1,\mu_2)$. Thus, as expected, the $p$-Wasserstein barycenter with weights $(1-t),t$ is the $p$-Wasserstein geodesic parametrized by $t(p)$. In the case of Proposition~\ref{prop:discrete_integrability}, with $\mu_2=\delta_{\hat x_2}$, as highlighted in Remark~\ref{rem:b_special_cases}, one has $\nabla\barp(x_1)=(1-t(p))$, hence by the change of variable formula  
\begin{align}
  \| g_p \|_{L^q}^q \le (1-t(p)) ^{d(1-q)} ||f_1||_{L^q}^q \, ,
\end{align}
regardless of any assumptions on $\spt \mu_1$, consistently with \cite[Lemma 4.22]{San15}. 
\item When $p=2$, thanks to \eqref{eq:barp_sharp_nu_1} and to Remark~\ref{rem:b_special_cases},  from the change of variable formula one directly infers that 
\begin{align}
  \| g_p \|_{L^q}^q \le \lambda_1 ^{d(1-q)} ||f_p||_{L^q} \, ,    
\end{align}
and without extra assumptions, such as 
$\text{dist}
    \Big(
        \barp(\spt\mu_1), \zp
    \Big) >0
$, consistently with the results in \cite[Theorem 5.11]{AC11}.
\end{itemize}
\end{remark}

\begin{example}[Counterexamples to integrability]
\label{ex:counterexamples}
 The following examples show that, when $p\neq 2$,  $N>2$,  and $\hat x \notin {\rm diag}(\RR^{d(N-1)})$, assumptions \eqref{ass:ass1_discrte_integrability} and \eqref{ass:ass2_discrte_integrability} in Proposition~\ref{prop:discrete_integrability} can not be omitted in general. 
Let $B \subset \RR^d$ be an open set of $\LL^d$-measure $1$ and consider $\mu_1 := f_1 \LL^d$ where $f_1 := \mathds{1}_B$.
We consider two cases of interest.

\noindent 
(\textit{Counterexamples for $p > 2$)}. \ 
For superquadratic costs, we assume that $\zp \in B$. Then we claim that the associated Wasserstein barycenter $\nu_p = g_p \LL^d$ is such that 
    \begin{align}
    \label{eq:discrete_non_integrability1}
        g_p \in L^q 
            \Leftrightarrow 
        q < \frac1{\alpha_p} = \frac{p-1}{p-2}
            \, .
    \end{align}
Note that this provides a counterexample to the $q$-integrability for $p>2$ as soon as $q \geq \alpha_p^{-1}$.

In order to show \eqref{eq:discrete_non_integrability1}, from \eqref{eq:formula_gp} and Proposition~\ref{prop:pbary} (in particular \eqref{eq:sharp_behavior_nonsing}), we see that
\begin{align}\label{eq:change_of_var_uniform_density}
    \int
        |g_p(z)|^q 
            \dd z 
        =
    \int 
        |f_1 \circ \barp^{-1}(z)|^q
        |J_{\barp^{-1}}(z)|^q
            \dd z
        &=
    \int_{\barp(B)}
        \big| 
            J_{\barp^{-1}}(z)
        \big|^q
            \dd z
\\
        &\simeq
    \int_{\barp(B)}
        \frac
            1
            {|z - \zp|^{dq\alpha_p}}
                \dd z
        \, .
\end{align}
The claimed conclusion then follows from the fact that $\barp$ is an invertible open (it has a continuous inverse) map and the fact that 
\begin{align}
    \int_{B_1(0)}
        \frac
            1
            {|z|^{dq\alpha_p}}
                \dd z
            < \infty 
        \quad \Leftrightarrow \quad 
            d q \alpha_p < d 
        \quad \Leftrightarrow \quad 
            q < \frac1{\alpha_p}
                \, .
\end{align}

\bigskip 
\noindent 
(\textit{Counterexamples for $1<p<2$}). \ 
Assume in this case
that $\hat x_i \in \barp(B)$ and that additionally $\hat x_i \neq \zp$,  for some $i=2,\dots,N$.  Then we have that
    \begin{align}
    \label{eq:discrete_non_integrability2}
        g_p \in L^q 
            \Leftrightarrow 
        q < \frac1{2-p}
            \, .
    \end{align}
Indeed, the estimate \eqref{eq:sharp_behavior_nonsing_p<2}, provided by Proposition~\ref{prop:pbary_p<2}, shows that 
     $
         J_{\barp^{-1}}(z) \simeq | z - \hat x_i |^{d(p-2)}
     $
      as $z \to \hat x_i$.
\end{example}

\subsection{$L^q$-estimates for general measures with distant support}
\label{sec:main_proof_by_approx}
In this section we prove Theorem \ref{thm:integrability_distant_intro}. In particular, 
by using a discretization approach, we show that Proposition~\ref{prop:discrete_integrability} can be extended to general measures $\mu_2, \dots , \mu_N$. For $q \in (1,+\infty)$, we denote by $q'\in (1,+\infty)$ its conjugate. As conclusions and assumptions are different depending on whether $p<2$ or $p>2$, we divide the proof into these two cases.

\subsubsection{Integrability for $p\geq 2$}\label{subsubsect:integrability_p>2}
Recall that we assume compactly supported measures \eqref{eq:compact_supports} and work with the assumption
 \begin{align}
 \label{eq:D>0}
 \tag{$hp_1$}
    \mathcal{D}
      :=
    \inf_{(x_1, \dots, x_N) \in \bigtimes_{i=1}^N \spt\mu_i}
    \big| 
        \xp(x_1, x_2, \dots , x_N) - \xp(x_2, \dots, x_N) 
    \big| >0
        \, .
 \end{align}
Note that, when the measures are  compactly supported, \eqref{eq:D>0_intro} is equivalent to
\begin{align}
\label{eq:distant_supports}
\tag{$hp_1^*$}
   \text{dist} 
   \bigg(
        \spt\mu_1 
            , 
        \xp
        \Big(   
            \bigtimes_{i=2}^N \spt\mu_i
        \Big)
   \bigg) >0
    \, . 
\end{align}
 Indeed, by contradiction: assume  $\mathcal{D}=0$, which, by compactness and continuity of $
\xp$, means that one can find  points $x_1, \dots, x_N$ with $x_i \in \spt \mu_i$ such that 
\begin{align}
    \xp(x_1, x_2, \dots , x_N) 
        = 
    \xp(x_2, \dots, x_N)
        = 
    \xp\big( \xp(x_2, \dots, x_N), x_2, \dots , x_N \big) 
        \, . 
\end{align}
As the barycenter map $x_1 \mapsto \xp(x_1, x_2, \dots, x_N)$ is injective  for every $x_2, \dots, x_N$, we conclude that 
\begin{align}
    x_1 = \xp(x_2, \dots, x_N)
        \, , \qquad 
    \text{with } 
        x_i \in \spt\mu_i 
            \, ,
\end{align}
which clearly contradicts the assumption \eqref{eq:distant_supports}.  Vice versa, if there exist $\xdots$ such that $x_1=\xp(x_2, \dots, x_N)$, it must hold $\xp(x_1,x_2, \dots, x_N)=\xp(x_2, \dots, x_N)$, contradicting \eqref{eq:D>0}.
\begin{remark}[Push-forward via an injective map]
\label{rem:pushfor_inj}
Let $\mu,\nu\in\Prob(\RRd)$ and $f$ a measurable function such that $\nu=f_\sharp \mu$. If $f$ is injective, then for every Borel set $A$, $\nu|_{f(A)}=(f|_A)_\sharp\mu$. Indeed, given a Borel set $B \subset \RR^d$, 
    \begin{align}
        \nu|_{f(A)}(B)
            &=
        \nu(B\cap f(A))
            \underset{\nu=f_\sharp\mu}{=}
        \mu
        \big(
            f^{-1}(B\cap f(A))
        \big)
            \underset{(*)}{=}
        \mu
        \big(
            f^{-1}(B)\cap f^{-1}(f(A))
        \big)   
    \\
            &\underset{f \, \text{injective}}{=}
        \mu 
        \big(
            f^{-1}(B)\cap A)
        \big)
            =
        \mu(f|_A^{-1}(B))
            =
        (f|_A)_\sharp\mu(B)
            \, , 
    \end{align}
    where in $(*)$ we simply used that $f^{-1}(C\cap D)=f^{-1}(C)\cap f^{-1}(D)$ for every set $C,D \subset \RR^d$.
\end{remark}
\begin{proof}[Proof of Theorem \ref{thm:integrability_distant_intro} $(p \geq 2)$]
    We proceed in two steps: first, we extend the validity of Proposition~\ref{prop:discrete_integrability} to the case when $\mu_2, \dots, \mu_N$ are \textit{empirical measures}. Secondly, we will make use of this result and provide an approximation argument with respect to which the sought inequality and our assumption \eqref{eq:D>0} are stable.

\smallskip 
\noindent
    \underline{Step 1}. \ 
Assume there exist atoms $\{x_j^i\}_{i,j} \subset \RR^d$ and masses $\{ m_j^i \}_{i,j} \subset \RR_+$ so that 
\begin{align}
\label{eq:empirical}
    \mu_i 
        =
    \sum_{j=1}^{K_i} 
        m_j^i   \delta_{x_j^i}
        , \qquad 
    \text{for }i = 2, \dots, N \, , \quad  j = 1 , \dots , K_i \in \mathbb{N}
        \, .
\end{align}
As $\mu_1 \ll \LL^d$, we have $\nu_p = g_p \LL^d = (\barp)_{\#}(\mu_1)$ with $\barp(x_1) = \xp(x_1, T_2(x_1), \dots, T_N(x_1))$. As the measures are empirical, for every $i=2, \dots, N$ the map
$
    T_i = R_i \circ S_1 
$
is piecewise constant and assumes precisely the values $\{ x_1^i , \dots, x_{K_i}^i \} \subset \RR^d$. 

We divide the support of $\mu_1$ in $\prod_{i=1}^N K_i$ sets, using the maps $T_2, \dots, T_N$, as follows: for each choice of multi-index $\bj= (j_2, \dots, j_N) \in \cJ$ with $j_i \in \{1, \dots, K_i\}$, we define 
\begin{align}
    \Omega_{\bj}
        :=
    \Big\{
        x_1 \in \RR^d \cap \spt \mu_1 
            \suchthat  
        T_i(x_1) = x_{j_i} 
    \Big\}
        \, .
\end{align}
Note that such sets are, by construction, disjoint: $\Omega_{\bj} \cap \Omega_{\bk} = \emptyset$ if $\bj \neq \bk$. Consider now the image of such sets via $\barp$ and set $\Omega_{\bj}^p:= \barp(\Omega_{\bj})$ for $j \in \cJ$. As $\barp$ is injective (Corollary~\ref{cor:barp_gleich_S1}), we conclude that the new sets are also disjoint, and all in all, we have 
\begin{align}
\label{eq:partition_image_disjoint}
    \Omega_{\bj}^p \cap \Omega_{\bk}^p = \emptyset 
        \quad 
    \text{if } \, \bj \neq \bk
        \qquad \text{and} \qquad 
    \nu_p 
    \Big( 
        \bigcup_{\bj \in \cJ}
        \Omega_{\bj}^p 
    \Big)
        =
        \nu_p 
    \big( 
        \barp(\spt \mu_1)
    \big)
        =
    1
        \, .
\end{align}
Moreover, for every given multi-index $\bj \in \cJ$, we have that 
\begin{align}
\label{eq:restriction_lip}
    \barp\big|_{\Omega\bj} 
        = 
    b_\bj
        \qquad \text{where} \quad 
    b_\bj(x_1) 
        := 
    \xp(x_1, x_{j_2}, \dots, x_{j_N})
        \, , \quad x_1 \in \RR^d 
            \, .
\end{align}
In particular, it is the restriction of a Lipschitz function in $\RR^d$ (cf. Lemma \ref{lem:lipschitz}). 
As a consequence, using once again that $\barp$ is injective (cf. Corollary~\ref{cor:barp_gleich_S1} and Remark \ref{rem:pushfor_inj}), we conclude that
\begin{align}
    \nu_p\big|_{\Omega_\bj^p}
        \stackrel{\text{inj}}{=}
   (\barp)_{\#}
   \Big(
        \mu_1\big|_{\Omega_\bj}
   \Big)
        =
    (b_\bj)_{\#} 
   \Big(
        \mu_1\big|_{\Omega_\bj}
   \Big)
        =
    (b_\bj)_{\#} \mu_1 \big|_{\Omega_\bj^p}
        \, , \qquad
    \forall \bj \in \cJ 
        \, .
\end{align}
From the above equality, using the change of variable formula  for locally Lipschitz map \cite[Theorem 7.1]{Ambrosio-Brue-Semola:2024} $b_\bj$, we deduce that, locally on $\Omega_\bj$, the density $g_p$ of $\nu_p$ can be written as 
\begin{align}
\label{eq:formula_gp_cov}
    (g_p \circ b_\bj)  J_{b_\bj} 
        =
    f_1 
        \qquad 
    \text{on } \Omega_\bj
        \qquad \Longleftrightarrow \qquad 
    g_p 
        =
    \big( f_1 \circ b_\bj^{-1}  \big)
        J_{b_\bj^{-1}} 
        \qquad 
    \text{on } \Omega_\bj^p
        \, .
\end{align}
Here $J_{b_\bj}$ denotes the Jacobian of $b_\bj$, i.e. 
\begin{align}
\label{eq:jac_inverse}
    J_{b_\bj}(x) := | \det \nabla b_\bj (x) | 
        = 
    J_{b_\bj^{-1}}(b_\bj(x))^{-1}
        \, , \qquad 
    x \in \RR^d 
        \, .
\end{align}
Therefore, by \eqref{eq:formula_gp_cov} and using that $\{ \Omega_\bj^p\}_\bj$ are disjoint, we can compute the norm
\begin{align}
\label{eq:estimate_Lq_general}
    \| g_p \|_{L^q}^q
        \stackrel{\text{disj}}{=}
    \sum_{\bj \in \cJ}
        \int_{\Omega_\bj^p}
            |g_p(z)|^q \dd z 
        &\stackrel{\eqref{eq:formula_gp_cov}}{=} 
    \sum_{\bj \in \cJ}
        \int_{\Omega_\bj^p}
            \big|
            \big( f_1 \circ b_\bj^{-1}  \big) (z)
        J_{b_\bj^{-1}}(z) \big|^q \dd z
\\
        &\stackrel{\text{c.o.v.}}{=} 
    \sum_{\bj \in \cJ}
        \int_{\Omega_\bj}
            | f_1(x_1) |^q
            (J_{b_\bj^{-1}}\circ b_\bj)(x_1)^{1-q}\dd x_1 
                \, ,
\end{align}

From now on, we simply argue in the same way we did for Proposition~\ref{prop:discrete_integrability}.
By a direct application of Proposition~\ref{prop:pbary} (in particular \eqref{eq:nonsharp_lb}), \eqref{eq:jac_inverse}, and \eqref{eq:D>0},
  we conclude that  for $x_1 \in \Omega_\bj$,
  \begin{align}
  \label{eq:lb_Jb_uniform}
      J_{b_\bj^{-1}} \circ b_\bj(\cdot)
   &
   \leq 
    C\lambda_1^{d(\alpha_p-1)}
      |b_\bj(\cdot)-\zp|^{d(2-p)}
      \leq 
    C\lambda_1^{d(\alpha_p-1)}
      \mathcal{D}^{d(2-p)}
  \end{align}
for some $C\in \RR_+$ (depending on $M$, given in \eqref{eq:compact_supports}), where $\mathcal{D}>0$ is the one given by \eqref{eq:D>0}.
By inserting this upper bound into \eqref{eq:estimate_Lq_general} we conclude that 
\begin{align}
    \| g_p \|_{L^q}^q
       \leq 
    \frac{C^{q-1}}{\big( \lambda^{ d(1-\alpha_p)}\mathcal{D}^{d(p-2)}\big)^{q-1}}
    \sum_{\bj \in \cJ}
        \int_{\Omega_\bj}
            | f_1(x_1) |^q
            \dd x_1 
    \stackrel{\text{disj}}{=}
      \frac{C^{q-1}}{\big( \lambda^{ d(1-\alpha_p)}\mathcal{D}^{d(p-2)}\big)^{q-1}} 
    \| f_1 \|_{L^q}^q 
                \, ,  
\end{align}
which provides the claimed integrability. 

\smallskip 
\noindent
    \underline{Step 2}. \ 
     We proceed by approximation. Let $\mu_1, \dots, \mu_N$ be measures satisfying \eqref{eq:D>0}. We construct $N-1$ sequences of measures $\mu_2^n , \dots, \mu_N^n$ so that, for every $i=2,\dots, N$,  
     \begin{enumerate}
         \item For every $n \in \NN$, $\mu_i^n$ is an empirical measure of the form \eqref{eq:empirical}.
         \item For every $n \in \NN$, we have $\spt \mu_i^n \subset \spt \mu_i$. 
         \item We have  $\mu_i^n \to \mu_i$ narrowly (i.e. in duality with $C_b(\RR^d)$). 
     \end{enumerate}
     The existence of such a construction is standard and therefore omitted. 

     Observe that by construction, we hence have 
     \begin{align}
          \bigtimes_{i=1}^N \spt\mu_i^n 
            \subset 
         \bigtimes_{i=1}^N \spt\mu_i
            \quad \Longrightarrow \quad 
        \mathcal{D}_n:= \mathcal{D}(\mu_1, \mu_2^n, \dots, \mu_N^n) 
            \geq 
         \mathcal{D}(\mu_1, \mu_2, \dots, \mu_N) = \mathcal{D}> 0
            \, ,
     \end{align}
     for all $n \in \NN$. 
     Denoted by $\nu_p^n = g_p^n \LL^d$ the corresponding $W_p$-barycenter of the measures $\mu_1, \mu_2^n, \dots, \mu_N^n$, we can therefore apply the reasoning from Step 1 and obtain
     \begin{align}
          \label{eq:lsc}
         \| g_p^n \|_{L^q}
            \leq 
         \frac
            {C}
            {\big( \lambda^{ d(1-\alpha_p)}\mathcal{D}_n^{d(p-2)}\big)^{q'}}
                \| f_1 \|_{L^q} 
            \leq
         \frac
            {C}
            {\big( \lambda^{ d(1-\alpha_p)}\mathcal{D}^{d(p-2)}\big)^{q'}}
                \| f_1 \|_{L^q} 
                    \, , \quad 
            \forall n \in \NN 
                \, .
     \end{align}
     The proof will thus be complete if we are able to show that  
     \begin{align}
          \label{eq:lsc_2}
         \liminf_{n\to \infty}
            \| g_p^n \|_{L^q} \geq \| g_p \|_{L^q} 
                \, .
     \end{align}
 We claim that $\nu_p^n \to \nu_p$ narrowly as $n \to \infty$. Indeed, by a standard stability result of optimal transport, see for instance \cite[Theorem 6.13]{Fri24} or \cite[Theorem 6.8]{Ambrosio-Brue-Semola:2024}, if $\gammap^n$ and $\gammap$ are respectively the optimal plans for \eqref{MMpbary}  with marginals $\mu_1^n,\dots,\mu_N^n$ and $\mu_1,\dots,\mu_N$, $\gammap^n\to \gammap$. We conclude by observing that  $\nu_p^n$ and $\nup$ are the push-forward with respect to the continuous map $\xp$ of, respectively, $\gammap^n$ and $\gammap$. The claimed lower semicontinuity property of the $L^q$-norm with respect to the narrow convergence of measures then follows by standard arguments.  Indeed, without loss of generality, we can assume $\sup_{n \in \NN} \| g_p^n \|_{L^q}<\infty$, if not there is nothing to prove. In this case, as $L^q$ is reflexive for $q \in (1,+\infty)$, up to a non-relabeled subsequence, there must exist $g \in L^q$ so that $g_p^n \rightharpoonup g$ weakly in $L^q$, or in other words 
    \begin{align}
        \lim_{n \to \infty} 
        \int g_p^n \varphi \dd x 
            =
        \int g \varphi \dd x 
            \,  ,\qquad 
        \forall \varphi \in L^{q'}
            \, .
    \end{align}
    As we already know that $\nu_p^n \to \nu_p$ narrowly, we conclude that 
    \begin{align}
        \int g \varphi \dd x 
            =
        \int g_p \varphi \dd x 
            \, , \qquad 
        \forall \varphi \in L^{q'} \cap C_b(\RR^d)
            \, .
    \end{align}
    As 
    $ L^{q'}(\RR^d) \cap C_b(\RR^d) \hookrightarrow L^{q'}(\RR^d)$ densely, this implies $g=g_p$. Finally, \eqref{eq:lsc_2} follows by the standard lower-semicontinuity of the $L^q$-norm with respect to the weak convergence. 
\end{proof}

\subsubsection{Integrability for $p\in(1,2)$}\label{subsubsect:integrability_p<2}
By means of the very same approach, one can show that a result of the same type as in  Proposition~\ref{prop:discrete_integrability} can be obtained when $p \in (1,2)$ under suitable assumptions, which take care of the different type of singularities, as already illustrated in Proposition~\ref{prop:pbary_p<2}.
Recall that in this regime we work once again with compactly supported measures \eqref{eq:compact_supports} as well as the slightly stronger geometric assumption 
 \begin{align}
 \label{eq:ass on supports}\tag{$hp_2$}
    m
        :=
    \inf_{i \in \{1, \dots, N\}}
         \inf_{(x_1, \dots, x_N) \in \bigtimes_{i=1}^N \spt\mu_i}
    \big| 
        x_i 
            -
        \xp(x_1, x_2, \dots , x_N)
    \big| >0
        \, .
 \end{align} 
%
As usual, for $q \in (1,+\infty)$, we denote by $q'\in (1,+\infty)$ its conjugate.

\begin{proof}[Proof of Theorem \ref{thm:integrability_distant_intro} $(p \in (1,2))$]
We just sketch the steps of the proof as it contains no substantially new ideas to the case with $p \geq 2$. The approximation argument used in that case works practically unchanged, with the only difference consisting in the lower bound in \eqref{eq:lb_Jb_uniform}, which, for $p \in (1,2)$, as a consequence of Proposition~\ref{prop:pbary_p<2}, is replaced by
 \begin{align}
 \label{eq:lb_Jb_uniform_p<2}
      J_{b_\bj^{-1}} \circ b_\bj(\cdot)
   &
   \leq 
    C\lambda_1^{-d(1+\beta_p)}
      (m_j \circ b_\bj)(\cdot)^{d(p-2)}
      \leq 
    C\lambda_1^{-d(1+\beta_p)}
      m^{d(p-2)}
  \end{align}
where $m_j(z) := \min_{i \geq 2} |x_i^j-z|$ satisfies $(m_j \circ b_\bj)(x_1) \geq m$ for every $\bj$ and $x_1 \in \spt\mu_1$ by construction. The rest of the proof works verbatim. 
\end{proof}

\section{A general estimate on the $L^q$-norm of $p$-Wasserstein barycenters}
\label{sec:general_estimate}The goal of this section is to prove the general estimate on the norm of the density of the $p$-Wasserstein barycenter stated in Proposition~\ref{prop:general_Lq_estimate_intro}. This holds without any assumption on the supports and it highlights possible sources of nonintegrability. 
This provides (cf. Remark~\ref{rem:alternative_proof_integrability}) an alternative proof of the integrability result given by Theorem~\ref{thm:integrability_distant_intro}, under the slightly stronger assumption that \eqref{eq:ass on supports_intro} holds for every $1<p<\infty$.

Recall the definition of approximate differentiability, taken from Definition 5.5.1 in \cite{AGS}.
\begin{definition}[Approximate differentiability]\label{def:approx_diff}
Let $\Omega \subset \mathbb{R}^d$ be an open set and $f:\Omega \to \mathbb{R}^m$. We say that $f$ has an approximate limit (respectively, approximate differential) at $x \in \Omega$ if there exists a function $g:\Omega \to \mathbb{R}^m$ continuous (resp. differentiable) at $x$ such that the set $\{f \neq g\}$ has density $0$ at $x$. In this case the approximate limit (resp. approximate differential) will be denoted by $\tilde{f}(x)$ (resp. $\nabla\tilde{f}(x)$). In this case, 
$ \tilde{\nabla} f(x) := \nabla g(x) $.

\end{definition}
\begin{proposition}[Approximate differentiability of $\barp$]\label{prop:diff barp}  
For $1<p<\infty$, the map $\barp$ (given by \eqref{eq:barp})
is approximately differentiable $\mu_1$-a.e.   
\end{proposition}
\begin{proof}
    By Corollary \ref{cor:barp_gleich_S1}, we know that $\barp=S_1$, $\LL^d$-a.e. on $\spt\mu_1$, where $S_1$ is given by \eqref{eq:mapS}. The claim follows directly by the approximate differentiability of optimal maps for the $p$-Wasserstein distance (see, for instance, Theorem 6.2.7 in \cite{AGS}).
\end{proof}

Now, we seek injectivity estimates on $\barp$ that  provide a lower bound on its (approximate) Jacobian. 
This is a direct consequence of the injectivity estimate of Lemma~\ref{lem:estimate}, which is an \textit{ad hoc} version of  \cite[Lemma 5.2]{BFR24-h} and \cite[Lemma 2.9]{BFR24-p}.

\color{black}

For every $S\subset\{1,\dots, N\}$, let us consider the definition of the sets $D_S$ and $D^1_S$, given, respectively, by \eqref{eq:defDS_intro} and \eqref{eq:D1S_intro}. Notice that if  
$\gamma_p = ({\rm Id}, T_2, \dots, T_N)_{\#} \mu_1 \in \Prob(\RR^{dN})$ is the optimal coupling for the problem \eqref{MMpbary}, then, 
\begin{align}\label{eq:D1S}
    D^1_{S}
        &:=
    \pi_1(D_S\cap\spt\gamma_p)
\\
        &=
    \Big\{
        x_1\in\spt\mu_1 \, : \, T_i(x_1)=\barp(x_1) \, \text{ if } i\in S, \, \text{ and } \, T_i(x_1) \neq \barp(x_1)  \, \text{ if }  i\not\in S
    \Big\}
        \, .
\end{align}
Recall that $\mathcal{F}_1:=\{S \, : \, 1\in S\}$. 

For any point $\xvec=(\xdots)\in\RR^{dN}$, recall the definition of $H_i(\xvec)$ and $\Lambda_i(\xvec)$ given in \eqref{eq:Hi_Lambda_i}. Then, given $T_2,\dots, T_N$ optimal for the multimarginal problem \eqref{MMpbary}, we define \begin{align}\label{eq:Hi1_Lambdai1}
H^1_i(x_1):=H_i\circ ({\rm Id},T_2,\dots,T_N) \quad \text{and} \quad 
    \Lambda_{i}^1(x_1)=\Lambda_i\circ ({\rm Id},T_2,\dots,T_N)(x_1).
\end{align} 
Notice that, consistently with \eqref{eq:positivity_Hi_Lambdai},
\begin{align}\label{eq:H_and_Lambda_positive}
  &H^1_i(x_1)\ge 0\, (\text{and thus} \ \Lambda^1_i(\xvec)\ge 0), \quad \text{for every} \ x_1\in\RRd,\\
    &H^1_i(x_1)=0 \ (\text{and thus} \ \Lambda^1_i(x_1) =0) \iff T_i(x_1)=\barp(x_1).  
\end{align}
 \begin{corollary}[Local regularity of $\barp$]\label{cor:injectivity estimate barp}
     Let $S\subset\{1,\dots,N\}$, $S \notin\mathcal{F}_1$. Then for every  $x_1\in D^1_S \cap \spt\mu_1$, there exists $r(x_1)>0$ such that, for every $y_1,\tilde y_1\in B_{r(x_1)}(x_1) \cap D_S^1 \cap \spt\mu_1$, 
     \begin{align}\label{eq:injectivity estimate barp}
  |\barp(y_1)-\barp(\tilde y_1)|\ge \frac 12  \frac {\left(\min_{i\notin S}\Lambda^1_i(x_1) \right)}{\max_{i \notin S}\|H^1_{i}(x_1)\|} |y_1-\tilde y_1|\, .
     \end{align} 
 \end{corollary}
\begin{proof}
    Let $S\notin\mathcal{F}_1$ and $x_1 \in \spt \mu_1 \cap D_S^1$. Then the point $\xvec:=(x_1,T_2(x_1),\dots, T_N(x_1))$ belongs to $D_S \cap \spt \gamma_p$. By optimality, $\spt\gamma_p \subset \RR^{dN}$ is $c_p$-monotone. Hence, one can apply  Lemma~\ref{lem:estimate} for every $\yvec \in B_{r(\xvec)}(\xvec)\cap D_S \cap \spt \gamma_p$, for some $r(\xvec)>0$. As $\barp(y_1) = \xp(\yvec)$ with $\yvec = ( y_1, T_2(y_1), \dots , T_N(y_1) ) \in D_S \cap \spt \gamma_p$ whenever $y_1 \in D_S^1 \cap \spt \mu_1$, the estimate \eqref{eq:injectivity estimate barp} follows from \eqref{eq:local_reg_cp_mon} by choosing  $r(x_1):= r(\xvec)$ (so that $B_{r(x_1)}(x_1)\subset \pi_1(B_{r(\xvec)}(\xvec))$).
\end{proof}


 \begin{corollary}
 \label{cor:jac_bound}
    For $\LL^d$-a.e. (and thus $\mu_1$-a.e.) $x_1\in  D^1_S \cap \spt\mu_1$, \begin{equation}\label{eq:bound_Jbarp}
        |\tilde J_{\barp}(x_1)|\ge \left(\frac 12  \frac {\left(\min_{i\notin S}\Lambda^1_i(x_1) \right)}{\max_{i \notin S}\|H^1_{i}(x_1)\|}\right)^d>0,
    \end{equation}
    where $\tilde J_{\barp}(x_1):=\det \tilde\nabla \barp(x_1)$ and $\tilde\nabla\barp (x_1)$ is the approximate gradient of $\barp$ at $x_1$.
\end{corollary}
\begin{proof}
As a consequence of Proposition \ref{prop:diff barp}, Corollary \ref{cor:injectivity estimate barp}, and Definition \ref{def:approx_diff}, the first inequality in \eqref{eq:bound_Jbarp} holds for each point $x_1$ such that $\widetilde\nabla\barp(x_1)$ exists and $\spt\mu_1\cap D_S^1$ has density $1$ at $x_1$. 
Indeed, take $x_1$ such that $\widetilde\nabla\barp(x_1)$ exists and $\spt\mu_1\cap D_S^1$ has density $1$ at $x_1$. By approximate differentiability we know that there exists a map $g: \mathbb{R}^d \to \mathbb{R}^d$ which is differentiable at $x_1$ such that the set 
$ E = \{ y : \barp(y) = g(y) \} $
has Lebesgue density $1$ at $x_1$. Therefore $g(y_1)=\barp(y_1),g(\tilde y_1)=\barp(\tilde y_1)$ and \eqref{eq:injectivity estimate barp} holds for every $y_1,\tilde y_1\in E\cap\spt\mu_1\cap D_S^1$, which  has density $1$ at $x_1$. As $g$ is differentiable at $x_1$, \eqref{eq:injectivity estimate barp} implies that $|\det \nabla g(x_1)|$ satisfies \eqref{eq:bound_Jbarp}. We conclude by the fact that $\widetilde\nabla{\barp}(x_1)=\nabla g(x_1)$.

The fact that the quantity is strictly greater than $0$ is a direct consequence of  \eqref{eq:H_and_Lambda_positive} and definition of $D_S^1$.
\end{proof}

We are finally ready to prove our  general $L^q$-estimate. 
\begin{proof}[Proof of Proposition~\ref{prop:general_Lq_estimate_intro}]
Recall that $\nup={\xp}_\sharp\gammap={\barp}_\sharp \mu_1$, where $\gammap$ is optimal for \eqref{MMpbary}.  Since $\barp$ is injective and $\LL^d$-a.e. differentiable on the union of $D_S$ with $S\notin\mathcal{F}_1$ (see Corollary~\ref{cor:barp_gleich_S1} and
Proposition~\ref{prop:diff barp}),
 using the properties of the push-forward (cf. Remark~\ref{rem:pushfor_inj}), we conclude that for every measurable function $\varphi : \RR^d \to [0,+\infty]$,
\begin{align}
    \int \varphi(x_1) f_1(x_1) \dd x_1
        &= 
    \int_{\bigcup\limits_{S\in\mathcal{F}_1}D_S^1}
        \varphi(x_1) f_1(x_1) \dd x_1
            + 
    \int_{\bigcup\limits_{S\notin\mathcal{F}_1}D_S^1}
        \varphi(x_1) f_1(x_1) \dd x_1
\\
        &= 
    \int_{\barp(\bigcup\limits_{S\in\mathcal{F}_1}D_S^1)}
        \varphi(\barp^{-1}(z)) g_p(z) \dd z   
            +
    \int_{\barp(\bigcup\limits_{S\notin\mathcal{F}_1}D_S^1)}{\varphi(\barp^{-1}(z))g_p(z)}\dd z
\\
        &= 
    \int_{\bigcup\limits_{S\in\mathcal{F}_1}D_S^1}
        \varphi(z) g_p(z) \dd z   
            +
    \int_{\bigcup\limits_{S\notin\mathcal{F}_1}D_S^1} \varphi(x_1) g_p(\barp(x_1))  \tilde J_{\barp}(x_1)  \dd x_1 
            \, , 
\end{align}
for $J_{\barp}(x_1):= | \det \nabla \barp(x_1) |$, where at last we used that
\begin{itemize}
    \item[i)]  $\barp(x_1) = x_1$ for $x_1 \in D_S^1$ with $S \in \mathcal{F}_1$; 
    \item[ii)] A standard change of variables formula on a set where $\barp$ is a.e. approximately differentiable. This can be done thanks to Lemma 5.5.3 in \cite{AGS}, with $\rho=f_1$ and $f=\barp$.
\end{itemize}
As this holds for every function $\varphi$, this shows that 
\begin{align}
    f_1 (x_1)
        = 
    \begin{cases}
        g_p(x_1)     & \text{if } x_1 \in \bigcup\limits_{S\in\mathcal{F}_1}D_S^1
            \, , 
    \\  
        g_p(\barp(x_1)) |\tilde J_{\barp}(x_1)|  & \text{if } x_1 \in 
            \bigcup\limits_{S\notin\mathcal{F}_1}D_S^1  
                \, , 
    \end{cases} 
\end{align}
for $\mu_1$-a.e. $x_1 \in \RR^d$. 
Therefore, we can compute the  $L^q$ norm of $g_p$ as 
\begin{align}
    \int |g_p(z)|^q\dd z
        &=
    \int_{\bigcup\limits_{S\in\mathcal{F}_1}D_S^1} 
        |g_p(z)|^q\dd z
            +
    \int_{\barp(\bigcup\limits_{S\notin\mathcal{F}_1}D_S^1)} 
        |g_p(z)|^q\dd z
\\
        &=
    \int_{\bigcup\limits_{S\in\mathcal{F}_1}D_S^1} 
        |f_1(z)|^q\dd z
            +
    \int_{\bigcup\limits_{S\notin\mathcal{F}_1}D_S^1}
        |g_p(\barp(x_1))|^q \tilde J_{\barp}(x_1)\dd x_1
\\
        &=
    \int_{\bigcup\limits_{S\in\mathcal{F}_1}D_S^1} 
    |f_1(z)|^q\dd z
            +
    \int_{\bigcup\limits_{S\notin\mathcal{F}_1}D_S^1}
        \frac{|f_1(x_1)|^q}{\tilde J_{\barp}(x_1)^{q-1}}\dd x_1
\\
        &\le 
    \int_{\bigcup\limits_{S\in\mathcal{F}_1}D_S^1}
        |f_1(z)|^q\dd z 
            +
    \frac12
        \int_{\bigcup\limits_{S\notin\mathcal{F}_1}D_S^1}
        \left(
        \frac
            {\max_{i \notin S}\|H^1_{i}(x_1)\|}
            {\min_{i\notin S}\Lambda^1_i(x_1)}\right)^{d(q-1)}
                |f_1(x_1)|^q\dd x_1
            \, , 
\end{align}
where the last inequality is given by Corollary \ref{cor:jac_bound}, which concludes the proof.
\end{proof}
\begin{remark}[Alternative proof of integrability in case of distant supports]\label{rem:alternative_proof_integrability}
The estimate provided in Proposition~\ref{prop:general_Lq_estimate_intro} holds under the simple assumption that $\mu_1 \ll \LL^d$, but, as mentioned, the right-hand side may be $+\infty$ even if $f\in L^q$. However, under assumptions \eqref{eq:compact_supports_intro} and \eqref{eq:ass on supports_intro}, Proposition~\ref{prop:general_Lq_estimate_intro} provides an alternative proof of $q$-integrability for $g_p$ \footnote{Notice that this is a slightly weaker result than Theorem~\ref{thm:integrability_distant_intro}, as the stronger assumption \eqref{eq:ass on supports_intro} is required for every $1<p<\infty$.}.
Indeed for every set of indexes $S$, 
\begin{align}
\label{eq:bounds_distant1}
    \max_{x_1 \in \spt\mu_1} \max_{i \notin S}\|H^1_{i}(x_1)\|
        \leq \begin{cases}
         p(p-1) m^{p-2} 
            \,  , \quad &\text{if } 1<p<2 \, , \\
         p(p-1) M^{p-2} 
            \,  , \quad &\text{if } p\ge 2 \, ,
        \end{cases}
\end{align}
and 
\begin{align}\label{eq:bounds_distant2}
    \min_{x_1 \in \spt\mu_1} \min_{i\notin S}\Lambda^1_i(x_1)
        \geq \begin{cases}
        \big( \min_{ i \notin S} \lambda_i \big)p\frac{M^{2(p-2)}}{(p-1)m^{p-2}} = \frac{m^{2(2-p)}}{(p-1)M^{2-p}} 
        \,, \quad &\text{if } 1<p<2 \, ,   \\
        \big( \min_{ i \notin S} \lambda_i \big)
   p \frac{m^{2(p-2)}}{(p-1)M^{p-2}} 
        \,, \quad &\text{if } p\ge 2 \,  . 
        \end{cases}
\end{align} 
Thanks to the above estimate, we have
\begin{align}
    \| g_p \|_{L^q(\RR^d)} 
        \leq \begin{cases}
        \Big(
        1 \vee 
        \frac
            {(p-1)^2 M^{2(2-p)}}
            {(\min_i \lambda_i) m^{2(2-p)}}
     \Big)^{d q'}
        \| f_1 \|_{L^q(\RR^d)}
        \, ,  \quad &\text{if } 1<p<2  \, \\
         \Big(
        1 \vee 
        \frac
            {(p-1)^2 M^{2(p-2)}}
            {(\min_i \lambda_i) m^{2(p-2)}}
     \Big)^{d q'}
        \| f_1 \|_{L^q(\RR^d)}
        \, ,  \quad &\text{if } p\ge 2 \, .
        \end{cases} 
\end{align}
which provides integrability in a similar spirit as Theorem~\ref{thm:integrability_distant_intro}.

\end{remark}

\section{Optimal maps for $p$-Wasserstein distance and barycenters via affine transformations}
\label{sec:examples}
In this final section of our work,  we first discuss optimal maps for the $p$-Wasserstein distance, between two measures, the second given by the push-forward with respect to an affine transformation of the first. Moreover, accordingly, we explicitly compute barycenters in the $p$-Wasserstein space for a special class of affine transformations.

\subsection{Optimality of affine transformations for the $p$-Wasserstein distance}
The first step of this chapter is to discuss the optimality of affine transformations for the $p$-Wasserstein distance. Recall that for $p=2$, and any $\mu \ll \LL^d$,  any $T=\nabla u$ with $u$ convex is optimal between $\mu$ and its image $T_{\#} \mu$. In particular, when looking at maps of the form $T = Ax + b$ for some $A \in \mathbb{M}(d)$ ($d\times d$ matrices with real entries) and $b \in \RR^d$ is indeed optimal whenever $A$  is symmetric and $A \geq 0$, as it is the gradient of the convex function $u(x) = \frac12 \langle Ax, x \rangle + b x$. 

In this chapter, we are going to discuss optimality for $W_p$ for general $1<p<\infty$. We show in particular that $A \geq 0$ does not suffice to guarantee optimality of $T$, and much more restrictive conditions on $A$ must be imposed.

It is well known (see for instance \cite[Theorem 1.2]{GMC96} or \cite[Theorem 6.2.4]{AGS}) that, given $\mu \in \calP_p(\RR^d)$ so that $\mu \ll \LL^d$ and any other $\nu \in \calP_p(\RR^d)$, there exists an  optimal map $T$ from $\mu$ to $\nu$,
so that
\begin{align}
\label{eq:optimality_p}
    T(x) = x - 
    \frac
        {\nabla \varphi(x)}
        {|\nabla \varphi(x)|^{\alpha_p}}
            \qquad 
        \mu\text{-a.e. } x \in \RR^d 
            \, ,
\end{align}
where $\alpha_p = \frac{p-2}{p-1}$ and $\varphi$ is a $p$-concave map, which means that the double $p$-transform coincides with the function itself, i.e. $(\varphi^{p})^{p} = \varphi$, where for a proper function $\varphi$ (somewhere finite) the $p$-transform is defined by 
\begin{align}
    \varphi^{p}(y)
        = 
    \inf_{x \in \RR^d} 
    \left\{
       \frac 1p |x-y|^p - \varphi(x) 
    \right\}
            \in [-\infty, +\infty) 
            \, , \qquad 
        \forall y \in \RR^d 
            \, .
\end{align}
Vice versa, for a given measure $\mu \in \calP_p(\RR^d)$, if one can find a $p$-concave function $\varphi$ so that \eqref{eq:optimality_p} holds true, then $T$ is optimal for $W_p(\mu, T_{\#} \mu)$ (as the duality gap between $T$ and $\varphi$ vanishes).

Note indeed that for $p=2$, maps are of the form $T(x) = x - \nabla \varphi(x) = \nabla \big( \frac12 |x|^2 - \varphi(x) \big)$ where now the $2$-concavity of $\varphi$ simply implies that $x \mapsto \frac12|x|^2 - \varphi(x)$ is a convex function.

The picture significantly changes when $p\neq2$, as our next theorem shows. In what follows, for $k\in \NN$, we denote by $\text{Id}_k \in \mathbb{M}(k)$ the corresponding $k\times k$ identity matrix. For a matrix $A$, we denote by $\sigma(A)$ the spectrum of $A$.
\begin{theorem}[Optimality of affine maps for $p\neq 2$]
\label{thm:optimality_affine_Wp}
    Let $A \in \mathbb{M}(d)$ be symmetric and fix $b \in \RR^d$ such that  $b \in {\rm range}({\rm Id} - A)$ whenever $A \neq {\rm Id}$. Consider $T:\RR^d \to \RR^d$ given by $T(x) = Ax + b$, for $x \in \RR^d$. Then $T$ is an optimal map between any measure $\mu \in \calP_p(\RR^d)$ and its image $T_{\#} \mu$ if and only if $\sigma(A)\subset\{1,\zeta\}$, for some  $\zeta\geq 0$.
\end{theorem}
In particular, it means that, up to orthogonal transformations, the matrix $A$ is of the form 
\begin{align}
\label{eq:representation_A_blocks}
    A
        =
    \left[ 
    \begin{array}{c|c} 
      \text{Id}_k & \mathbf{0} \\ 
      \hline 
      \mathbf{0} & \zeta \text{Id}_{d-k} 
    \end{array} 
    \right]
\end{align}
for some $ k \in \{ 0, \dots , d \}$ and $\zeta \in [0, +\infty)$, 
or in other words, that $A = {\rm Id} + (\zeta-1) P$, where $P$ is an orthogonal projection.  Significant cases include
\begin{enumerate}
    \item $A = \text{Id}_d$. \ 
    In this case $T(x) = x + b$ is simply a translation, and it is always optimal between any measure $\mu \in \calP_p$ and its translation for every $b \in \RR^d$. 
    \item $A=0$ (i.e. $k=0$ and $\zeta = 0$). \ 
    In this case $T(x) = b$ is constant, and clearly optimal between $\mu$ and $\delta_{b}$, for every $\mu \in \calP_p(\RR^d)$, every $b \in \RR^d$.
    \item $A = \zeta \text{Id}_d$ for $\zeta \in (0,+\infty)$. \ 
    Note that in this case $\mathrm{range}({\rm Id}-A) = \RR^d$, hence $T(x) = \zeta x + b$ is optimal for every $b \in \RR^d$, and consists of a translation and a dilation.
    \item Intermediate cases: $ k \notin \{ 0 , d \}$. \ 
    In this case, the map $T$ is, up to a translation, leaving invariant a subspace of dimension $k$ and (nonnegatively) dilating its orthogonal subspace of dimension $d-k$. 
\end{enumerate}

Observe that the optimality of the translations (case (1)) can also be proved directly as a consequence of Jensen's inequality: indeed, let $T(x) = x + b$ and $\mu \in \calP_p(\RR^d)$. Then for every admissible coupling $\pi \in \calP(\RR^d \times \RR^d)$ between $\mu$ and $\nu = T_{\#}\mu$, by Jensen's inequality we have that
\begin{align}
    \int_{\RR^d \times \RR^d} 
|x-y|^p  \dd \pi(x,y) 
            &\geq 
        \left|
            \int (x-y) \dd \pi(x,y) 
        \right|^p
             =
        \left|
            \int x \dd \mu(x)
                - 
            \int y \dd \nu(y)
        \right|^p
\\  
            &=
        \left|
            \int x \dd \mu(x) 
                - 
            \int (x+b) \dd \mu(x) 
        \right|^p 
            =
        |b|^p 
\\
            &= 
        \int_{\RR^d} 
            |x-T(x)|^p \dd \mu(x) 
            \,  ,
\end{align}
which shows the claimed optimality. 

\begin{proof}
We start by showing that if $T(x) = Ax + b$ is optimal between $\mu \in \calP_p(\RR^d)$ and its image, then under the standing assumptions on $b \in \RR^d$, $A$ must necessarily satisfy $A \geq 0$ and its spectrum has at most one eigenvalue different from $1$.

First of all, as translations are optimal (see discussion above), without loss of generality we can assume $A \neq {\rm Id}$ and thus $b \in {\rm range}({\rm Id}-A)$.

In this setting, we claim we can in fact also assume that $b=0$. Indeed, let $c \in\RR^d$ so that $b = c - Ac$. Then the map $\tilde T :=\tau_c^{-1} \circ T \circ \tau_c$ is linear, as 
\begin{align}
    \tilde T(x) 
        = 
    A(x+c) +b - c 
        =
    Ax + Ac + b - c
        = 
    Ax 
        \, ,
\end{align}
for every $x \in \RR^d$. Furthermore, one can see that $T$ is optimal between $\mu$ and $\nu$ if and only if $\tau_c^{-1} \circ T \circ \tau_c$ is optimal between $(\tau_c^{-1})_{\#}\mu$ and $(\tau_c^{-1})_{\#}\nu$. Indeed, for every $S \in $ Adm$(\mu,\nu)$\footnote{A map $S$ belongs to Adm$(\mu,\nu)$ if $S_{\#} \mu = \nu$.}, then $\tilde S := \tau_c^{-1} \circ S \circ \tau_c \in $ Adm$((\tau_c^{-1})_{\#}\mu,(\tau_c^{-1})_{\#}\nu)$ and 
\begin{align}
    \| \tilde S - {\rm Id} \|_{L^p((\tau_c^{-1})_{\#}\mu)}
        =
    \| \tau_c^{-1} T - \tau_c^{-1} \|_{L^p(\mu)}
        =
    \| T - {\rm Id} \|_{L^2(\mu)}
        \, ,
\end{align}
which shows the sought equivalence.

Therefore, one can work directly with $\tilde T$, i.e. assuming $b=0$. 
As $T$ is optimal from any measure onto its image, by picking $\mu$ fully supported and absolutely continuous one, we conclude from  \eqref{eq:optimality_p} that
\begin{align}
    \label{eq:optimality_p_2}
        \frac
        {\nabla \varphi(x)}
        {|\nabla \varphi(x)|^{\alpha_p}}
            = 
        \unA  x
        \, , \quad 
        \LL^d\text{-a.e. } x \in \RR^d
            \, , \qquad
        \text{where} \quad 
        \unA = \text{Id}_d - A
            \, .
    \end{align}
    for some $c_p$-concave function $\varphi$.
    We observe that if $R \in \text{SO}(d)$ and $\varphi$ satisfies  \eqref{eq:optimality_p_2}, then the function $\phi$ defined by $\phi(x) := \varphi(Rx)$ solves 
    \begin{align}
    \label{eq:invariance_rotations}
        \frac
            {\nabla \phi(x)}
            {|\nabla \phi(x)|^{\alpha_p}}
                = 
        \frac
            {R^T\nabla \varphi(Rx)}
            {|\nabla \varphi(Rx)|^{\alpha_p}}          
                =
            \big( R^T \unA R \big)  x    
                \, , \qquad 
        \LL^d\text{-a.e. }  x \in \RR^d 
            \, ,
    \end{align}
    hence, the problem is rephrased in terms of a new matrix $R^T \unA R$, which has the same spectrum as $\unA$.
    Therefore, without loss of generality we can assume $\unA$ diagonal. Secondly, we note that \eqref{eq:optimality_p_2} implies 
    \begin{align}
    \label{eq:range_prop_grad}
        | \nabla \varphi(x) |^{1-\alpha_p} 
            = 
        | \unA x |
            \quad 
                \Rightarrow
            \quad 
        \nabla \varphi(x) 
            =
        | \unA x |^{\frac{\alpha_p}{1-\alpha_p}} 
            \unA x 
            = 
        | \unA x |^{p-2}
            \unA x 
                \, . 
    \end{align}
    This shows that  when $p>2$, $\varphi$ is almost everywhere equal to a function in $C^2(\RR^d)$, while for $1<p<2$, to a function in $C^1(\RR^d)\cap C^2(\RR^d\setminus \ker(A))$. Therefore, from now on we assume that $\varphi \in C^2(\RR^d)$ for $p \geq 2$, while for $1<p<2$, that  $\varphi \in C^1(\RR^d)\cap C^2(\RR^d\setminus \ker(\unA))$. Moreover, \eqref{eq:optimality_p} (or equivalently \eqref{eq:range_prop_grad}) holds true for every $x \in \RR^d$ for $p \geq 2$ and for every $x \in \RR^d \setminus \ker(\unA)$ for $1 < p <2$.
    Furthermore, this shows that
    \begin{align}
        \label{eq:gradient_const_kernel}
        \nabla \varphi(x) \in \text{range}(\unA) 
            \qquad 
        \forall x \in \RR^d
            \, ,
    \end{align}
  and $\varphi$ is constant in the directions of $\ker(\unA)$. Since $\unA$ is diagonal, either $\ker(\unA) = \{ 0 \}$ or $\ker(\unA) = \text{Span} \{ e_1, \dots , e_k \}$, for some $k \in \{ 1, \dots, d \}$. In the second case, property \eqref{eq:gradient_const_kernel} shows that $\partial_j \varphi(x) = 0$ for every $j \leq k$, hence $\varphi$ does not depend on the first $k$ variables. In other words, there exists $\hat \varphi: \RR^{d-k} \to \RR$ such that 
    \begin{align}
    \label{eq:independent_coordinates}
        \varphi(x) = \hat \varphi(\hat x)
            \, , \qquad
        \forall \hat x \in \RR^{d-k} 
            \, , \,  
        x = (x_1, \dots, x_k, \hat x) \in \RR^d  
            \, .
    \end{align}
    Therefore, either $k=d$, i.e. $\unA=0$ and $A=\mathrm{Id}$ (case (1)), or $k<d$. In the latter case, if we denote by  $\hat A \in \mathbb{M}(d-k)$ the (now invertible) linear map so that $\unA x = (0,\hat \unA \hat x)$, the considerations in \eqref{eq:invariance_rotations} ensure that $\hat \varphi$ solves 
    \begin{align}
        \frac
        {\nabla \hat \varphi(\hat x)}
        {|\nabla \hat \varphi(\hat x)|^{\alpha_p}}
            = 
        \hat \unA  \hat x
            \, , \qquad
       \forall \hat x \in \RR^{d-k} 
            \, .
    \end{align}
    Note that $A$ has at most one eigenvalue different from $1$ if and only if $\unA$ has at most one eigenvalue different from zero, which is then equivalent to the fact that $\hat \unA = \lambda \text{Id}_{d-k}$ for some $\lambda \in \RR\setminus \{ 0 \}$. Moreover, for $y=(y_1 , \dots, y_k, \hat y) \in \RR^d$, $\hat y \in \RR^{d-k}$,
    \begin{align}
    \label{eq:ctrasnf_restriction}
        \varphi^{p}(y) 
            = 
        \inf_{x \in \RR^d}  
        \left\{
            \frac 1p |x-y|^p - \varphi(x) 
        \right\}
            =
        \inf_{\hat x \in \RR^{d-k} }
        \left\{
           \frac 1p |\hat x- \hat y|^p - \hat \varphi(\hat x) 
        \right\}
            =
        \hat \varphi^{p}(\hat y) 
            \, .
    \end{align}
    It is then clear that $\varphi$ is $p$-concave (in $\RR^d$) if and only if $\hat \varphi$ is $p$-concave (in $\RR^{d-k}$).
    This means that without loss of generality, we can assume that $\ker(\unA) = \{0\}$ (hence $\hat \varphi = \varphi$), or else we simply work with $\hat \unA$. With this restriction, when $1<p<2$, $\varphi\in C^2(\RR^d\setminus\{0\})$. By taking a second derivative in \eqref{eq:optimality_p_2}, a simple chain rule shows that
    \begin{align}
    \label{eq:optimality_p_3}
        \frac{1}{|\nabla \varphi(x)|^{\alpha_p}}
        \bigg(
            \nabla^2 \varphi(x) 
                - 
            \frac1{\alpha_p} 
            \frac
                {\nabla \varphi(x)}
                {|\nabla \varphi(x)|}
                    \otimes 
            \Big(
                \nabla^2 \varphi(x)
                \frac
                    {\nabla \varphi(x)}
                    {|\nabla \varphi(x)|}
            \Big)
        \bigg)
            = 
        \unA 
            \, , \qquad 
        \forall x \in \RR^d  \setminus {\{ 0 \} } 
            \, .
    \end{align}
    As $\unA$ is symmetric, and the Hessian of $\varphi$ as well, this implies that 
    \begin{align}
        \nabla \varphi(x)
                \otimes 
        \Big(
            \nabla^2 \varphi(x)
                \nabla \varphi(x)
        \Big)
    \end{align}
    is symmetric as well. But this rank-one matrix is symmetric if and only if the two factors of the tensor products are parallel to each other, namely there exists a map $\lambda: \RR^d \to \RR$ such that 
    \begin{align}
        \nabla^2 \varphi(x) \nabla \varphi(x) 
            = 
        \lambda(x)
            \nabla \varphi(x) 
                \, , \qquad 
            \forall x \in \RR^d  \setminus {\{ 0 \} } 
                \, .
    \end{align}
    In other words, the gradient of $\varphi$ is everywhere an eigenvector of the Hessian of $\varphi$. By plugging in this new information in \eqref{eq:optimality_p_3}, we conclude that 
    \begin{align}
         \frac{1}{|\nabla \varphi(x)|^{\alpha_p}}
        \bigg(
            \nabla^2 \varphi(x) 
                - 
            \frac1{\alpha_p} 
            \lambda(x)
            \frac
                {\nabla \varphi(x)}
                {|\nabla \varphi(x)|}
                    \otimes 
                \frac
                    {\nabla \varphi(x)}
                    {|\nabla \varphi(x)|}
            \Big)
        \bigg)
            = 
        \unA 
            \, , \qquad 
        \forall x \in \RR^d \setminus {\{ 0 \} } 
            \, .
    \end{align}
    In particular, for every $x\in\RR^d \setminus {\{ 0 \} } $, we infer that
    \begin{align}
        \unA \nabla \varphi(x) 
            = 
        \underline{\lambda}(x) \nabla \varphi(x) 
            \, , \qquad 
        \underline{\lambda}(x) 
            =
        \frac1{|\nabla\varphi(x)|^{\alpha_p}}
        \lambda(x) 
        \Big( 
            1 - \frac1{\alpha_p}
        \Big)
            \, ,
    \end{align}
    which shows that for every $x \in \RR^d \setminus {\{ 0 \} } $, $\nabla \varphi(x)$ is an eigenvector of the (constant) matrix $\unA$. Our claim is that in fact $\underline{\lambda}(x) \equiv \underline{\lambda}$ is constant. 

    To see this, observe that \eqref{eq:range_prop_grad} and $\ker(\unA) = \{ 0 \}$ ensure $\nabla \varphi(x) =0$ only at  $x =0$. In particular, apart from the value $\lambda(0)$ which is not uniquely determined, the remaining one are uniquely determined by 
    \begin{align}
        \unlambda(x) 
            =
        \frac{\langle \unA \nabla \varphi(x), \nabla \varphi(x) \rangle}{|\nabla \varphi(x)|^2}
            \, , \qquad 
        \text{for} \quad 
            x \in \RR^d  \setminus {\{ 0 \} } 
                \, ,
    \end{align}
    and therefore it is continuous on $\RR^d \setminus \{0\}$ (it is in fact $C^1$, as $\varphi \in C^2$). Assume now there exists $x, y \in \RR^d \setminus \{0\}$ such that $\unlambda(x) \neq \unlambda(y)$. As $\RR^d \setminus \{0\}$ is connected for $d \geq 2$, pick any continuous curve $\gamma^{xy}: [0,1] \to \RR^d \setminus \{0\}$ from $x=\gamma^{xy}(0)$ to $y=\gamma^{xy}(1)$. Then the function 
    $
        \unlambda^{xy}: [0,1] \to \RR  
    $
    defined by 
    $
        \unlambda^{xy} = \unlambda \circ \gamma^{xy} 
    $
    is a continuous curve satisfying $\unlambda^{xy}(0) = \unlambda(x) \neq \unlambda(y) = \unlambda^{xy}(1)$. Denote by $t_0 \in [0,1)$ the largest time so that $\unlambda^{xy}(t) = \unlambda^{xy}(0)$, or equivalently
    \begin{align}
       t_0 := \max  
        \big\{t \in [0,1) \: : \: \unlambda^{xy}(t) = \unlambda^{xy}(0) \big\}
                \, .
    \end{align}  
    In particular, by continuity and construction we have that 
    \begin{align}
    \label{eq:nonconstant_lambda}
        \unlambda^{xy}(t) \neq \unlambda^{xy}(t_0) 
            \, , \qquad 
        \forall t > t_0 
            \, .
    \end{align}
    Now, the spectral theorem for symmetric matrices ensures that eigenvectors corresponding to different eigenvalues are necessarily orthogonal. As $\nabla\varphi(x)$ is an eigenvector of $\unA$ with eigenvalue $\unlambda(x)$, \eqref{eq:nonconstant_lambda} would imply for $t > t_0$
    \begin{align}
        0
            =
        \langle 
            \nabla \varphi(\gamma^{xy}(t_0) )
                , 
            \nabla \varphi(\gamma^{xy}(t))
        \rangle 
            \xrightarrow[t \to t_0]{} 
        \langle 
            \nabla \varphi(\gamma^{xy}(t_0))
                , 
            \nabla \varphi(\gamma^{xy}(t_0))
        \rangle
            = 
        \big| 
            \nabla \varphi(\gamma^{xy}(t_0))
        \big|^2 
            \neq 0 
    \end{align} 
    as $\gamma^{xy}(t_0) \in \RR^d \setminus \{0\}$, which is clearly a contradiction. 

    We thus showed that 
    \begin{align}
       \unA \nabla \varphi(x) 
            = 
        \underline{\lambda} \nabla \varphi(x) 
            \, , \quad 
        \forall x \in \RR^d 
            \, , \qquad
        \text{for some} \:  
            \unlambda \in \RR 
                \, .
    \end{align}

    In order to conclude the proof that $\unA = \unlambda \text{Id}$, we shall prove that the span of $\{ \nabla \varphi(x) \, : \, x \in \RR^d\}$ is the whole space $\RR^d$. Indeed, pick a $v \in \RR^d$ such that $\langle v , \nabla \varphi(x) \rangle = 0$ for every $x \in \RR^d$. Then from \eqref{eq:range_prop_grad} we infer that 
    \begin{align}
        0 
            = 
        \langle 
            \nabla \varphi(x) 
                , 
            v
        \rangle 
            = 
        |\unA x |^{p-2} 
        \langle 
            \unA x
                , 
            v 
        \rangle 
            = 
        |\unA x |^{p-2} 
        \langle 
            x
                , 
            \unA v
        \rangle
            \, , \quad 
        \forall x \in \RR^d
            \, , 
    \end{align}
   as $\unA$ is symmetric. As $\ker(\unA) = \{ 0 \}$, we infer that $\unA v =0$, and therefore that $v=0$. We show now that $A\ge 0$ or, equivalently, that $\unlambda\in (-\infty,1]$.  By \eqref{eq:range_prop_grad} we infer that 
    $
    \hat\varphi(\hat x) 
            = \frac1p \lambda|\lambda|^{p-2}|\hat x|^p
    $, hence, if $\lambda >1$, computing its $p$-transform,
    \begin{align}
        \hat \varphi^{p}(\hat y) 
            =
        \frac1p
        \inf_{\hat x \in \RR^{d-k} }
            \Big\{
                |\hat x - \hat y|^p - \lambda^{p-1}|\hat x|^p
            \Big\}=-\infty
                \, , \text{for all} \, \hat y,
    \end{align}
    which shows that there can not exist any $p$-concave potential satisfying \eqref{eq:optimality_p_2}. Recall indeed that from \eqref{eq:ctrasnf_restriction} we have that $\varphi$ is $p$-concave if and only if $\hat \varphi$ is. 
    This concludes the proof that $A$ must necessarily satisfy the claimed conditions. 

    We are left to show that this is in fact sufficient, namely that if $A\geq 0$ and its spectrum has at most one eigenvalue different from $1$, then $T(x) = Ax + b$ is optimal between any $\mu \in \calP_p(\RR^d)$ and its image $T_{\#}\mu$. To do that, we are going to explicitly construct a $p$-concave potential $\varphi$ satisfying \eqref{eq:optimality_p_2}. Arguing as at the beginning of the proof, by our assumptions on $b\in\RR^d$ we can without loss of generality assume $b=0$ and directly work in the system of coordinates which makes $A$ diagonal, namely we assume that $A$ has the form \eqref{eq:representation_A_blocks} for some $k \in \{0, \dots , d \}$ and $\zeta \in [0, +\infty)$. 
    By duality (cfr. discussion underneath \eqref{eq:optimality_p}), to prove optimality it suffices to construct a $p$-concave 
    function $\hat \varphi:\RR^{d-k} \to \RR$ such that 
    \begin{align}
        \frac
        {\nabla \hat \varphi(\hat x)}
        {|\nabla \hat \varphi(\hat x)|^{\alpha_p}}
            = 
       (1-\zeta) \hat x 
            \, , \qquad
       \forall \hat x \in \RR^{d-k} 
            \, ,
    \end{align}
    and subsequently define $\varphi$ as in \eqref{eq:independent_coordinates}. Recall indeed that from \eqref{eq:ctrasnf_restriction} we have that $\varphi$ is $p$-concave if and only if $\hat \varphi$ is. Set $\lambda := (1-\zeta) \in (-\infty , 1]$, then \eqref{eq:range_prop_grad} reads  $\nabla \hat \varphi (\hat x) =  \lambda |\lambda|^{p-2} |\hat x|^{p-2} \hat x$, which (up to additive constants) implies that 
    \begin{align}
        \hat\varphi(\hat x) 
            = 
        \frac1p 
            \lambda |\lambda|^{p-2} 
            |\hat x|^p
                \, , \qquad 
            \forall \hat x \in \RR^{d-k}
                \, .
    \end{align}
    The final  step is to prove that for any $\lambda \in (-\infty,1]$, the map $\hat \varphi$ is indeed $p$-concave. 

        \smallskip 
    \noindent
    \underline{Case 1}: $\lambda \leq 0$. \ 
    In this case, we have that 
    $
    \hat\varphi(\hat x) 
            = 
        -\frac1p 
            |\lambda|^{p-1}
            |\hat x|^p
    $, hence computing its $p$-transform 
    \begin{align}
        \hat \varphi^{p}(\hat y) 
            =
        \frac1p
        \inf_{\hat x \in \RR^{d-k} }
            \Big\{
                |\hat x - \hat y|^p + |\lambda|^{p-1}|\hat x|^p
            \Big\}
                \, , 
    \end{align}
    which is nothing but a (suitably weighted) Euclidean $p$-barycenter between $0$ and $\hat y$, and therefore explicitly computable. Indeed it is the unique solution of the associated Euler-Lagrange equation 
    \begin{align}
        |\hat x - \hat y|^{p-2}(\hat x - \hat y) + |\lambda|^{p-1}|\hat x|^{p-2}\hat x=0,
    \end{align}
    whose unique solution is given by $ \hat x =\frac{1}{1+|\lambda|}\hat y$.
    This implies that 
    \begin{align}
        \hat \varphi^{p}(\hat y) 
            = 
        \frac1p
            g_p(\lambda)
            |\hat y|^p
                \, , \qquad 
        g_p(\lambda) :=
            \Big(
                \frac{|\lambda|}{1+|\lambda|}
            \Big)^{p-1}
                \in [0,1) 
                    \, .
    \end{align}
 
    The second $p$-transform is then given by 
    \begin{align}
    (\hat \varphi^{p})^{p}(\hat x) 
            =
        \frac1p
        \inf_{\hat y \in \RR^{d-k} }
            \Big\{
                |\hat x - \hat y|^p - g_p(\lambda)|\hat y|^p
            \Big\}
                \, .
    \end{align}
    As $g_p(\lambda) \in [0,1)$, it follows that, for every $\hat x \in \RR^{d-k}$, the function $\hat y \mapsto |\hat x - \hat y|^p - g_p(\lambda)|\hat y|^p$ is coercive and smooth, 
    and therefore admits a minimizer.  The corresponding Euler--Lagrange equation reads as
    \begin{align}
        |\hat y - \hat x|^{p-2}(\hat y - \hat x) = g_p(\lambda) |\hat y|^{p-2} \hat y
            \, , 
    \end{align}
    which in particular implies 
    \begin{align}
        |\hat y - \hat x|^{p-1} = g_p(\lambda) |\hat y|^{p-1}
            \quad \Longleftrightarrow \quad 
        |\hat y - \hat x| = \frac{|\lambda|}{1+|\lambda|} |\hat y| 
            \, .
    \end{align}
    Substituting this back into the previous optimality conditions, we explicitly find
    \begin{align}
        \hat y - \hat x 
            = 
        g_p(\lambda)^{1-\alpha_p} \hat y
            = 
        \frac{|\lambda|}{1+|\lambda|} \hat y
            \quad \Longleftrightarrow \quad 
        \hat y 
            =
        \big( 1+|\lambda| \big) 
            \hat x
                \, , 
    \end{align}
    which is the sought unique minimiser. By computing the corresponding minimal value, we conclude that 
    \begin{align}
        (\hat \varphi^{p})^{p}(\hat x)
            &=
        \frac1p 
        \Big(
            |\lambda|^p - g_p(\lambda) (1+|\lambda|)^p 
        \Big)
            |\hat x|^p 
            =
        \frac1p 
        \Big(
            |\lambda|^p - |\lambda|^{p-1} (1+|\lambda|)
        \Big)
            |\hat x|^p 
\\  
            &=
        -\frac1p |\lambda|^{p-1} |\hat x|^p 
            = 
        \hat \varphi (\hat x) 
            \, ,
    \end{align}
    for every $\hat x \in \RR^{d-k}$, which proves the claimed $p$-concavity.

        \smallskip 
    \noindent
    \underline{Case 2}: $\lambda \in (0,1)$. \ 
    In this case, we have that $\hat \varphi(\hat x) = \frac1p \lambda^{p-1}|\hat x|^p$. Alternatively, we can write 
    \begin{align}
        \hat \varphi (\hat x) 
            = 
        \frac1p 
            g_p(\bar \lambda)  |\hat x|^p 
                \, , \qquad 
        \text{where} 
            \: \: 
        \bar\lambda \leq 0  
            \: \: 
        \text{is such that} 
            \: \:
        g_p(\bar \lambda) = \lambda^{p-1} 
            \: \: \Longleftrightarrow \: \: 
        \bar \lambda 
            = 
        \frac\lambda{\lambda -1}
            \, .
    \end{align}
    But in the proof of Case 1, we have shown that $\hat \varphi = \hat \psi^{p}$, for $\hat \psi(\hat y) = - \frac1p |\bar \lambda|^{p-1} |\hat y|^p$, which is then $p$-concave. Therefore $\hat \varphi$ is also $p$-concave.

    \smallskip 
    \noindent
    \underline{Case 3}: $\lambda =1$. \ 
    This corresponds to the degenerate case in which the matrix $A$ has kernel, and thus $T$ is noninjective. And clearly
    \begin{align}
        \hat \varphi^{p}(\hat y) 
            = 
        \frac1p 
        \inf_{\hat x \in \RR^{d-k} }
            \Big\{
                |\hat y - \hat x|^p - |\hat x|^p
            \Big\}
            = 
        \begin{cases}
            0 
                &\text{if } \hat y =0 
        \\
            -\infty 
                &\text{otherwise} .
        \end{cases} 
    \end{align}
    Taking a second $p$-transform, we find
    \begin{align}
        (\hat \varphi^{p})^{p}(\hat x)    
            = 
        \inf_{\hat y \in \RR^{d-k} }
            \Big\{
                \frac1p|\hat x - \hat y|^p - \hat \varphi^{p}(\hat y) 
            \Big\}
            =
        \frac1p |\hat x|^p 
            = 
        \varphi (\hat x)
            \, , 
    \end{align}
    for every $\hat x \in \RR^{d-k}$, which concludes the proof of the sought $p$-concavity.
\end{proof}

\subsection{Barycenters of measures under affine transformations}
Using the characterization of optimal affine maps for $W_p$ provided in Theorem~\ref{thm:optimality_affine_Wp}, we now describe the $p$-Wasserstein barycenters between measures which are obtained via suitable affine transformations of a common measure $\mu\in \calP_p(\RR^d)$. 
    \begin{proposition}[Barycenters under affine transformations]
    \label{prop:bary_affine}
        Let $\mu\in \calP_p(\RR^d)$ be a reference probability measure. Let us consider probability measures of the form $\mu_i = \big( A_i \cdot + v_i \big)_{\#} \mu$, in two cases: either $A_i = \emph{Id}$  or $v_i = v \in \RR^d$ for every $i = 1 , \dots, N$, and in the second case,   
        \begin{itemize}
            \item $A_1$ is invertible. 
            \item $A_i$ are symmetric, positive semidefinite, $d\times d$ matrices. 
            \item The spectrum of the matrices satisfies $\sigma(A_i) \subset \{ 1, \zeta_i \}$ for some $\zeta_i \in [0,+\infty)$. 
            \item The matrices commute, i.e. $[A_i,A_j]=0$ for every $i,j=1, \dots, N$, and the eigenspaces associated to the eigenvalue $1$ of every $A_j \neq \emph{Id}$ coincide. 
        \end{itemize}
        Then the $p$-Wasserstein barycenter between $\mu_1, \dots, \mu_N$ is given by 
        \begin{align}
        \label{eq:statement_bary_affine}
            \overline\nu_p 
                = 
            \big( 
                \overline{A}\cdot + \overline v
            \big)_{\#}\mu
                \, , \qquad 
            \overline{A}:= \xp(A_1, \dots, A_N) 
                \quad \text{and} \quad 
            \overline{v} := \xp(v_1, \dots, v_N)
                \, , 
        \end{align}
        where with the barycenters of matrices we intend 
        \begin{align}
        \label{eq:bary_matrix}
            \xp(A_1, \dots, A_N) 
                :=
            \argmin_{B}
            \sum_{i=1}^N 
                \lambda_i 
                \big| A_i - B \big|^p 
                    \, , \qquad 
                |A| := \emph{Tr}(A^T A)^{\frac12}
                    = 
                \Big( \sum_{i,j=1}^N  A_{ij}^2 
                \Big)^{\frac12}
                    \, ,
        \end{align}
        for every $A$. 

    \end{proposition}

    \begin{proof}

        Thanks to our assumptions, without loss of generality, we can assume the matrices $A_i$ to be already in diagonal form. Indeed, 
        for $D_i := R A_i R^T$, we have that, for every matrix $B$, if we set $B':= R B R^T$ , then 
        \begin{align}
            \sum_{i=1}^N 
                \lambda_i 
                \big| R^T D_i R - B \big|^p 
                =
            \sum_{i=1}^N 
                \lambda_i 
                \big| R^T (D_i  - B') R\big|^p
                =
            \sum_{i=1}^N 
                \lambda_i 
                \big| D_i - B' \big|^p 
                    \, ,
        \end{align}
        which in particular ensures that $\xp(A_1, \dots, A_N) = R^T \, \xp(D_1, \dots, D_N)R$. Therefore we can assume $A_i=D_i$ being of the form 
        \begin{align}
\label{eq:representation_A_blocks_common}
    A_i
        =
    \left[ 
    \begin{array}{c|c} 
      \text{Id}_k & \mathbf{0} \\ 
      \hline 
      \mathbf{0} & \zeta_i \text{Id}_{d-k} 
    \end{array} 
    \right]
\end{align}
    for some common $k \in \{0, \dots, d \}$. 
    A direct consequence of this assumption is that the barycenter of these matrices is also of the same form. Indeed, 
    \begin{align}
        \sum_{i=1}^N 
            \lambda_i 
            \big| A_i - B \big|^p 
                = 
         \sum_{i=1}^N 
                \lambda_i 
                \Big(
                    \sum_{m,l=1}^d 
                        |(A_i)_{ml} - B_{ml}|^2
                \Big)^{\frac{p}2}
                =
         \sum_{i=1}^N 
                \lambda_i 
                \Big(
                    \sum_{l=1}^d
                        |(A_i)_{ll} - B_{ll}|^2
                            +
                 \sum_{m\neq l}
                        |B_{ml}|^2
                \Big)^{\frac{p}2}
                    \, .
    \end{align}
    This shows that the optimizers $B$ in the definition of $\xp(A_1, \dots, A_N)$ are also diagonal (i.e. $B_{jk}=0$ for every $j\neq k$). For $b \in \RR^d$, we denote by $B=b \cdot \text{Id}$ the corresponding diagonal matrix so that $B_{kk} = b_k$. Vice versa, for a given matrix $B$, we denote by $\diag(B)$ the corresponding diagonal vector $( B_{11}, \dots, B_{dd}) \in \RR^d$. With this notation at hand, from the above considerations we conclude that 
    \begin{align}
    \label{eq:bary_common_diag}
        \xp(A_1, \dots, A_N) 
            = 
        \argmin_{B = b \cdot \text{Id}, \, b \in \RR^d}
            \sum_{i=1}^N 
            \lambda_i
            \Big(
            \sum_{l =1}^d 
            | (A_i)_{ll} - b_l |^2
            \Big)^{\frac{p}2}
                = 
            \overline \xi \cdot \text{Id}
                \, , 
    \end{align}
    where $\overline\xi  \in \RR^d$ is given by 
    \begin{align}
        \overline\xi
            := 
        \xp\big( \diag(A_1), \dots, \diag(A_N) \big)
            = 
        \big(
            \underbrace{1, \dots, 1}_ {k} 
                , 
            \underbrace{\overline\zeta , \dots, \overline\zeta}_{d-k} 
        \big)
            \in \RR^d 
            \, , \qquad 
        \overline\zeta = \xp(\zeta_1, \dots, \zeta_N)
            \geq 0
                \, .
    \end{align}
    Here the latter equality follows from similar considerations as above, as 
    \begin{align}
        \xp\big( \diag(A_1), \dots, \diag(A_N) \big)
            =
        \argmin_{b \in \RR^d}
        \sum_{i=1}^N 
        \lambda_i 
        \Big(
            \sum_{l=1}^k
            |b_l - 1|^2 
                + 
            \sum_{l=k+1}^d
            |b_l - \zeta_i|^2
        \Big)^{\frac{p}2}
            \, , 
    \end{align}
    and thus minimisers satisfy $b_l = 1$ for every $l \in \{1, \dots, k \}$.

    Another consequence of \eqref{eq:bary_common_diag} is that the invertibility of $A_1$, together with the fact that $A_i$ are all nonnegative definite, ensures that $\overline{A}:= \xp(A_1, \dots, A_N)$ is invertible as well, as the $p$-barycenters of $N$ nonnegative numbers of which at least one is strictly positive is strictly positive as well, hence the triviality of the kernel of $\overline{A}$. 
    
    Let us denote by $R_i$ the affine transformations given by
    \begin{align}
    \label{eq:optimal_maps_affine_bary}
        R_i z
            := 
        A_i \overline{A}^{-1} z + w_i 
            \, , \qquad 
        w_i := v_i - A_i \overline{A}^{-1} \overline v 
            \in \RR^d 
                \, , \qquad 
            \forall z \in \RR^d 
                \, ,
    \end{align}
    which is nothing but the composition of the map $A_i\cdot + v_i$ with the inverse of $\overline{A} \cdot + \overline{v}$.
    Observe that, by construction, if $A_j\neq {\rm Id }$ for some $j$ -- and thus $v_i=\overline v$ for every $i$ -- then $w_i \in { \rm range }({\rm Id} - A_i \overline{A})$. 
    Therefore, thanks to \eqref{eq:bary_common_diag}, the maps $R_i$ satisfy the assumptions of   Theorem~\ref{thm:optimality_affine_Wp}, and thus are optimal for $W_p(\overline\nu, \mu_i)$. Furthermore, as shown in the very proof of Theorem~\ref{thm:optimality_affine_Wp}, we know that there exist Kantorovich potentials $\psi_i$ 
    (thus optimal for $\frac1pW_p^p(\overline\nu, \mu_i)$) so that 
    \begin{align}
    \label{eq:global}
      R_i(z) = z - 
    \frac
        {\nabla \psi_i(z)}
        {|\nabla \psi_i(z)|^{\alpha_p}} 
            \, , \qquad
        \forall z \in \RR^d 
            \, .
    \end{align}

   We prove the optimality of $\overline\nu$ using Proposition~\ref{lem:characterization}. Indeed, the fact that \eqref{eq:global} holds \textit{globally} on the whole $\RR^d$ (whereas typically such conditions only hold on the support of the barycenter) allows for the proof of
   \begin{align}\label{eq:xp_identiy_global}
        \xp\circ (R_1 , \dots, R_N) (z) = z 
            \, , \qquad  \forall z \in \RR^d 
                \, .
    \end{align}
   Using \eqref{eq:xp_identiy_global} and \eqref{eq:equivalence_identity_sumzeropotentials}, one can show that 
      \begin{align}
    \label{eq:claim_affine_bary}
        \sum_{i=1}^N 
            \psi_i(z) 
        = 
            C \in \RR 
                \, , \qquad 
            \forall z \in \RR^d 
                \, , 
    \end{align}
    and conclude that $\overline\nu$ is the $p$-Wasserstein barycenter between $\mu_1, \dots, \mu_N$ by Proposition~\ref{lem:characterization}.
   
\noindent 
\textit{Proof of \eqref{eq:xp_identiy_global}}. 
Recalling the definition of $R_i$ as composition of two affine maps, the sought equality is equivalent to 
\begin{align}
\label{eq:claim_affine_joint_pre}
    \xp(A_1 z + v_1, \dots, A_N z + v_N) 
        =
    \overline{A}z + \overline{v} 
        \, , \qquad 
    \forall z \in \RR^d 
        \, .
\end{align}
In other words, we have to prove that
\begin{align}
\label{eq:claim_affine_joint}
    \sum_{i=1}^N 
    \lambda_i 
    \big|A_i z + v_i - (\overline A z + \overline v)\big|^{p-2} \big(A_i z + v_i - (\overline A z + \overline v) \big) = 0
        \, , \qquad 
    \forall z \in \RR^d 
        \, .
\end{align}

This is precisely where the structure assumptions on $A_i$ and $v_i$ come into play. In the first case where we have no dilation and $A_i = \textrm{Id}$ for every $i=1, \dots, N$ (and therefore  $\overline{A} = \textrm{Id}$), then \eqref{eq:claim_affine_joint} reduces to the very definition of $\overline{v} = \xp(v_1, \dots, v_N)$, whence the claimed equality. 

Let us show the validity of \eqref{eq:claim_affine_joint} when the shift is constant, i.e. when $v_j = v \in \RR^d$ for every $j = 1, \dots, N$. As the barycenter of vectors which are shifted \textit{by the same vector} $v$ is the shifted (by $v$) barycenter of the original vectors, in this case \eqref{eq:claim_affine_joint} becomes 
\begin{align}
    \sum_{i=1}^N 
    \lambda_i 
    \big|A_i z  - \overline A z \big|^{p-2} \big(A_i z - \overline A z \big) = 0
        \, , \qquad 
    \forall z \in \RR^d 
        \, .
\end{align}
As the matrices $A_1, \dots , A_N$ commute, we have already observed that we can work in coordinates which make the matrices diagonal, and where by \eqref{eq:bary_common_diag} we have that 
\begin{align}
\label{eq:bary_common_diag2}
     \xp(A_1, \dots, A_N) 
        = 
    \overline \xi \cdot \text{Id}
        \, , \qquad 
    \overline\xi =   \big(
            \underbrace{1, \dots, 1}_ {k} 
                , 
            \underbrace{\overline\zeta , \dots, \overline\zeta}_{d-k} 
        \big)
            \in \RR^d 
            \, , \quad  
        \overline\zeta = \xp(\zeta_1, \dots, \zeta_N)
            \, ,
\end{align}
where we recall that $\zeta_j$ is the only eigenvalue of $A_j$ possibly different from one.  In particular, when writing $z=(z_1, \dots, z_k, \tilde z)$ with $\tilde z \in \RR^{d-k}$, the validity of \eqref{eq:claim_affine_joint} is equivalent to
\begin{align}
    0
        =
    \sum_{i=1}^N 
    \lambda_i 
    \big|\zeta_i \tilde z  - \overline\zeta  \tilde z \big|^{p-2} \big(\zeta_i\tilde z - \overline\zeta \tilde z \big) 
        =
    \bigg(
    \sum_{i=1}^N 
        \lambda_i
        \big|\zeta_i   - \overline\zeta  \big|^{p-2} \big(\zeta_i  - \overline\zeta \big) 
    \bigg)
    | \tilde z|^{p-2}\tilde z     
        \, , \qquad 
    \forall \tilde z \in \RR^{d-k}
        \, ,
\end{align}
or, in other words, 
\begin{align}
    \sum_{i=1}^N 
        \lambda_i
        \big|\zeta_i   - \overline\zeta  \big|^{p-2} \big(\zeta_i  - \overline\zeta \big) 
    = 0
        \, .
\end{align}
This is precisely the condition arising from $\overline\zeta = \xp(\zeta_1, \dots, \zeta_N)$.  
    \end{proof}

\begin{remark}[What goes wrong with general affine transformations]
    The structural assumptions of Proposition \ref{prop:bary_affine} are two-fold. 
    
    On the one hand, the fact that the matrices must have at most one eigenvalue different from one is intrinsically related to the fact that such a property characterises the optimality of the associated transport maps given by \eqref{eq:optimal_maps_affine_bary}, as described in Theorem~\ref{thm:optimality_affine_Wp}. The commutation property between the matrices $A_i$ is what ensures that such a property is also preserved for $\overline{A}$, and therefore for the composition in \eqref{eq:optimal_maps_affine_bary}.

    On the other hand, the need for having either pure translations or affine maps with the same shift is a consequence of the lack of linearity of the Euclidean p-barycenter. Once working in coordinates which make each $A_i$ in diagonal form, it is clear that proving \eqref{eq:claim_affine_joint_pre} amounts to showing that ($\xp$ on the right-hand side below is meant as a one-dimensional barycenter)
\begin{align}
     \Big( 
        \xp(\xi_1 \tilde z + v_1, \dots, \xi_N \tilde z + v_N) 
    \Big)_l
        = 
    \xp(\xi_i , \dots, \xi_N ) 
    \tilde z^l
        +
    \xp(v_1^l , \dots,  v_N^l)
        \, , \quad 
\end{align}
for every $\tilde z \in \RR^{d-k}$ and $ l = 1 , \dots, d-k$. 
This is not true when $p\neq2$, while it is simply a consequence of the linearity of the arithmetic mean for the case $p=2$. It would be interesting to understand what the $W_p$-barycenters would be in the case of genuinely affine transformations, which unfortunately goes beyond the reach of this work. 
\end{remark}

To conclude this section and the paper, let us highlight the explicit structure of the maps involved in the barycentric problem in the setting of the previous Proposition~\ref{prop:bary_affine}. In particular, in this case $g_p$ is clearly regular, as $J_{\barp}$ is constant.

 Let $\mu\in \calP_p(\RR^d)$ be a reference probability measure. Consider probability measures of the form $\mu_i = \big( A_i \cdot + v_i \big)_{\#} \mu$, in the two different regimes.  We first discuss the case of pure translations (i.e. $A_i = {\rm Id}$ for every $i$) and then the case of commuting, block matrices.

\bigskip 
\noindent
\textit{Translations}. \ 
Consider $A_i = {\rm Id}$ for every $i = 1 , \dots, N$. From Proposition~\ref{prop:bary_affine}, we know that  the $p$-Wasserstein barycenter between $\mu_1, \dots, \mu_N$ is given by 
\begin{align}
    \overline\nu_p 
        = 
    \big( 
        \cdot + \overline v
    \big)_{\#}\mu
        =
    \big( 
        \cdot -v_1 + \overline v
    \big)_{\#}\mu_1
        \, , \qquad
    \text{where} \quad 
    \overline{v} := \xp(v_1, \dots, v_N)
        \, .
\end{align}

In particular, whenever $\mu = f_1 \LL^d$ with $f_1 \in L^q$, we have $g_p \in L^q$ and $\| g_p \|_{L^q} =\| f_1 \|_{L^q}$. Indeed, in this particular case, the maps $T_i$ are explicitly given by 
\begin{align}
    T_i(x) = x - v_1 + v_i 
        \, , \qquad 
    x \in \RR^d   \, , \, i = 1, \dots , N \, ,
\end{align}
and the map $\barp$ is simply given by 
\begin{align}
    \barp(x_1) = \xp(x_1, T_2(x_1), \dots, T_N(x_1)) 
        &=
    \xp(x_1 - v_1 + v_1, x_1 - v_1 + v_2 , \dots, x_1 - v_1 + v_N)
\\
        &=
    x_1 - v_1 + \xp(v_1, v_2, \dots, v_N)
\\
        &=
    x_1 - v_1 + \overline{v}
        \, ,
\end{align}
hence it is the translation by the vector $\overline{v} - v_1$.

\bigskip 
\noindent
\textit{Linear transformations}. \ 
Consider $v_i = v$, and without loss of generality we assume $v=0$ and assume that $A_1,\dots, A_N$ satisfy all the assumptions of Proposition~\ref{prop:bary_affine}.
Then by Proposition~\ref{prop:bary_affine}, the $p$-Wasserstein barycenter between $\mu_1, \dots, \mu_N$ is given by  
\begin{align}
\label{eq:statement_bary_affine}
    \overline\nu_p 
        = 
    \big( 
        \overline{A}\cdot
    \big)_{\#}\mu
        \, , \qquad \text{where} \quad
    \overline{A}:= \xp(A_1, \dots, A_N) 
        \, .
\end{align}
In this case, the maps $T_i$ are given by 
\begin{align}
    T_i(x) = A_i A_1^{-1} x
        \, , \qquad 
    x \in \RR^d   \, , \, i = 1, \dots , N \, . 
\end{align}
Therefore, the map $\barp$ can be computed as
\begin{align}
    \barp(x_1) = \xp(x_1, T_2(x_1), \dots, T_N(x_1)) 
        &=
    \xp(A_1 A_1^{-1} x_1,  A_2 A_1^{-1} x_1  , \dots, A_N A_1^{-1} x_1)
\\
        &=
    \xp(A_1 A_1^{-1},  A_2 A_1^{-1} , \dots, A_N A_1^{-1} ) x_1
\\
        &=
    \xp(A_1,  A_2  , \dots, A_N )  A_1^{-1} x_1
\\
        &=
    \overline{A} A_1^{-1} x_1
        \, ,
\end{align}
hence it is the linear transformation with associated matrix $\overline{A} A_1^{-1}$. Note that in the latter computation, we used the 1 homogeneity of the barycenter map (as a map over matrices) for our special choice of linear transformations, which holds true as a consequence of the $p$-homogeneity of the $p$-Euclidean barycenter and the property shown in \eqref{eq:bary_common_diag}. This property generally fails to hold for general matrices $A_1, \dots , A_N$, as the $p$-barycenter between matrices is indeed not $1$-homogeneous in the space of all self-adjoint matrices. Such property holds more generally if we assume that the matrices commute, the proof being very similar to the proof in \eqref{eq:bary_common_diag} when diagonal in a common basis.

\appendix
\section{Injectivity estimate and proof of Theorem \ref{th:absolutecontinuity}}
\label{appendix}
The aim of this appendix is to prove Theorem~\ref{th:absolutecontinuity}. Lemma \ref{lem:estimate} below is crucial in Section~\ref{sec:general_estimate} for the proof of  Corollary~\ref{cor:injectivity estimate barp}, and consequently of Proposition~\ref{prop:general_Lq_estimate_intro}.

The next lemma, whose proof is a direct consequence of the definition of $\xp$ and its Lipschitz regularity (\textcolor{blue}{Lemma \ref{lem:lipschitz} -- recall as well the definition of $L_p$ therein}), describes the properties of the $p$-barycenter on the sets $D_S$, defined in \eqref{eq:defDS_intro}.
\begin{lemma}\label{lem:reducedproblem}
Let  $S\subsetneq\{1,\dots, N\}$. Then for every $\xvec=(\xdots)\in D_S$, $x_i$ is a solution of the variational problem 
\begin{equation}\label{eq:reducedbary}    \xpS(\xvec_{S^\mathsf{c}}):=\underset{z\in\RR^d}\argmin\sum_{j\not\in S}\lambda_j|x_j-z|^p \quad \text{for every}\quad  i\in S,
\end{equation}
where, if $|S|=K<N$, $\xvecS\in \RR^{(N-K)d}$ is the vector with components $x_j$, $j\not\in S$. In particular, $ \xpS$ is Lipschitz on $\RR^{(N-K)d}$ with Lipschitz constant $L_p$. Notice that, since $x_i=\xp(\xvec)$ on $D_S$,  $\xp(\xvec)=\xpS(\xvec_{S^\mathsf{c}})$  for every $\xvec \in D_S$.
\end{lemma}
The following Lemma is a consequence of \cite[Lemma 2.9]{BFR24-p} applied to the case $h(\cdot)=|\cdot|^p$. Lemma 2.9 in \cite{BFR24-p} provides inequality \eqref{eq:firstinequality} below and is in turn a refined version of \cite[Lemma 5.2]{BFR24-h}. 
Here we use the notation $B_r(x)$ to denote the open ball centered on $x$ of radius $r$.
Recall the definition of $c_p$-monotone sets \eqref{eq:c-monotonicity}.

 \begin{lemma}[Local regularity on $c_p$-monotone set]\label{lem:estimate}
Let $\Gamma \subset \RR^{Nd}$ be a $c_p$-monotone set. Then, for every $S\subsetneq\{1,\dots,N\}$ and $\xvec\in D_S$, there exists $r(\xvec)>0$ such that \begin{align}
     \label{eq:local_reg_cp_mon}
 |\xp(\yvec)-\xp(\ytildevec)|&\ge \frac 12  \frac {\left(\min_{i\notin S} \Lambda_i(\xvec) \right)}{\max_{i\notin S}|H_{i}(\xvec)|}\Big( \sum_{i\notin S}|y_i-\tilde y_i|^2 \Big)^{\frac12},
     \end{align}
   for every $\yvec, \ytildevec\in  B_{r(\xvec)}(\xvec)\cap D_S \cap  \Gamma$, where $H_i$ and $\Lambda_i$ are defined in \eqref{eq:Hi_Lambda_i}.
 \end{lemma}
Note that
\begin{align}\label{eq:positivity_Hi_Lambdai}
    &H_i(\xvec)\ge 0\, (\text{and thus} \ \Lambda_i(\xvec)\ge 0), \quad \text{for every} \ \xvec\in\RR^{dN},\\
    &H_i(\xvec)=0 \ (\text{and thus} \ \Lambda_i(\xvec) =0) \iff x_i=\xp(\xvec),
\end{align}
hence, for every $\xvec\in D_S$, $H_{i}(\xvec)>0$, i.e. it is a strictly positive definite matrix, and $\Lambda_i(\xvec)>0$, for every $i\notin S$.

 \begin{proof}
Fix $\Gamma \subset \RR^{Nd}$ a $c_p$-monotone set and let $\xvec \in\RR^{Nd}$ be such that $x_i\neq\xp(\xvec)$, for some $i=1,\dots,N$.  By  \cite[Lemma 2.9]{BFR24-p}, we  know that for every $\eps>0$, there exists $r(\xvec)>0$ such that for $\yvec, \ytildevec\in B_{r(\xvec)}(\xvec) \cap \Gamma$, it holds
\begin{equation}\label{eq:firstinequality}
  (y_i-\tilde{y}_{i})^TH_i(\xvec) \left(\xp(\yvec)-\xp(\ytildevec)\right)\ge \Lambda_{i}(\xvec) |y_{i}-\tilde y_{i}|^2-\eps N(1+|H_{i}(\xvec)|)|\yvec-\ytildevec|^2,
\end{equation}
for $\Lambda_i(\xvec)>0$ the minimum eigenvalue of the matrix  $H_i(\xvec)\overline H(\xvec)^{-1}H_i(\xvec)$. 
  
Now let $S\subsetneq\{1,\dots,N\}$ and $\xvec\in D_S$. 
Then, by \eqref{eq:firstinequality}, for every $\yvec, \ytildevec\in B_{r(\xvec)}(\xvec)\cap D_S \cap \Gamma$ we have
\begin{align}
 \sum_{i\notin S}(y_i-\tilde{y}_{i})^TH_i(\xvec) \left(\xp(\yvec)-\xp(\ytildevec)\right)\ge \sum_{i\notin S}\Lambda_i(\xvec) |y_{i}-\tilde y_{i}|^2-\eps N^2(1+|H_i(\xvec)|)|\yvec-\ytildevec|^2
    \, .
\end{align}

Notice now that 
\begin{align}
    |\yvec-\ytildevec|^2=\sum_{i\notin S}|y_{i}-\tilde y_{i}|^2+\sum_{i\in S} |y_{i}-\tilde y_{i}|^2=\sum_{i\notin S}|y_{i}-\tilde y_{i}|^2+\sum_{i\in S}|\xp(\yvec)-\xp(\ytildevec)|^2,
\end{align}
where the last equality follows directly from the definition of $D_S$. By Lemma \ref{lem:reducedproblem}, we know that $\xp$ is  Lipschitz with respect to the components whose index does not belong to S, i.e. 
\begin{align}
    |\xp(\yvec)-\xp(\ytildevec)|^2=| \xpS(\yvec_{S^\mathsf{c}})- \xpS(\ytildevec_{S^\mathsf{c}})|^2\le L_p\sum_{i\notin S}|y_{i}-\tilde y_{i}|^2
    \, .
\end{align}
%
All in all, putting the estimates together, we get that for $\xvec \in D_S$ and every $\eps>0$, there exists $r=r(\xvec,\eps)>0$, such that 
\begin{equation}\label{eq:firstinequality_sum}
  \, \, \sum_{i\notin S}(y_i-\tilde{y}_{i})^TH_i(\xvec) \left(\xp(\yvec)-\xp(\ytildevec)\right)\ge \sum_{i\notin S}(\Lambda_i(\xvec)-\eps  L_p N^2(1+|H_i(\xvec)|))|y_{i}-\tilde y_{i}|^2,
\end{equation}
for every $\yvec, \ytildevec \in B_r(\xvec) \cap D_S \cap \Gamma$. 
In order to conclude the  proof, it is then enough to pick $\eps<\frac 12\frac{\min_{i\notin S}\{\Lambda_i(\xvec)\}}{{\color{blue}L_p}N^2\max_{i\notin S}\{(1+|H_i(\xvec)|)\}}$ and choose $r(\xvec) = r(\xvec,\eps(\xvec))$ accordingly, as the sought estimate \eqref{eq:local_reg_cp_mon} straightforwardly follows by \eqref{eq:firstinequality_sum} and Cauchy--Schwarz inequality.  
 \end{proof}
As a direct consequence we obtain the following:
\begin{corollary}\label{cor:DScase}
Let $\Gamma \subset \RR^{Nd}$ be a $c_p$-monotone set and $S\subset\{1,\dots, N\}$, $S\neq\{1,\dots, N\}$. Then there exists a countable cover $\{U_m\}_{m\in\NN}$ of the set $D_S\cap\Gamma$ with the following property: For every $m\in\NN$, there exists $L_m>0$ such that 
 \begin{equation}\label{eq:inverselipschitzp}
     |\xp(\yvec)-\xp(\ytildevec)|\ge L_m\left(\sum_{j\not\in S}|y_j-\widetilde y_j|^2\right)^{\frac 12}
 \end{equation}
 for every $\yvec=(\ydots),\ytildevec=(\ytildedots)\in D_S\cap\Gamma\cap U_m$.
\end{corollary}
\begin{proof}
Clearly, $\Gamma\cap D_S\subset\bigcup_{\xvec\in D_S}B(\xvec,r(\xvec))$. As
 every subset of $\RR^{dN}$ is second countable, we can extract countably many points $\{\xvec_m\}\subset D_S$ such that $D_S\cap\Gamma\subset\bigcup_{m\in\NN}B_{r_m}(\xvec)$. The claim then follows with $U_m\coloneq B(\xvec_m,r_m)$ and $L_m= \frac 12  \frac {\left(\min_{i\notin S} \Lambda_i(\xvec_m) \right)}{\max_{i\notin S}|H_{i}(\xvec_m)|}$.
\end{proof}
\begin{remark}
Corollary \ref{cor:DScase} provides the same injectivity estimate as \cite[Proposition 3.7]{BFR24-p}, with the difference that in Corollary \ref{cor:DScase} the result holds for any $1<p<\infty$, while in the aforementioned paper it is proved only for $p\ge 2$. This injectivity estimate is the key for the proof of the absolute continuity of the $p$-Wasserstein barycenter $\nup$ (Theorem \ref{th:absolutecontinuity}). In particular, it allows for the proof of \eqref{eq:(1)_lemm_diameterDS} in Lemma \ref{lem:diameterDS} below. Notice that in \cite{BFR24-p} the corresponding result is stated in Lemma 3.3 only for $S=\emptyset$ or in Lemma 3.8 for any $S\neq\{1,\dots,N\}$, but with the assumption that $p\ge 2$. 
The fact that \eqref{eq:(1)_lemm_diameterDS} in Lemma \ref{lem:diameterDS} holds or any $S\neq\{1,\dots,N\}$ allows weakening of the assumption on the marginals $\mu_1,\dots,\mu_N$ in Theorem \ref{th:absolutecontinuity}, by requiring only one to be absolutely continuous, for any $1<p<\infty$.
 \end{remark}
 The proof of \eqref{eq:(1)_lemm_diameterDS} in Lemma  \ref{lem:diameterDS} below is the same as the one of \cite[Lemma~3.9 ]{BFR24-p} and it is a consequence of Corollary \ref{cor:DScase}. The proof of \eqref{eq:(2)_lemm_diameterDS} is a direct application of the definition of $D_S$. For it we refer to \cite[Lemma 3.4]{BFR24-p}.  
\begin{lemma}\label{lem:diameterDS}
 Let $\Gamma \in \RR^{Nd}$ be a $c_p$-monotone set.
 \begin{enumerate}
     \item\label{eq:(1)_lemm_diameterDS} Let $S\subset\{1,\dots,N\}$ be such that $S\neq\{1,\dots,N\}$, and let $\{U_m\}_{m\in\NN}$ be the countable cover of $D_S\cap\Gamma$ defined in Corollary \ref{cor:DScase}. 
If $E\subset\RR^d$ is such that $\diam(E)<\delta$ for some $\delta>0$, then 
 \[\diam(\pi^i(\xp^{-1}(E)\cap D_S\cap U_m\cap\Gamma) )<\frac {\delta} {L_m} \quad \text{for every} \quad i\notin S.\]
 In particular, if $E\subset\RRd$ is such that $\LL^d(E)=0$, then
    \begin{equation*}
     \LL^d(\pi^i(\xp^{-1}(E)\cap D_S\cap U_m \cap\Gamma))=\HH^d(\pi^i(\xp^{-1}(E)\cap D_S\cap U_m\cap\Gamma))=0,
 \end{equation*}
 for every $i\not\in S$ and for every $m\in\NN$.
 \item \label{eq:(2)_lemm_diameterDS} Let $S\subset\{1,\dots, N\}$ be such that $S\neq\emptyset$.
If $E\subset\RR^d$ is such that $\diam(E)<\delta$ for some $\delta>0$, then 
 \[\diam(\pi^i(\xp^{-1}(E)\cap D_S) )<\delta \quad \text{for every} \quad i\in S.\]
 In particular, if $E\subset\RRd$ is such that $\LL^d(E)=0$, then 
    \begin{equation*}\label{eq:hausdorffzero2}
     \LL^d(\pi^i(\xp^{-1}(E)\cap D_S))=\HH^d(\pi^i(\xp^{-1}(E)\cap D_S))=0 \quad \text{for every} \quad i\in S.
 \end{equation*}
 \end{enumerate}
 \end{lemma}

We are now ready to prove Theorem \ref{th:absolutecontinuity}.
\begin{proof}[Proof of Theorem \ref{th:absolutecontinuity}]\label{proof:ThAC}
    Let us consider a set $E$ such that $\LL^d(E)=0$, and consider the family $\mathcal{F}_1=\{S\subset \{1,\dots, N\} \, : \, 1\in S \}$. As $\gamma_p$ is optimal for \eqref{MMpbary}, its support $\spt\gamma_p$ is $c_p$-monotone. Therefore, let $\{U^S_m\}$ be the countable cover of $D_S\cap\spt\gamma_p$ provided by Corollary \ref{cor:DScase}.
Then,
 \begin{align*}
(\xp)_\sharp\gamma_p(E)&=\gamma_p\left(\xp^{-1}(E)\cap\spt\gamma_p\cap \bigcup_{S\subset \{1,\dots,N\}}D_S\right)\\&\le\sum_{S\not\in \mathcal{F}_1}\gamma_p\left(\xp^{-1}(E)\cap\spt\gamma_p\cap D_S\right)+\sum_{S\in \mathcal{F}_1}\gamma_p\left(\xp^{-1}(E)\cap\spt\gamma_p\cap D_S\right)\\
&\le\sum_{S\not\in \mathcal{F}_1}\mu_1\left(\pi^1(\xp^{-1}(E)\cap\spt\gamma_p\cap D_S)\right)+\sum_{S\in \mathcal{F}_1}\mu_1\left(\pi^1(\xp^{-1}(E)\cap\spt\gamma_p\cap D_S)\right)\\
&\le\sum_{S\not\in \mathcal{F}_1}\sum_{m\in\NN}\mu_1\left(\pi^1(\xp^{-1}(E)\cap\spt\gamma_p\cap D_S\cap  U^S_m)\right) +\sum_{S\in \mathcal{F}_1}\mu_1\left(\pi^1(\xp^{-1}(E)\cap\spt\gamma_p\cap D_S)\right) .
 \end{align*}
Notice that the second inequality is due to the marginal constraint $\pi^1_\sharp\gamma_p=\mu_1$. By Lemma~\ref{lem:diameterDS}, $\LL^d(\pi^1(\xp^{-1}(E)\cap\spt\gamma_p\cap D_S\cap U^S_m))=\mathcal{H}^d(\pi^1(\xp^{-1}(E)\cap\spt\gamma_p\cap D_S \cap U^S_m))=0$ for every $S\not\in \mathcal{F}_1$ and every $m\in \NN$. Lemma~\ref{lem:diameterDS} then gives $\LL^d(\pi^1(\xp^{-1}(E)\cap\spt\gamma_p\cap D_S))=\mathcal{H}^d(\pi^1(\xp^{-1}(E)\cap\spt\gamma_p\cap D_S))=0$ for every  $S\in\mathcal{F}_1$. We conclude thanks to the absolute continuity of $\mu_1$.    
\end{proof}

{
\makeatletter
\let\addcontentsline\@gobblethree
\section*{Acknowledgments}
CB gratefully acknowledges funding from the Deutsche Forschungsgemeinschaft (DFG -- German Research Foundation) -- Project-ID 195170736 -- TRR109. 
LP gratefully acknowledges funding from the Deutsche Forschungsgemeinschaft (DFG -- German Research Foundation) under Germany’s Excellence Strategy -GZ 2047/1, Projekt-ID 390685813. Financial support by the Deutsche Forschungsgemeinschaft (DFG) within the CRC 1060, at University of Bonn project number 211504053, is also gratefully acknowledged. 
LP is also thankful for the support under the U-GOV project, identification number PSR$\_$LINEA8A$\_$25SMANT$\_$05.

Both the authors would like to thank the IAM (Institüt für Angewandte Mathematik) of the University of Bonn, the mathematics department "Federigo Enriques" of University of Milan, as well as the T\"UM (Technische Universit\"at M\"unchen) for their kind hospitality during the preparation of this project.

The authors are also very grateful to both anonymous referees, for multiple and important feedback throughout the review process. Several parts of the work have been significantly improved thanks to their work, we want to thank them sincerely.
\makeatother
}

%

\end{document}